\documentclass[reqno,11pt,a4paper]{amsart}
\usepackage{mathtools}
\usepackage{mathrsfs}
\mathtoolsset{showonlyrefs=true}

\usepackage{amsmath}
\usepackage{amscd}
\usepackage{amsfonts}
\usepackage{amssymb}
\usepackage{amsthm}
\usepackage{latexsym}
\usepackage{graphicx}
\usepackage{enumerate}
\usepackage{hyperref}
\usepackage{color}

\numberwithin{equation}{section}
\newtheorem{definition}{Definition}[section]
\newtheorem{theorem}{Theorem}[section]
\newtheorem{lemma}[theorem]{Lemma}
\newtheorem{corollary}[theorem]{Corollary}
\newtheorem{proposition}[theorem]{Proposition}

\theoremstyle{remark}
\newtheorem{remark}{Remark}[section]

\DeclareMathOperator{\Div}{div}
\DeclareMathOperator{\supp}{supp}
\renewcommand{\geq}{\geqslant}
\renewcommand{\ge}{\geqslant}
\renewcommand{\leq}{\leqslant}
\renewcommand{\le}{\leqslant}
\newcommand{\R}{\ensuremath{\mathbb{R}}}
\newcommand{\N}{\ensuremath{\mathbb{N}}}
\newcommand{\dd}{\mathrm{d}}
\newcommand{\Id}{\ensuremath{\mathrm{Id}}}

\allowdisplaybreaks
\begin{document}
\title[The Euler-Alignment system]
{Global well-Posedness and asymptotic behavior in critical spaces for
  the compressible Euler system with velocity alignment}

\author[Xiang Bai]{Xiang Bai}
\address{Laboratory of Mathematics and Complex Systems (MOE), School of Mathematical Sciences, Beijing Normal University, Beijing 100875, P.R. China}
\email{xiangbai@mail.bnu.edu.cn}
\author{Qianyun Miao}
\address{School of Mathematics and Statistics, Beijing Institute of Technology, Beijing 100081, P. R. China}
\email{qianyunm@bit.edu.cn}
\author{Changhui Tan}
\address{Department of Mathematics, University of South Carolina, Columbia SC 29208, USA}
\email{tan@math.sc.edu}
\author{Liutang Xue}
\address{Laboratory of Mathematics and Complex Systems (MOE), School of Mathematical Sciences, Beijing Normal University, Beijing 100875, P.R. China}
\email{xuelt@bnu.edu.cn}


\begin{abstract}
In this paper, we study the Cauchy problem of the compressible Euler system with strongly singular velocity alignment.
We prove the existence and uniqueness of global solutions in critical Besov spaces to the considered system with small initial data. The local-in-time solvability is also addressed.
Moreover, we show the large-time asymptotic behavior and optimal decay estimates of the solutions as $t\to \infty$.


\end{abstract}

\maketitle
\section{Introduction}\label{INTR}
\setcounter{section}{1}\setcounter{equation}{0}

\subsection{The Euler-alignment system}
We consider the following Cauchy problem of the compressible Euler system in $\mathbb{R}^{+}\times\mathbb{R}^N$,
\begin{equation}\label{eq.EA}
\begin{cases}
  \partial_t\rho+\Div(\rho u)=0,\\
  \partial_t(\rho u)+\Div(\rho u\otimes u) + \nabla P(\rho)=\mathcal{D}(u,\rho),\\
  (\rho,u)|_{t=0}=(\rho_0,u_0),
\end{cases}
\end{equation}
where $\rho$ is the density, $u=(u_1,\cdots, u_N)$ is the velocity field, and $P(\rho)$ stands for the pressure, which is given by the power law
\begin{equation}\label{eq:pressure}
  P(\rho)=\kappa\rho^{\gamma},\quad \kappa>0,\,\, \gamma\ge 1.
\end{equation}
The term $\mathcal{D}(u,\rho)$ represents the nonlocal velocity alignment which is given as follows
\begin{align}\label{eq:Du-rho}
  \mathcal{D}(u,\rho)(t,x)=-\rho(t,x)\int_{\R^N}\phi(x-y)\big(u(t,x)-u(t,y)\big)\rho(t,y)\dd y.
\end{align}
Here $\phi$ is called the \emph{communication weight}, measuring the strength of the alignment interactions. It is naturally assumed to be a non-negative and radially decreasing function.

The system \eqref{eq.EA} is known as the \emph{Euler-alignment system}. It is the macroscopic representation of the celebrated Cucker-Smale model \cite{cucker2007emergent}
\begin{equation}\label{eq:CuckerSmale}
\begin{cases}
  \dot{X}_i(t) = V_i(t),\quad 1\leq i\leq M, \\
  \dot{V}_i(t) = -\frac{1}{M} \sum\limits_{j\neq i}\phi(X_i(t)-X_j(t)) (V_i(t)-V_j(t)),
\end{cases}
\end{equation}
an $M$-agent interacting system that describes the collective motions in animal flocks.
Ha and Tadmor \cite{ha2008particle} formally derive \eqref{eq.EA} from \eqref{eq:CuckerSmale} through a kinetic equation
\begin{equation}\label{eq:kinetic}
  \partial_tf+v\cdot\nabla_x f+\Div_vQ[f,f]=0,\,\,
  Q[f,f](t,x,v)=-f(t,x,v)\int_{\R^N}\int_{\R^N}\phi(x-y)(v-w)f(t,y,w)\dd
  y\dd w.
\end{equation}
Hydrodynamic limiting systems \eqref{eq.EA} with different type of pressures \eqref{eq:pressure} can be rigorous derived from \eqref{eq:kinetic}, including the pressure-less dynamics ($\kappa=0$) \cite{figalli2018rigorous}, isothermal pressure ($\kappa>0, \gamma=1$)\cite{karper2015hydrodynamic}, and others \cite{fetecau2016first,poyato2017euler}.

The global well-posedness theory for the pressureless Euler-alignment system \eqref{eq.EA} with $P\equiv0$ has been established in \cite{tadmor2014critical} for the case when the communication weight is bounded and Lipschitz. A critical threshold phenomenon was discovered: global regularity depends on initial data. A sharp threshold condition is obtained in \cite{carrillo2016critical} for the system in 1D with the help of an auxiliary quantity $G=\partial_xu+\phi\ast\rho$ that satisfies the continuity equation. The theory has been extended to the case when the communication weight is weakly singular: unbounded but integrable \cite{tan2020euler}. For higher dimensions, sharp results are only available for radial \cite{tan2021eulerian} and uni-directional \cite{lear2022existence} data, due to the lack of the auxiliary quantity, see also \cite{he2017global}.

Another interesting type of communication weights are strongly singular near the origin,
with a prototype taking the following form
\begin{equation}\label{eq:phialpha}
  \phi(x) = \phi_\alpha (x) =\frac{ c_{\alpha,N} }{|x|^{N+\alpha}}, \qquad c_{\alpha,N}= \frac{2^{\alpha} \Gamma(\frac{\alpha+N}{2})}{\pi^{N/2}\Gamma(-\frac\alpha2)}.
\end{equation}
It is evident that for $\alpha\in(0,2)$, the singular alignment $\mathcal{D}(u,\rho)$ with such a weight $\phi_\alpha$ can be expressed as a commutator form related to the fractional Laplace operator $\Lambda^\alpha = (-\Delta)^{\frac{\alpha}{2}}$ (see Definition \ref{def:fraLap}):
\begin{equation}\label{eq:alignment}
  \mathcal{D}(u,\rho)= - \rho \big(\Lambda^\alpha (\rho u)-u\Lambda^\alpha \rho\big)
  = -\rho \big([\Lambda^\alpha, u] \rho \big).
\end{equation}
This nonlocal dissipation has an intriguing regularization effect to the solutions. Global regularity is obtained for any non-viscous smooth initial data for the system in the one-dimensional torus by Shvydkoy and Tadmor \cite{shvydkoy2017eulerian} for $1\le\alpha<2$, and by Do et al \cite{do2018global} for $0<\alpha<1$ (see also \cite{shvydkoy2018eulerian}). The results are then applied to general singular alignment interactions \cite{kiselev2018global}, as well as taking into account the misalignment effect \cite{miao2021global}. For the multi-dimensional case, global well-posedness are only known for small initial data around an equilibrium state. See the work of Shvydkoy \cite{shvydkoy2019global} for smooth initial data $(\rho_0,u_0)\in H^{N+4}(\mathbb{T}^N)\times H^{N+3+\alpha}(\mathbb{T}^N)$, and Danchin et al \cite{danchin2019regular} for small initial data that lie in critical Besov space $(\rho_0,u_0)\in\dot{B}_{N,1}^{1}(\R^N)\times\dot{B}_{N,1}^{2-\alpha}(\R^N)$, subject to additional regularity assumptions. Global regularity for general large initial data remains a challenging open problem.

The global well-posedness theory for the Euler-alignment system \eqref{eq.EA} with pressure is much less understood compared with the pressure-less system. When the communication weight $\phi$ is bounded and Lipschitz, Y.-P. Choi \cite{choi2019global} proved global regularity of the system with isothermal pressure, for small smooth initial data in the periodic domain $\mathbb{T}^N $. A similar result was obtained in \cite{tong2020global} for the system with isentropic pressures ($\kappa>0, \gamma>1$).

The main focus of this paper is on the Euler-alignment system with pressure and with a strongly singular communication weight \eqref{eq:phialpha}. The goal is to understand the interplay between the pressure \eqref{eq:pressure} and the nonlocal regularization from the alignment \eqref{eq:alignment}.

In 1D periodic domain $\mathbb{T}$, Constantin, Drivas and Shvydkoy \cite{constantin2020entropy} proved the global existence of smooth solutions for the system with an additional local dissipation term of the form \eqref{eq:Dloc}. They make use of the auxiliary quantity to build a hierarchy of entropies. The result does not require a smallness assumption, but is limited to one dimension.

For the system in $\mathbb{T}^N$, Chen, Tan and Tong \cite{chen2021global} established the global well-posedness for smooth initial data with a smallness assumption. The result is partially extended to the whole space $\mathbb{R}^N$. However, an additional linear damping term is required to obtain the desired result.

We would like to comment that most global well-posedness results in the literature on the Euler-alignment system \eqref{eq.EA} with strongly singular alignment \eqref{eq:alignment} are on the periodic domain $\mathbb{T}^N$. One important reason is that solutions can lose regularity when vacuum arises \cite{tan2019singularity, arnaiz2021singularity}. It is easier to obtain \emph{a priori} positive lower bound on the density under periodic setup, as mass cannot diffuse to infinity. Additional analytical treatments are required to guarantee no vacuum formations for the system in the whole space $\R^N$.

\subsection{The barotropic compressible Navier-Stokes system}
To study the Euler-alignment system \eqref{eq.EA} in $\mathbb{R}^N$, we shall mention a very related system. If we replace the dissipation term $\mathcal{D}(u,\rho)$ by
\begin{equation}\label{eq:Dloc}
  \mathcal{D}_{\mathrm{loc}}(u) = \mu_1 \Delta u +\mu_2 \nabla\Div u,\quad \mu_1>0,\mu_1+\mu_2>0,
\end{equation}
the system \eqref{eq.EA} becomes the classical \emph{barotropic compressible Navier-Stokes system}, which has been intensely studied in the recent decades.

Serrin \cite{serrin1959uniqueness} and Nash \cite{nash1962probleme} established the local existence and uniqueness of smooth non-vacuous solutions. One can also see Solonnikov \cite{Solo73} and Valli \cite{valli1982existence} for the local well-poseness of strong solutions with Sobolev regularities. Matsumura and Nishida \cite{matsumura1979initial,matsumura1980initial} proved the global existence and uniqueness of strong solutions provided that initial data $(\rho_0,u_0)$ is a small perturbation of constant non-vacuous state $(\bar{\rho},0)$ in three dimension, and under an additional $L^1$-smallness of the initial perturbation, they showed the following optimal decay estimate
\begin{align}\label{eq:L2-decay}
  \|(\rho -\bar{\rho}, u)(t)\|_{L^2} \lesssim (1+t)^{-\frac{3}{4}},\quad \forall\,t>0.
\end{align}
Later, noting that the barotropic compressible Navier-Stokes system is invariant under the transformation
\[
  \rho(t,x)\mapsto \rho(\lambda^2 t, \lambda x),\quad u(t,x)\mapsto \lambda u(\lambda^2 t,\lambda x),\quad \lambda >0,
\]
with a modification of the pressure $P\mapsto \lambda^2 P$, Danchin \cite{danchin2000global} proved the global existence and uniqueness of strong solution in the framework of critical $L^2$-based Besov space with initial data close to a stable equilibrium. More precisely, under the following smallness condition in critical Besov spaces (see Definition \ref{def:Besov})
\begin{equation}\label{eq:NSsmall}
  \Vert \rho_0-\bar{\rho}\Vert_{ \widetilde{B}^{\frac{N}{2}-1,\frac{N}{2}}}
  +\Vert u\Vert_{\dot{B}^{\frac{N}{2}-1}_{2,1}}<\varepsilon,
\end{equation}
the barotropic compressible Navier-Stokes system has a global unique solution.

Furthermore, for small perturbation of non-vacuous equilibrium $(\bar{\rho},0)$ in $L^p$-type Besov norms, Charve and Danchin \cite{charve2010global} and Chen, Miao and Zhang \cite{chen2010global} independently constructed the global unique strong solution in the framework of critical $L^p$-based Besov spaces. 
One can also see \cite{haspot2011existence} for a simpler proof of the same result by using a good unknown called the effective velocity.
Concerning the large-time behavior of the above obtained global strong solutions, Okita \cite{okita2014optimal} considered the $N\geq 3$ dimension and established the optimal time-decay estimate of global solutions in the critical $L^2$-framework with an additional smallness condition on $\Vert \rho_0-\bar{\rho}\Vert_{\dot{B}^{0}_{1,\infty}}+\Vert u_0\Vert_{\dot{B}^{0}_{1,\infty}}$.
Danchin \cite{danchin2018fourier} gave an another description of the time-decay estimate as above with $N\ge 2$.
Danchin and Xu \cite{danchin2017optimal} showed that under an additional smallness condition of low frequencies
(see \eqref{eq:low-high-norm} for the definition of norm $\|\cdot\|^\ell_{\dot B^s_{p,r}}$)
\begin{equation}\label{eq.xu.small}
  \Vert (\rho_0-\bar{\rho},u_0)\Vert_{\dot B^{-s_0}_{2,\infty}}^\ell <\varepsilon,
  \quad s_0 = N\left(\tfrac{2}{p} - \tfrac{1}{2}\right),
\end{equation}
the $L^p$ norm of the global critical solutions constructed in \cite{charve2010global,chen2010global,haspot2011existence} decays like $t^{-N(\frac{1}p-\frac{1}{4})}$ for $t\to +\infty$ (exactly as \eqref{eq:L2-decay} with $p=2, N=3$). One can see Xu \cite{xu2019low} for a different low-frequency smallness assumption to get the same time-decay estimate. Recently, Xin and Xu \cite{xin2021optimal} replaced the smallness condition \eqref{eq.xu.small} with a mild assumption like $\Vert (\rho_0-\bar{\rho},u_0)\Vert_{\dot B^{-s_0}_{2,\infty}}^\ell <\infty$ and obtained the optimal time-decay estimate in the general critical $L^p$-framework.

\subsection{Main result: global well-posedness}
In this paper, we consider the Euler-alignment system \eqref{eq.EA} in $\R^N$, with power-law type pressure \eqref{eq:pressure} and strongly singular alignment interactions \eqref{eq:alignment} with $1<\alpha<2$. We mainly study the global well-posedness of the system with initial data $(\rho_0,u_0)$ around the non-vacuous equilibrium $(\rho\equiv1,u\equiv0)$, with minimal regularity assumptions on the initial data.

Analogous to the study of the barotropic compressible Navier-Stokes system, we observe that the system \eqref{eq.EA} is invariant under the transformation
\begin{equation}\label{eq:scale-trans}
  \rho(t,x)\mapsto \rho(\lambda^\alpha t, \lambda x),\quad u(t,x)\mapsto \lambda^{\alpha-1} u(\lambda^\alpha t,\lambda x),\quad \lambda >0,	
\end{equation}
with a modification of the pressure $P\mapsto \lambda^{2\alpha-2} P$. Therefore, we shall aim to solve the Euler-alignment system \eqref{eq.EA} in the critical function space which is invariant with respect to the transform \eqref{eq:scale-trans}. Obviously, the homogeneous Besov space $\dot{B}^{\frac{N}{2}}_{2,1}\times\dot{B}^{\frac{N}{2}+1-\alpha}_{2,1}$ of initial data $(\rho_0-1,u_0)$ is a suitable space that is scaling critical. However,  spectral analysis of the linearized equation of system \eqref{eq.EA} (see \eqref{eq.lineareqbn0} below) indicates that the regularity $\rho\in\dot{B}^{\frac{N}{2}}_{2,1}$ is not enough to control the pressure term, as it is not invariant under the scaling \eqref{eq:scale-trans}. Instead, we work on a hybrid Besov space $\widetilde{B}^{\frac{N}{2}+1-\alpha,\frac{N}{2}}=\dot B^{\frac{N}{2}+1-\alpha}_{2,1} \cap \dot B^{\frac{N}{2}}_{2,1}$ (see Definition \ref{def:Besov}). This approach is pioneered by Danchin \cite{danchin2000global} on the barotropic compressible Navier-Stokes system.


\begin{theorem}[Global well-posedness]\label{cor.global.solution}
Let $N\geq2$. Consider the Euler-alignment system \eqref{eq.EA} with pressure \eqref{eq:pressure}, alignment interaction \eqref{eq:alignment} with $1<\alpha<2$, and initial data $\rho_0-1\in \widetilde{B}^{\frac{N}{2}+1-\alpha,\frac{N}{2}}(\R^N)$ and $u_0 \in \dot{B}^{\frac{N}{2}+1-\alpha}_{2,1}(\R^N)$.
There exists a small constant $\varepsilon>0$, such that if
\begin{equation}\label{eq:cond0}
  \Vert \rho_0-1\Vert_{ \widetilde{B}^{\frac{N}{2}+1-\alpha,\frac{N}{2}}}
  +\Vert u_0\Vert_{\dot{B}^{\frac{N}{2}+1-\alpha}_{2,1}}<\varepsilon,
\end{equation}
then the Euler-alignment system \eqref{eq.EA} has a global unique solution $(\rho,u)$ such that
\begin{equation}\label{eq:rho>0}
  \rho>0,\quad \textrm{in}\;\;\R^+\times \R^N,
\end{equation}
and
\begin{equation}\label{eq:rho-u-bdd0}
\begin{split}
  \rho-1 \in C_b([0,+\infty);\widetilde{B}^{\frac{N}{2}+1-\alpha,\frac{N}{2}}), \quad
  u \in C_{b}([0,+\infty);\dot{B}^{\frac{N}{2}+1-\alpha}_{2,1}) \cap L^1(\mathbb{R}^+;{\dot{B}^{\frac{N}{2}+1}_{2,1}}).
\end{split}
\end{equation}


Moreover, if additionally $(\rho_0-1,u_0)\in \widetilde{B}^{s,s+\alpha-1}(\mathbb{R}^N)\times\dot{B}^{s}_{2,1}(\mathbb{R}^N)$ with $s>\frac{N}{2}+1-\alpha$,
then the above constructed solution also belongs to the corresponding space,  i.e.,
\begin{equation}\label{eq:rho-u-hreg}
\begin{aligned}
  \rho-1\in {\widetilde{L}^\infty (\mathbb{R}^+;\widetilde{B}^{s,s+\alpha-1})},
  \quad u\in{\widetilde{L}^\infty (\mathbb{R}^+;\dot{B}^s_{2,1})}\cap{L^1(\mathbb{R}^+;\dot{B}^{s+\alpha}_{2,1}}).
\end{aligned}
\end{equation}
See Definitions \ref{def:Besov} and \ref{def:CL} for the spaces involved.
\end{theorem}

\begin{remark}
The critical space we work on contains very rough initial data. In particular, $u_0$ is not necessarily Lipschitz. Therefore, even local well-posedness can not be obtained directly from the classical Cauchy-Lipschitz theory. We include a local well-posedness result in Theorem \ref{th.local.solution.nl}. The smallness condition \eqref{eq:cond0} can be relaxed for local-in-time solutions.
\end{remark}

\begin{remark}
Our result works on dimensions $N\geq2$. When $N=1$, we are able to establish \emph{a priori} estimates of the Euler-alignment system as in Section \ref{priori}. However, some technical estimates fail when constructing the solution e.g. \eqref{eq:par-t-sig}, \eqref{eq:par-t-u}, as well as the uniqueness argument. This is due to the roughness of initial data that we consider. If we further assume
\[ 
\partial_x \rho_0\in \widetilde{B}^{\frac{3}{2}-\alpha,\frac{1}{2}}(\mathbb{R})
\quad\text{and}\quad
\partial_x u_0\in \dot B^{\frac{3}{2}-\alpha}_{2,1}(\mathbb{R}),
\]
Theorem \ref{cor.global.solution} can be easily extended to $N=1$.
\end{remark}

\begin{remark}
For less singular communication weight with $0<\alpha\le1$, we observe a different spectral structure in the linearized equation \eqref{eq.lineareqbn0}. Hence, the global well-posedness result is expected to be different. This case will be discussed in a separate work.
\end{remark}


\subsection{Asymptotic behavior}
Next, we turn our attention to the asymptotic behavior of the Euler-alignment system \eqref{eq.EA}. The system inherits a remarkable \emph{flocking} phenomenon from the Cucker-Smale model \eqref{eq:CuckerSmale}. For the pressure-less system with bounded Lipschitz communications, Tadmor and Tan \cite{tadmor2014critical} showed that for any smooth subcritical initial data, the solution converges (in appropriate sense) to a traveling wave profile
\begin{equation}\label{eq:flock}
\rho(t,x)\to\rho_\infty(x-\bar{u}t),\qquad u(t,x)\to\bar{u}.
\end{equation}
There are two ingredients of flocking. First, the support of density $\rho$ stays bounded in all time, i.e., if $\rho_0$ is compactly supported, then $\rho_\infty$ is compactly supported as well. Another ingredient is the \emph{velocity alignment}. Here, $\bar{u}$ represents the average velocity. It is determined by initial data, thanks to the conservation of momentum. Without loss of generality, we assume $\bar{u}=0$ throughout the paper. See \cite{do2018global,shvydkoy2017eulerian2} for discussions on flocking for pressure-less Euler-alignment system with strongly singular alignment interactions \eqref{eq:alignment} on $\mathbb{T}$. Additional geometric structures of the limiting profile $\rho_\infty$ is investigated in \cite{leslie2019structure,lear2022geometric}. 

When pressure is presented, the asymptotic density profile is known to be a constant $\rho_\infty(x)\equiv\bar{\rho}$. For simplicity, we set $\bar{\rho}=1$. The asymptotic flocking behavior is proved in \cite{tong2020decay} for bounded alignment interactions, and \cite{chen2021global} for strongly singular alignment interactions. Both results considered the periodic domain $\mathbb{T}^N$.

To our best knowledge, most existing results on asymptotic behaviors for the Euler-alignment system with singular alignment interactions are on the periodic domains. The decay rates in \eqref{eq:flock} are exponentially in time. For the system in $\R^N$, we do not expect the decay rate to be exponential. Rather, analogous to the heat equation, the diffusion leads to a polynomial rate of decay in time.

Our next result is concerned with the asymptotic behavior of the Euler-alignment system \eqref{eq.EA} in $\R^N$.
\begin{theorem}[Asymptotic behavior]\label{thm:decay}
Let $N\ge 2$ and $1<\alpha<2$.
Assume $(\rho,u)$ is a global solution of the Euler-alignment system \eqref{eq.EA} that satisfies \eqref{eq:rho>0} and
\begin{equation}\label{eq:rho-u-bdd}
  \rho-1\in \widetilde{L}^{\infty}(\mathbb{R}^{+};\widetilde{B}^{\frac{N}{2}+1-\alpha,\frac{N}{2}}), \quad 
  u\in \widetilde{L}^{\infty}(\mathbb{R}^{+};\dot{B}^{\frac{N}{2}+1-\alpha}_{2,1})\cap L^{1}(\mathbb{R}^{+};{\dot{B}^{\frac{N}{2}+1}_{2,1}}).
\end{equation}
Then for every $0< s< 1- \frac{1}{\alpha}$ we have
\begin{equation}\label{eq:rho-u-decay1}
  \Vert (\rho-1,u)^\ell(t)\Vert_{\dot{B}^{\frac{N}{2}+1-\alpha+s\alpha}_{2,1}} + \Vert \rho^h(t)-1\Vert_{ \dot{B}^{\frac{N}{2}}_{2,1}}+\Vert u^h(t)\Vert_{\dot{B}^{\frac{N}{2}+1-\alpha}_{2,1}}
  \le C(1+t)^{-s},
\end{equation}
where the constant $C>0$ depends on the norms of $(\rho-1,u)$ in \eqref{eq:rho-u-bdd}. Besides, we have
\begin{equation}\label{eq:rho-u-decay2}
  \lim_{t\to \infty}\big(\Vert \rho(t)-1\Vert_{\widetilde{B}^{\frac{N}{2}+1-\alpha,\frac{N}{2}}}+\Vert u(t)\Vert_{\dot{B}^{\frac{N}{2}+1-\alpha}_{2,1}}\big) = 0.
\end{equation}
If we assume, in addition, $(\rho_0-1,u_0)^\ell\in \dot{B}^{-s_0}_{2,\infty}(\R^N)$ with $s_0\in (\alpha-\frac{N}{2}-1,\frac{N}{2})$,
then for all $s_1\in [-s_0, \frac{N}{2}+1-\alpha]$,
\begin{equation}\label{eq:rho-u-decay3}
  \Vert \rho(t)-1\Vert_{\widetilde{B}^{s_1,\frac{N}{2}}}+\Vert u(t)\Vert_{\widetilde{B}^{s_1,\frac{N}{2}+1-\alpha}} \le C(1+t)^{-\frac{s_1+s_0}{\alpha}}.
\end{equation}
\end{theorem}

\begin{remark}
The decay rate obtained in \eqref{eq:rho-u-decay3} is optimal. It agrees with the rate of decay for the solutions to the fractional heat equation, that is a linearized equation of \eqref{eq.EA}.
\end{remark}

\begin{remark}
For small initial data $(\rho_0,u_0)$ satisfying \eqref{eq:cond0}, condition \eqref{eq:rho-u-bdd} is guaranteed by Theorem \ref{cor.global.solution}. Hence the decay estimates follow directly from Theorem \ref{thm:decay}.
On the other hand, Theorem \ref{thm:decay} only assumes that the solution $(\rho-1,u)$ is bounded as in \eqref{eq:rho-u-bdd}. In the existing literature on the barotropic compressible Navier-Stokes system  \cite{danchin2017optimal,okita2014optimal,xu2019low} and so on, smallness assumptions on the solution $(\rho-1,u)$ are required to obtain the decay estimates. We adopt a different approach that greatly relaxes the assumptions compared with the aforementioned work.

Let us point out that our decay estimate \eqref{eq:rho-u-decay3} requires $s_0<\frac{N}{2}$. The endpoint $s_0=\frac{N}{2}$ is not captured by our approach on a basic lack of paraproduct in endpoint $\dot{B}^{-N/2}_{2,1}$. Under additional smallness conditions on initial data, the decay estimate can be proved in an alternative way, analogous to \cite{danchin2017optimal,okita2014optimal,xu2019low}.
\end{remark}

\subsection{Outline of the paper}
The paper is organized  as follows. In Section \ref{Preliminary}, we introduce the definition of hybrid Besov space and fractional Laplace operator, and present some useful auxiliary lemmas.
In Section \ref{priori}, we reformulate our system to \eqref{eq.EAsigma.nl}, establish the \textit{a priori} estimates for a linearized system \eqref{eq.lineareqbn0} and the nonlinear system \eqref{eq.EAsigma.nl}.
Section \ref{global_solution} are devoted to the proof of global existence and uniqueness for our system \eqref{eq.EA}.
In Section \ref{decay}, we present the proof of large time behavior stated in Theorem \ref{thm:decay}.

\section{Preliminary}\label{Preliminary}
This section includes some basic analytical tools needed in this paper. We first introduce the concept of Besov spaces and  some properties. Then we recall the definition of fractional Laplacian and the Kato-Ponce type commutator estimates, particularly in Besov spaces.

\subsection{Besov spaces and some related estimates}
We first introduce the Littlewood-Paley decomposition.
One can choose a nonnegative radial function $\varphi\in C_c^\infty(\mathbb{R}^N)$ be
supported in
the annulus $\{\xi\in \mathbb{R}^N: \frac{3}{4} \leq |\xi|\leq  \frac{8}{3} \}$ such that (e.g. see \cite{bahouri2011fourier})
\begin{align}\label{eq:chi-phi}
  \sum_{j\in\mathbb{Z}}\varphi_j(\xi)=1,\ \ \ \ \forall \xi\in\mathbb{R}^N \setminus \{0\},
\end{align}
where $\varphi_j(\xi)=\varphi(2^{-j}\xi)$.
We define the localization operator:
$$
\dot{\Delta}_ju=\varphi_j(D)u,\ \ \ \dot{S}_ju=\sum_{k\le j-1}\dot{\Delta}_ku.
$$

Now we present the definition of (homogeneous) Besov space and its hybrid type. 
\begin{definition}[Besov sapce and hybrid Besov sapce]\label{def:Besov}
Let $s,s_1,s_2\in \mathbb{R}$, $(p,r)\in [1,\infty]^2$. Denote by $\mathcal{P}(\R^N)$ the set of all polynomials and by
$\mathcal{S}_h^\prime(\mathbb{R}^N):=\mathcal{S}'(\mathbb{R}^N)/\mathcal{P}(\R^N)$ the quotient space.
We define the homogeneous Besov space $\dot{B}^s_{p,r}=\dot{B}^s_{p,r}(\R^d)$ as
\begin{equation*}
  \dot{B}^s_{p,r}(\mathbb{R}^N):=\Big\{u\in \mathcal{S}_h^\prime(\mathbb{R}^N);
  \|u\|_{\dot{B}^s_{p,r}(\mathbb{R}^N)}
  =\big{\Vert} \big\{2^{js}\Vert{\dot{\Delta}_ju}  \Vert_{L^p(\mathbb{R}^N)}\big\}_{j\in\mathbb{Z}}\big{\Vert}_{\ell^r(\mathbb{Z})}<\infty\Big\}.
\end{equation*}
Let $j_0\in \mathbb{Z}$, and we define the hybrid Besov space $\widetilde{B}^{s_1,s_2} = \widetilde{B}^{s_1,s_2} (\R^N)$ as the set of all $u\in  \mathcal{S}_h^\prime(\mathbb{R}^N)$ such that
\begin{equation*}
  \Vert u \Vert_{\widetilde{B}^{s_1,s_2}} := \sum_{j=-\infty}^{j_0}2^{js_1}\Vert \dot{\Delta}_j  u\Vert_{L^2}+\sum_{j=j_0+1}^{+\infty}2^{js_2}\Vert \dot{\Delta}_j u\Vert_{L^2}<\infty.
\end{equation*}
\end{definition}

By restricting the norm of $\dot B^s_{p,r}$ to the low or high frequency parts of tempered distributions,
we also get that, for some $j_0\in \mathbb{Z}$,
\begin{align}\label{eq:low-high-norm}
  \|u\|_{\dot B^s_{p,r}}^\ell := \big\|\{2^{js} \|\dot \Delta_j u\|_{L^p}\}_{j\leq j_0} \big\|_{\ell^r},
  \quad \textrm{and}\quad
  \|u\|_{\dot B^s_{p,r}}^h := \big\|\{2^{js} \|\dot \Delta_j u\|_{L^p}\}_{j\geq j_0} \big\|_{\ell^r}.
\end{align}

\begin{definition}\label{def:CL}
Let $s, s_1,s_2\in \mathbb{R}$, $(p,q,r)\in[1,\infty]^3$, $T>0$.
The Chemin-Lerner space $\widetilde{L}^q_T(\dot{B}^s_{p,r})$ is defined by
\begin{align}
  \|u\|_{\widetilde{L}^q_T(\dot{B}^s_{p,r})}
  := \big{\Vert} \big\{2^{js}\Vert{\dot{\Delta}_ju}  \Vert_{L^q_T(L^p)}\big\}_{j\in\mathbb{Z}}\big{\Vert}_{\ell^r(\mathbb{Z})},
\end{align}
and $\widetilde{L}^q_T(\widetilde{B}^{s_1,s_2})$ is defined by that for $j_0\in \mathbb{Z}$,
\begin{align}
  \|u\|_{\widetilde{L}^q(\widetilde{B}^{s_1,s_2})}
  :=\sum_{j=-\infty}^{j_0}2^{js_1}\Vert \dot{\Delta}_j  u\Vert_{L^q_T(L^2)}
  + \sum_{j=j_0+1}^{+\infty}2^{js_2}\Vert \dot{\Delta}_j u\Vert_{L^q_T(L^2)}.
\end{align}
\end{definition}

The following Bony's paraproduct decomposition is very useful in the proof.
\begin{definition} \label{def-2.5}
For any $u, v\in \mathcal{S}_h^\prime(\mathbb{R}^N)$, $u v$ has the Bony's paraproduct decomposition:
\begin{equation}\label{eq:Bony-dec}
  u v = T_u v+ T_v u + R(u,v),
\end{equation}
where
\begin{align}\label{def:paraprod}
  {T}_{u}v=\sum_j {\dot{S}_{j-1}u}{\dot{\Delta}_j v}, \quad \textrm{and} \quad {R}(u,v)=\sum_j \dot{\Delta}_j u\widetilde{\dot{\Delta}}_j v, \quad \widetilde{\dot{\Delta}}_j := \dot{\Delta}_{j-1}+\dot{\Delta}_j + \dot{\Delta}_{j+1}.
\end{align}
\end{definition}


We have the following product estimates in the hybrid Besov space $\widetilde{B}^{s_1,s_2}$.
\begin{lemma}\label{lemma.Bilinear}
\begin{enumerate}[(1)]
\item For any $s_1\leq s_1'\leq \frac{N}{2}$ and $s_2\in\R$, there exists a constant $C=C(s_1,s_1',N)$ such that
\begin{align}\label{eq:prod-es1}
  \|T_{\nabla u} v\|_{\widetilde{B}^{s_1 +s_2 -\frac{N}{2}, s_1'+ s_2 -\frac{N}{2}}} \leq C \|u\|_{\widetilde{B}^{s_1,s_1'}} \|v\|_{\dot B^{s_2+1}_{2,1}}.
\end{align}
\item For any $(s_1,s_1',s_2)\in \R^3$ satisfying $s_1'\geq s_1$ and $s_1 +s_2 >0$, there exists a constant $C=C(s_1,s_1',s_2,N)$ such that
\begin{align}\label{eq:prod-es2}
  \|R( u, \nabla v)\|_{\widetilde{B}^{s_1 +s_2 -\frac{N}{2}, s_1'+ s_2 -\frac{N}{2}}} \leq C \|u\|_{\dot B^{s_2+1}_{2,1}} \|v\|_{\widetilde{B}^{s_1,s_1'}}.
\end{align}
\item For any $(s_1,s_1',s_2)\in (-\frac{N}{2},\frac{N}{2}]^3$ satisfying $s_1 + s_2 >0 $ and $s_1'\geq s_1$, there exists a constant $C=C(s_1,s_1',s_2,N)$ such that
\begin{align}\label{eq:prod-es0}
  \| u v\|_{\widetilde{B}^{s_1+s_2 -\frac{N}{2}, s_1' +s_2 -\frac{N}{2}}}
  \le C \|u\|_{\widetilde{B}^{s_1,s_1'}} \| v\|_{\dot B^{s_2}_{2,1}} .
\end{align}
\item
For any $(s_1,s_2)\in \R^2$ satisfying $-\frac{N}{2}<s_1-s_2\leq 0$, there exists a constant $C=C(s_1,s_2,N)$ such that
\begin{align}\label{eq:prod-es4}
  \|u v\|_{\widetilde{B}^{s_1,s_2}} \leq C \big( \|u\|_{\widetilde{B}^{s_1,s_2}} \|v\|_{L^\infty}
  + \|u\|_{\widetilde{B}^{\frac{N}{2}+s_1-s_2,\frac{N}{2}}} \|v\|_{\dot {B}^{s_2}_{2,1}} \big),
\end{align}
and
\begin{align}\label{eq:prod-es5}
  \|T_{\nabla v} u\|_{\widetilde{B}^{s_1,s_2}} + \|R(u,\nabla v)\|_{\widetilde{B}^{s_1,s_2}}
  \leq C  \|u\|_{\dot B^{s_2+1}_{2,1}} \|v\|_{\widetilde{B}^{\frac{N}{2} +s_1-s_2,\frac{N}{2}}}.
\end{align}
\end{enumerate}
\end{lemma}

\begin{proof}[Proof of Lemma \ref{lemma.Bilinear}]
(1) For the proof of \eqref{eq:prod-es1}, noting that $\widetilde{B}^{s_1,s_1'} (\R^N) = \dot B^{s_1}_{2,1} \cap \dot B^{s_1'}_{2,1} (\R^N)$ for every $s_1 \leq s_1'$,
it suffices to show that for every $s_1\leq \frac{N}{2}$ and $s_2\in \R$,
\begin{align}\label{eq:prod-es1-2}
  \|T_{\nabla u} v\|_{\dot B^{s_1 + s_2 -\frac{N}{2}}_{2,1}} \leq C \|u\|_{\dot B^{s_1}_{2,1}} \|v\|_{\dot B^{s_2+1}_{2,1}},
\end{align}
with $C=C(s_1,N)>0$.
While concerning \eqref{eq:prod-es1-2}, using Bernstein's inequality we have
\begin{align*}
  \sum_{j\in\mathbb{Z}}2^{j (s_1 + s_2 -\frac{N}{2})} \|\dot \Delta_j (T_{\nabla u} v)\|_{L^2}
  & \leq \sum_{j\in\mathbb{Z}}2^{j(s_1 + s_2 -\frac{N}{2})} \sum_{|j'-j|\leq 4} \|\dot \Delta_j (\dot S_{j'-1} \nabla u\, \dot \Delta_{j'} v)\|_{L^2} \\
  & \leq C \sum_{j'\in\mathbb{Z}} 2^{j'(s_1 + s_2 -\frac{N}{2})} \|S_{j'-1} u\|_{L^\infty} 2^{j'} \|\dot \Delta_{j'} v\|_{L^2} \\
  & \leq C \sum_{j'\in\mathbb{Z}} \Big( 2^{j'(s_1-\frac{N}{2})} \sum_{k\leq j'-1} 2^{k\frac{N}{2}}\|\dot \Delta_k u\|_{L^2} \Big) 2^{j'(s_2 +1)}\|\dot\Delta_{j'} v\|_{L^2} \\
  & \leq C \|u\|_{\dot B^{s_1}_{2,1}} \|v\|_{\dot B^{s_2 +1}_{2,1}} ,
\end{align*}
as desired.

(2) The proof of \eqref{eq:prod-es2} is quite analogous with that of \eqref{eq:prod-es1} and \cite[Theorem 2.52]{bahouri2011fourier}, and we omit the details.

(3) When $s_1 = s_1' \in (-\frac{N}{2},\frac{N}{2}] $ and $s_2\in (-\frac{N}{2}, \frac{N}{2}]$, we indeed have (see \cite[Corollary 2.55]{bahouri2011fourier} for $(s_1,s_2)\in (-\frac{N}{2},\frac{N}{2})^2$
and \cite[Remark 2.48]{bahouri2011fourier} for a slight modification for the case $s_1=\frac{N}{2}$ or $s_2=\frac{N}{2}$)
\begin{align*}
  \| u v\|_{\dot B^{s_1+s_2 - N/2}_{2,1}}
  \le C \|u\|_{\dot B^{s_1}_{2,1}} \| v\|_{\dot B^{s_2}_{2,1}} .
\end{align*}
Thus \eqref{eq:prod-es0} directly follows from the above inequality.

(4) When $s_1=s_2 >0$, \eqref{eq:prod-es4} is guaranteed by the classical inequality in \cite[Corollary 2.54]{bahouri2011fourier}.
For the general case $0<s_1 < s_2$, we use Bony's decomposition \eqref{eq:Bony-dec}, and noting that for every $j\in \mathbb{Z}$,
\begin{align*}
  & \|\dot \Delta_j T_u v\|_{L^2} \leq \sum_{|j'-j|\leq 4} \|\dot \Delta_j (\dot S_{j'-1} u\, \dot \Delta_{j'}v)\|_{L^2}
  \leq C \sum_{|j'-j|\leq 4} \|\dot S_{j'-1} u\|_{L^\infty} \|\dot \Delta_{j'} v\|_{L^2}, \\
  & \|\dot \Delta_j R(u,v)\|_{L^2}\leq C \sum_{k\geq j-3} 2^{j \frac{N}{2}} \|\dot\Delta_j (\dot\Delta_k u\,\widetilde{\dot\Delta}_k v)\|_{L^1}
  \leq C \sum_{k\geq j-3} 2^{j\frac{N}{2}}\|\dot \Delta_k u\|_{L^2} \|\widetilde{\dot\Delta}_k v\|_{L^2} ,
\end{align*}
we have
\begin{align}\label{es:Tuv-1}
  \|T_u v\|_{\widetilde{B}^{s_1,s_2}} & \leq C \sum_{j\leq j_0} \sum_{|j'-j|\leq 4} 2^{j'(s_1-s_2)}\|\dot S_{j'-1} u\|_{L^\infty} 2^{j' s_2}\|\dot \Delta_{j'} v\|_{L^2} \nonumber\\
  & \quad + C \sum_{j\geq j_0+1} \sum_{|j'-j|\leq 4} \|\dot S_{j'-1} u\|_{L^\infty} 2^{j's_2} \|\dot \Delta_{j'}v\|_{L^2}
  \leq C \|u\|_{\widetilde{B}^{\frac{N}{2}+s_1-s_2,\frac{N}{2}}} \|v\|_{\dot B^{s_2}_{2,1}},
\end{align}
and
\begin{align}\label{es:Ruv-1}
  \|R(u,v)\|_{\widetilde{B}^{s_1,s_2}} & \leq C \sum_{j\leq j_0} \sum_{j-3\leq k\leq j_0} 2^{(j-k) (\frac{N}{2}+ s_1 - s_2)}
  2^{k(\frac{N}{2} + s_1 -s_2 )} \| \dot \Delta_k u\|_{L^2} 2^{k s_2} \|\widetilde{\dot \Delta}_k v\|_{L^2} \nonumber\\
  & \quad + C  \sum_{j\leq j_0} \sum_{k\geq \max\{j-3,j_0\}} 2^{(j-k) \frac{N}{2}}
  2^{k\frac{N}{2}} \| \dot \Delta_k u\|_{L^2} 2^{k s_2} \|\widetilde{\dot \Delta}_k v\|_{L^2} \nonumber \\
  & \quad + C \sum_{j\sim k\sim j_0 , k\leq j_0}
  2^{k(\frac{N}{2} + s_1 -s_2 )} \| \dot \Delta_k u\|_{L^2} 2^{k s_2} \|\widetilde{\dot \Delta}_k v\|_{L^2} \nonumber \\
  & \quad + C \sum_{j\geq j_0+1} \sum_{k\geq \max\{j-3,j_0\}} 2^{(j-k) \frac{N}{2}}
  2^{k\frac{N}{2}} \| \dot \Delta_k u\|_{L^2} 2^{k s_2} \|\widetilde{\dot \Delta}_k v\|_{L^2} \nonumber \\
  & \leq C \|u\|_{\widetilde{B}^{\frac{N}{2} +s_1 -s_2,\frac{N}{2}}} \|v\|_{\dot B^{s_2}_{2,1}} .
\end{align}
Via a direct computation we also get $\|T_v u\|_{\widetilde{B}^{s_1,s_2}} \leq C \|v\|_{L^\infty} \|u\|_{\widetilde{B}^{s_1,s_2}}$.
Hence, collecting the above estimates leads to \eqref{eq:prod-es4}.

While for \eqref{eq:prod-es5}, it can be easily obtained by arguing as the deduction in \eqref{es:Tuv-1} and \eqref{es:Ruv-1}.
\end{proof}

In the analysis of asymptotic behavior, the following weighted paraproduct and reminder estimates play an important role.
\begin{lemma}\label{lemma.bony.low.weight}
Let $j_0 \in \mathbb{Z}$, $r\in\mathbb{R}$, $f, g\in\mathcal {S}_{h}^{\prime}(\mathbb{R}^N)$ and $ \{\psi_j\}_{j\in \mathbb{Z}}$ be a positive sequence.
\begin{enumerate}[(1)]
\item For every $r_1\le 0$, we have
\begin{equation}\label{eq.weight.1}
  \sum_{j\le j_0}\psi_j2^{jr}\Vert\dot{\Delta}_j(T_{f} g)\Vert_{L^2}
  \le C \Big(\sum_{j'\le j_0+4} 2^{j'r_1}\Vert \dot \Delta_{j'} f\Vert_{L^\infty}\Big) \sum_{j\le j_0} \sum_{\vert j-j'\vert \le 4}2^{j'(r-r_1)} \Vert \dot \Delta_{j'}g \Vert_{L^2}\psi_j.
\end{equation}
\item For every $r_2\in\mathbb{R}$, we have
\begin{equation}\label{eq.weight.2}
  \sum_{j\le j_0}\psi_j2^{jr}\Vert\dot{\Delta}_j(T_f g)\Vert_{L^2}
  \le C \Big(\sup_{j'\le j_0+4}2^{j'(r-r_2)}\Vert \dot\Delta_{j'} g\Vert_{L^2}\Big) \sum_{j\le j_0}\sum_{j'\le j+4}2^{jr_2}\Vert \dot\Delta_{j'}f \Vert_{L^\infty}\psi_j.
\end{equation}
\item
If $r> - \frac{N}{2}$, $r_3\in\mathbb{R}$ and $0\le \beta\le r+\frac{N}{2}$, we have
\begin{equation}\label{eq:LP-rem}
\begin{aligned}
  \sum_{j\le j_0}\psi_j2^{jr}\Vert\dot{\Delta}_jR(f,g)\Vert_{L^2}
  \le C \Big( \sup_{j'\in\mathbb{Z}}2^{j'(r_3+\frac{N}{2}-\beta )} \Vert \dot \Delta_{j'}g \Vert_{L^2} \Big) \\
  \times\sum_{j\le j_0}\sum_{j'>j-4}2^{(j-j')(r+\frac{N}{2}-\beta )}
  &2^{j'(r-r_3)}\Vert \dot\Delta_{j'}f \Vert_{L^2}\psi_j.
\end{aligned}
\end{equation}
\end{enumerate}
\end{lemma}

\begin{proof}[Proof of Lemma \ref{lemma.bony.low.weight}]
(1) By virtue of the spectral support property, we see that
\begin{equation}\label{eq.ab.s}
\begin{aligned}
  \sum_{j\le j_0}\psi_j2^{jr}\Vert\dot{\Delta}_j(T_{f} g)\Vert_{L^2}
  & \le C\sum_{j\le j_0}\sum_{\vert j-j'\vert \le 4}\psi_j2^{jr} \Vert \dot\Delta_j( \dot{S}_{j'-1}f \, \dot \Delta_{j'} g )\Vert_{L^2} \\
  &\le C\sum_{j\le j_0}\sum_{\vert j-j'\vert \le 4}\psi_j2^{jr}\Vert\dot{S}_{j'-1}f\Vert_{L^\infty} \Vert \dot \Delta_{j'} g \Vert_{L^2}.
\end{aligned}
\end{equation}
Plugging the definition of $\dot{S}_{j'-1}f$ to \eqref{eq.ab.s}, we deduce that for every $r_1\le 0$,
\begin{equation*}
\begin{aligned}
  \sum_{j\le j_0}\psi_j2^{jr}\Vert\dot{\Delta}_j(T_{f} g)\Vert_{L^2} & \le C \Big(\sup_{j'\le j_0+4} 2^{j'r_1} \Vert \dot S_{j'-1} f\Vert_{L^\infty} \Big)
  \sum_{j\le j_0}\sum_{\vert j-j'\vert \le 4}\psi_j2^{j'(r-r_1)} \Vert \dot\Delta_{j'} g \Vert_{L^2} \\
  & \leq C\Big(\sum_{j'\le j_0+4} 2^{j'r_1}\Vert \dot \Delta_{j'} f\Vert_{L^\infty}\Big) \sum_{j\le j_0} \sum_{\vert j-j'\vert \le 4}2^{j'(r-r_1)} \Vert \dot \Delta_{j'}g \Vert_{L^2}\psi_j.
\end{aligned}
\end{equation*}

(2) The inequality \eqref{eq.weight.2} for every $r_2\in \mathbb{R}$ can be deduced in the same manner.

(3) By using the spectral property of the dyadic operators, we get
\begin{equation*}
  \sum_{j\le j_0}\psi_j2^{jr}\Vert\dot{\Delta}_jR(f,g)\Vert_{L^2}
  \le C\sum_{j\le j_0}\sum_{j'>j-3}\psi_j2^{j(r+\frac{N}{2}-\beta )} \Vert\widetilde{\dot{\Delta}}_{j'}f\Vert_{L^2} \Vert \dot\Delta_{j'}g \Vert_{L^2},
\end{equation*}
where we have used the fact that $2^{j\beta}\le C$ for every $j\le j_0$.
Consequently, the desired inequality \eqref{eq:LP-rem} directly follows.
\end{proof}

Let us state a continuity result for the composition in Besov space (for the proof one can see Proposition 1.5.13, Corollary 1.4.9 of \cite{danchin2005fourier} and Proposition A.3 of \cite{danchin2017optimal}).
\begin{lemma}\label{lemma.composition}
Let $s\in \R $ and $(p, r)\in[1,\infty]^2$ be such that
\begin{align*}
  0< s<\frac{N}{p},\quad \text{or}\quad s=\frac{N}{p} \text{ and } r=1.
\end{align*}
Let $I$ be an open interval of $\R$, and let $f: I\rightarrow \R$ satisfy $f(0)=0$ and $f',f''\in W^{[s]+1,\infty}(I;\mathbb{R})$.
Assume that $u,v\in \dot B^s_{p,r}\cap L^\infty(\R^N)$ has values in $J \subset I$.
Then the function $f(u)$ belongs to $\dot B^s_{p,r}(\R^N)$, and there exists a constant $C=C(s,I,J,N)$ such that
\begin{equation}\label{eq:compest1}
  \|f(u)\|_{\dot{B}^s_{p,r}} \le C (1 + \|u\|_{L^{\infty}})^{[s]+1} \|f'\|_{W^{[s]+1,\infty}(I)}
  \|u\|_{\dot{B}^s_{p,r}},	
\end{equation}
and
\begin{align}\label{eq:compest2}
  \|f(v) -f(u)& \|_{\dot{B}^s_{p,r}} \leq C (1+\|v\|_{L^\infty})^{[s]+1} \|f''\|_{W^{[s]+1,\infty}(I)}  \\
  & \quad \cdot \Big(\|v-u\|_{\dot{B}^s_{p,r}} \sup_{\tau\in[0,1]} \|u+\tau(v-u)\|_{L^\infty}
  + \|v-u\|_{L^\infty} \sup_{\tau\in [0,1]} \|u+\tau(v-u)\|_{\dot{B}^s_{p,r}}\Big),	
\end{align}
where $[s]$ denotes the integer part of $s$.

In the case $s>-\frac{N}{p}$ then $u\in \dot{B}^{s}_{p,r}\cap\dot{B}^{\frac{N}{p}}_{p,1}$ implies that
\begin{align}\label{eq:comp-Bes1}
  \|f(u)\|_{\dot{B}^s_{p,r}} \le C (1 + \|u\|_{\dot{B}^{\frac{N}{p}}_{p,1}})
  \|u\|_{\dot{B}^s_{p,r}}.
\end{align}
\end{lemma}

\begin{remark}\label{rmk:compos}
  After some trivial modification, the above composition inequalities can be adapted to the hybrid Besov space $\widetilde{B}^{s_1,s_2}$ with $0<s_1,s_2\leq \frac{N}{2}$.
\end{remark}

The result below is useful in the existence part.
\begin{lemma}[{\cite[Proposition 2.93]{bahouri2011fourier}}]\label{lem:Bes-K}
  Let $s>0$, $(p,r)\in [1,\infty]^2$ and $K$ be a compact set of $\R^N$.
Assume that $u\in \dot B^s_{p,r}(\R^N)$ is a tempered distribution with the support included in $K$.
Then $u\in B^s_{p,r}(\R^N)$ (the usual nonhomogeneous Besov space) and there exists a universal constant $C>0$ such that
\begin{align*}
  \|u\|_{B^s_{p,r}(\R^N)} \leq C (1+ |K|)^{\frac{s}{N}} \|u\|_{\dot B^s_{p,r}(\R^N)}.
\end{align*}
\end{lemma}

The following three lemmas are concerned with the commutator estimates.
\begin{lemma}\label{corollary.commutator_S_j}
Let $\alpha >1$. Then the following inequality holds true:
\begin{equation}\label{eq:com-es1}
  \Vert[\Lambda^{\alpha-1},\dot S_{j-1}v\cdot\nabla]\dot{\Delta}_j \sigma\Vert_{L^2}
  \le C 2^{j(\alpha-1)}\Vert \nabla v\Vert_{L^{\infty}}\Vert \dot{\Delta}_j \sigma\Vert_{L^2}.
\end{equation}
\end{lemma}

\begin{proof}[Proof of Lemma \ref{corollary.commutator_S_j}]
Taking advantage of the spectrum support property, we see that
\begin{align*}
  &\quad [\Lambda^{\alpha-1},\dot S_{j-1}v\cdot\nabla]\dot{\Delta}_j \sigma= \sum_{\vert j-j'\vert\le4}[\Lambda^{\alpha-1}\dot{\Delta}_{j'},\dot{S}_{j-1}v\cdot\nabla]\dot{\Delta}_j \sigma \\
  & = \sum_{|j'-j|\leq 4} 2^{j'(N+\alpha-1)} \int_{\R^N} \widetilde{h}(2^{j'}y) \big(\dot S_{j-1}v(x-y) -\dot S_{j-1}v(x)\big)\cdot \nabla \dot\Delta_j \sigma(x-y) \dd y,
\end{align*}
where $\Lambda^{\alpha-1}\dot\Delta_{j'} = 2^{j'(N+\alpha -1)} \widetilde{h}(2^{j'}x)*$ and $\widetilde{h} = \mathcal{F}^{-1}(|\xi|^{\alpha-1} \varphi(\xi))\in \mathcal{S}(\R^N)$.
Hence, combined with H\"older's and Bernstein's inequalities, the desired inequality \eqref{eq:com-es1} easily follows.
\end{proof}

\begin{lemma}[{\cite[Lemma 2.100 and Remark 2.103]{bahouri2011fourier}}]\label{lemma.Remark2.103}
Let $-\frac{N}{2}< s\le \frac{N}{2}$.
There exists a constant $C$ such that
\begin{equation*}
\begin{aligned}
  \Big{\Vert} \Big(2^{js}\Vert\dot{S}_{j-1}v\cdot\nabla\dot{\Delta}_j f-\dot{\Delta}_j(v\cdot\nabla f)\Vert_{L^2}\Big)_j
  \Big{\Vert}_{\ell^{1}}\le C\Vert \nabla v\Vert_{\dot{B}^{\frac{N}{2}}_{2,1}}\Vert f\Vert_{\dot{B}^s_{2,1}}.
\end{aligned}
\end{equation*}
\end{lemma}

\begin{lemma}[{\cite[Lemma 10.25]{bahouri2011fourier}}]\label{lemma.commutator_estimate}
Let $m\in\mathbb{R}$, and let $A(D)$ be a smooth homogeneous multiplier of degree $m$.
There exists a constant $C = C(m,A,N)$ such that for all $p \in [1,\infty]$ the following inequality holds true:
\begin{equation*}
  \|\dot{S}_{j-1}a\dot{\Delta}_jA(D)b  - A(D)\dot{\Delta}_jT_{a} b \|_{L^p}
  \le C2^{j(m-1)}\sum_{j',j''=j-4}^{j+4}\big(\Vert \nabla \dot{\Delta}_{j''}a\Vert_{L^{\infty}}+\Vert \nabla \dot{S}_{j-1}a\Vert_{L^{\infty}}\big)
  \Vert \dot{\Delta}_{j'}b\Vert_{L^p}.
\end{equation*}
\end{lemma}

\subsection{Fractional Laplacian and Kato-Ponce type inequality}
\begin{definition}[Fractional Laplacian]\label{def:fraLap}
For every $u$ belonging to the Schwartz class $\mathcal{S}(\R^N)$, the fractional Laplacian operator $\Lambda^\alpha$ with $0<\alpha<2$ is defined as
\begin{align}\label{def:Lam-alp}
  {\Lambda}^\alpha u(x)=( - \Delta)^{\frac{\alpha}{2}} u (x) = c_{\alpha,N}\, \mathrm{p.v.} \int_{\mathbb{R}^N} \frac{u(x)-u(y)}{|x-y|^{N+\alpha}} dy,
\end{align}
with $c_{\alpha,N} = \frac{2^\alpha  \Gamma ( N/2 + \alpha/2 )}{\pi^{N/2} \Gamma (- \alpha /2 ) }$.
\end{definition}

We first recall the following Kato-Ponce inequality associated with the fractional Laplacian operator $\Lambda^\alpha$,
of which proof can be found in Li's paper \cite[Theorem 1.2]{li2019kato}. 
\begin{lemma}\label{lemma.kato_ponce}
Let $\alpha>0$, $\alpha_1,\alpha_2\ge0$ be with $\alpha_1+\alpha_2=\alpha$, and let $1<p, p_1, p_2 <\infty$ be satisfying $\frac{1}{p_1} + \frac{1}{p_2} =\frac{1}{p}$. Then for every $u,v\in\mathcal{S}(\mathbb{R}^N)$, we have
\begin{equation}\label{eq:Li-KPin1}
\begin{aligned}
  \Big\| \Lambda^\alpha(uv)- \sum_{|k|< \alpha_1} \frac{1}{k!} \partial^k u \,\Lambda^{\alpha,k} v - \sum_{|m| \le \alpha_2}\frac{1}{m!} \partial^{m} v\, \Lambda^{\alpha,m} u \Big\|_{L^p}  \lesssim_{\alpha,\alpha_1,\alpha_2,p,N} \| \Lambda^{\alpha_1} u \|_{L^p} \| \Lambda^{\alpha_2} v \|_{\operatorname{BMO}},
\end{aligned}
\end{equation}
and
\begin{equation}\label{eq:Li-KPin2}
\begin{aligned}
  \Big\| \Lambda^{\alpha}(uv)- \sum_{|k|\le \alpha_1} \frac{1}{k!} \partial^k u\, \Lambda^{\alpha,k} v - \sum_{|m| \le \alpha_2}
  \frac{1}{m!} \partial^{m} v \,\Lambda^{\alpha,m} u \Big\|_{L^p}  \lesssim_{\alpha,\alpha_1,\alpha_2,p,p_1,p_2,N} \| \Lambda^{\alpha_1} u \|_{L^{p_1}} \| \Lambda^{\alpha_2} v \|_{L^{p_2}},
\end{aligned}
\end{equation}
where $k=(k_1,\cdots,k_N)$, $m=(m_1,\cdots,m_N)$, $\widehat{\Lambda^{\alpha,k} v} (\xi) = i^{-|k|} \partial_{\xi}^k (|\xi|^\alpha) \widehat v(\xi)$
and the $\operatorname{BMO}$ semi-norm is given by
$ \Vert u\Vert_{\operatorname{BMO}}=\Vert (\sum\limits_{j\in\mathbb{Z}}\vert \dot{\Delta}_j u\vert^2)^{\frac{1}{2}}\Vert_{L^{\infty}}$.
\end{lemma}

A consequence of Lemma \ref{lemma.kato_ponce} is the following commutator estimate. 
\begin{lemma}\label{cor.commu.est}
Let $1< \alpha < 2$, $1<p <\infty$ and $p_1, p_2 \in [p, \infty]$ be satisfying $\frac{1}{p_1} + \frac{1}{p_2} =\frac{1}{p}$. Then we have
\begin{equation}\label{eq:Commu-est-Lp}
\begin{aligned}
  \Vert \Lambda^{\alpha}(uv)- u\Lambda^{\alpha}v \Vert_{L^p}
  \lesssim_{\alpha,N} \Vert\nabla u \Vert_{L^{p_1}} \Vert \Lambda^{\alpha-1}v \Vert_{L^{p_2}}+\Vert \Lambda^{\alpha}u \Vert_{L^{p_1}} \Vert v \Vert_{L^{p_2}}.
\end{aligned}
\end{equation}
\end{lemma}

\begin{proof}[Proof of Lemma \ref{cor.commu.est}]
  For every $p<p_2\le \infty$, by taking $\alpha_1=1$ and $\alpha_2=\alpha-1$, the inequalities \eqref{eq:Li-KPin1}-\eqref{eq:Li-KPin2} become
\begin{align*}
  \|\Lambda^\alpha(uv) - u \Lambda^\alpha v -v\Lambda^\alpha u \|_{L^p} \lesssim_{\alpha,N} 
  \|\Lambda u\|_{L^{p_1}} \|\Lambda^{\alpha-1} v\|_{L^{p_2}} \lesssim_{\alpha,N} \|\nabla u\|_{L^{p_1}} \|\Lambda^{\alpha-1} v\|_{L^{p_2}},
\end{align*}
thus the desired estimate \eqref{eq:Commu-est-Lp} follows from the triangle inequality and H\"older's inequality.
Similarly, by taking $\alpha_1=\alpha-1$ and $\alpha_2=1$ and switching $u,v$ in \eqref{eq:Li-KPin1},
and using the Calder\'on-Zygmund theorem of singular integral operator,
we obtain
\begin{align*}
  \|\Lambda^\alpha(uv) - v \Lambda^\alpha u - u \Lambda^\alpha v - \partial u \, \Lambda^{\alpha,1}v\|_{L^p} \lesssim_{\alpha,N}
  \|\Lambda u\|_{\mathrm{BMO}} \|\Lambda^{\alpha-1} v\|_{L^p} \lesssim_{\alpha,N} \|\nabla u\|_{L^\infty} \|\Lambda^{\alpha-1}v\|_{L^p},
\end{align*}
which leads to \eqref{eq:Commu-est-Lp} in the case $p_2=p$.
\end{proof}

Now, we present a useful commutator estimate in Besov space as follows.
\begin{lemma}\label{cor.K-P.Besov}
Let $1< \alpha < 2$, $s>-\frac{N}{2}$ and $s_1,s_2\ge 0$. Then we have
\begin{equation}\label{eq:Commu-est-Besov}
\begin{aligned}
  \Vert \Lambda^{\alpha}(uv)- u\Lambda^{\alpha}v \Vert_{\dot{B}^{s}_{2,1}}
  \lesssim_{\alpha,N,s,s_1,s_2} \Vert u \Vert_{\dot{B}^{\frac{N}{2}+1-s_1}_{2,1}} \Vert v \Vert_{\dot{B}^{s+s_1+\alpha-1}_{2,1}}
  +\Vert u \Vert_{\dot{B}^{s+s_2+\alpha}_{2,1}} \Vert v \Vert_{\dot{B}^{\frac{N}{2}-s_2}_{2,1}}.
\end{aligned}
\end{equation}
\end{lemma}

\begin{proof}[Proof of Lemma \ref{cor.K-P.Besov}]
  We here apply Lemma \ref{cor.commu.est} to show \eqref{eq:Commu-est-Besov}.
Using Bony's decomposition \eqref{eq:Bony-dec} leads to
\begin{align*}
  \Lambda^{\alpha}(uv)-  u\Lambda^{\alpha}v =\big(\Lambda^{\alpha}T_{u}v -T_{u}(\Lambda^{\alpha}v)\big)+(\Lambda^{\alpha}T_{v}u - T_{\Lambda^{\alpha}v}u)
  + \big(\Lambda^{\alpha}R(u,v)-R(u,\Lambda^{\alpha}v)\big)=: \sum_{j=1}^3 \Pi_j.
\end{align*}
In light of Lemma \ref{cor.commu.est} and the spectrum support property of dyadic operators, we get
\begin{align}\label{eq:KP-Besov-T1}
  \Vert \Pi_1 \Vert_{\dot{B}^{s}_{2,1}}
  &\lesssim \sum_{j\in\mathbb{Z}}2^{js}\Vert \Lambda^{\alpha}(\dot{S}_{j-1}u\dot{\Delta}_j v)-\dot{S}_{j-1}u\, \Lambda^{\alpha}\dot{\Delta}_jv \Vert_{L^2} \nonumber\\
  &\lesssim \sum_{j\in\mathbb{Z}}2^{js}\big(\Vert \Lambda^\alpha \dot{S}_{j-1}u\Vert_{L^{\infty}} \Vert\dot{\Delta}_j v\Vert_{L^2} +
  \Vert \nabla \dot{S}_{j-1}u\Vert_{L^{\infty}}\Vert\Lambda^{\alpha-1}\dot{\Delta}_j v\Vert_{L^2}\big) \nonumber \\
  & \lesssim \sum_{j\in \mathbb{Z}} \Big(2^{-js_1}\sum_{j'\leq j-1} 2^{j'(\frac{N}{2}+1)} \|\dot \Delta_{j'} u\|_{L^2} \Big)
  2^{j(s +s_2+\alpha-1)} \|\dot \Delta_j v\|_{L^2} \nonumber \\
  &\lesssim \Vert u \Vert_{\dot{B}^{\frac{N}{2}+1-s_1}_{2,1}} \Vert v \Vert_{\dot{B}^{s+s_1+\alpha-1}_{2,1}}.
\end{align}
Similarly, we have
\begin{align}\label{eq:KP-Besov-T2}
  \Vert \Pi_2 \Vert_{\dot{B}^{s}_{2,1}}
  &\lesssim \sum_{j\in\mathbb{Z}}2^{js}\Vert \Lambda^{\alpha}(\dot{S}_{j-1}v\dot{\Delta}_j u)-\Lambda^{\alpha}(\dot{S}_{j-1}v) \, \dot{\Delta}_ju \Vert_{L^2} \nonumber\\
  &\lesssim \sum_{j\in\mathbb{Z}}2^{js}\big(\Vert \Lambda^{\alpha}\dot{\Delta}_ju\Vert_{L^2} \Vert\dot{S}_{j-1} v\Vert_{L^{\infty}}
  + \Vert \nabla \dot{\Delta}_ju\Vert_{L^2} \Vert\Lambda^{\alpha-1}\dot{S}_{j-1} v\Vert_{L^{\infty}}\big) \nonumber \\
  &\lesssim \Vert u \Vert_{\dot{B}^{\alpha+s+s_2}_{2,1}} \Vert v \Vert_{\dot{B}^{\frac{N}{2}-s_2}_{2,1}}.
\end{align}
For the term $\Pi_3$, we do not need to use the commutator structure, and we infer that for every $s>-\frac{N}{2}$,
\begin{align}\label{eq:KP-Besov-R}
  \Vert \Pi_3 \Vert_{\dot{B}^{s}_{2,1}}
  & \lesssim \sum_{j\in\mathbb{Z}} 2^{js} \big(2^{j\alpha} \|\dot \Delta_j R(u,v)  \|_{L^2} + \|\dot \Delta_j R(u,\Lambda^\alpha v)\|_{L^2} \big) \nonumber \\
  &\lesssim \sum_{j\in\mathbb{Z}}\sum_{k> j-3}2^{j(s+\frac{N}{2})} \big(
  2^{j\alpha} \Vert \dot \Delta_j(\dot{\Delta}_k u \widetilde{\dot{\Delta}}_k v) \Vert_{L^1} +
  \Vert \dot \Delta_j (\dot{\Delta}_k u\,\Lambda^\alpha(\widetilde{\dot{\Delta}}_k v)) \Vert_{L^1} \big) \nonumber \\
  &\lesssim \sum_{j\in\mathbb{Z}}\sum_{k > j-3} 2^{(j-k)(s+\frac{N}{2})}
  2^{k(s+s_2+\alpha)}\Vert\dot{\Delta}_k u\Vert_{L^2} 2^{k (\frac{N}{2} -s_2)}\Vert\widetilde{\dot{\Delta}}_k v\Vert_{L^2}\nonumber \\
  &\lesssim \Vert u \Vert_{\dot{B}^{s +s_2+\alpha}_{2,1}} \Vert v \Vert_{\dot{B}^{\frac{N}{2}-s_2}_{2,1}}.
\end{align}
Combining these inequalities \eqref{eq:KP-Besov-T2}--\eqref{eq:KP-Besov-R} completes the proof of this lemma.
\end{proof}

As a direct consequence of Lemma \ref{cor.K-P.Besov} with $s=\frac{N}{2}+1-r$, $s_1=r_1-\alpha$ and $s_2=r-r_1$, we have the following commutator estimate.
\begin{corollary}\label{corollary.kato_ponce}
Let $1<\alpha<2$ and $\alpha\le r_1\le r< N+1$. Then the following inequality holds true:
\begin{equation}\label{eq:KP-ine1}
\begin{aligned}
  \Vert \Lambda^\alpha(uv)- u\Lambda^\alpha v\Vert_{\dot{B}^{\frac{N}{2}+1-r}_{2,1}}  \lesssim_{\alpha,r,r_1,N} \Vert u \Vert_{\dot{B}^{\frac{N}{2}+1+\alpha-r_1}_{2,1}}\Vert v \Vert_{\dot{B}^{\frac{N}{2}+r_1-r}_{2,1}}.
\end{aligned}
\end{equation}
\end{corollary}

\begin{remark}
  It is worth mentioning that Danchin et al in \cite[Lemma 3.1]{danchin2019regular} established a related commutator estimate (with different Besov index),
and they mainly use the intrinsic definition of Besov space.
\end{remark}

\section{A priori estimates}\label{priori}
This section is devoted to establishing \emph{a priori} estimates for the Euler-alignment system \eqref{eq.EA}, with isothermal or isentropic pressures \eqref{eq:pressure}, and strongly singular alignment interactions \eqref{eq:alignment}.

Let us begin with rewriting the system into a more treatable form
by introducing a new quantity
\begin{equation}\label{eq.rho.sigma}
  \sigma :=
\begin{cases}
  \frac{\sqrt{\kappa\gamma}}{\gamma-1}(\rho^{\gamma-1}-1),\ &\gamma> 1,\\
  \sqrt{\kappa} \ln \rho ,\ &\gamma=1.
\end{cases}
\end{equation}
It is easy to see that $\rho = 1 + h(\sigma)$ with
\begin{equation}\label{eq:h}
h(\sigma) :=
\begin{cases}
  \left(\frac{\gamma-1}{\sqrt{\kappa\gamma}}\sigma+1\right)^{\frac{1}{\gamma-1}}-1,\ &\gamma>1,\\
  e^{\sigma/\sqrt{\kappa}}-1,\ &\gamma=1.
\end{cases}
\end{equation}
Consequently, $(\sigma,u)$ satisfies the following coupled system
\begin{equation}\label{eq.EAsigma.nl}
\begin{cases}
\partial_t \sigma+u\cdot \nabla\sigma+(\gamma-1)\sigma\Div u+ \lambda\Div u=0,\\
\partial_t u+u\cdot\nabla u +\mu\Lambda^\alpha u +\lambda\nabla \sigma =-\mu \big(\Lambda^\alpha (uh(\sigma)) - u \Lambda^\alpha h(\sigma)\big),\\
(\sigma,u)|_{t=0}=(\sigma_0,u_0),
\end{cases}
\end{equation}
where $\lambda : = \sqrt{\kappa\gamma}$ and $\mu:=\frac{1}{c_{\alpha,N}}$.

In the following, we first study paralinearized equations of \eqref{eq.EAsigma.nl}, and obtain \emph{a priori} estimates in a hybrid Besov space. The analysis is inspired by the work of Danchin \cite{danchin2000global} (see also \cite{charve2010global,chen2010global}). We then calculate some nonlinear estimates and obtain \emph{a priori} local/global uniform estimates for the system \eqref{eq.EAsigma.nl}. 

\subsection{A priori estimates for the paralinearized equations}
In this subsection we study the following paralinearized system
\begin{equation}\label{eq.lineareqbn0}
\begin{cases}
  \partial_t \sigma + T_v\cdot\nabla \sigma + \lambda\Div{u}=F,\\
  \partial_t u + T_v \cdot \nabla u+ \mu\Lambda^\alpha u+\lambda\nabla \sigma=G, \\
  (\sigma, u)|_{t=0} = (\sigma_0,u_0),
\end{cases}
\end{equation}
where  $v$ is a given vector field of $\mathbb{R}^N$, and $F$, $G$ are given source terms. We denote
\begin{align*}
  T_v\cdot \nabla g:=\sum_{i=1}^N T_{v_i}\partial_{x_i}g = \sum_{i=1}^N \sum_{j\in \mathbb{Z}} \dot{S}_{j-1} v_i \dot\Delta_j \partial_{x_i}g.
\end{align*}
The system \eqref{eq.lineareqbn0} contains the major linear structures of \eqref{eq.EAsigma.nl}. \textit{A priori} energy estimates of  \eqref{eq.lineareqbn0} are stated as follows.
\begin{proposition}\label{prop.lin.pro}
For every $s\in\mathbb{R}$, the smooth solution $(\sigma,u)$ of the paralinearized system \eqref{eq.lineareqbn0} satisfies that
\begin{equation}\label{eq.lin.pro}
\begin{aligned}
  &\quad \Vert \sigma\Vert_{\widetilde{L}_t^{\infty} (\widetilde{B}^{s,s+\alpha-1})} + \Vert u\Vert_{\widetilde{L}_t^\infty (\dot{B}^s_{2,1})}
  + \int_0^t \Vert \sigma(\tau)\Vert_{\widetilde{B}^{s+\alpha,s+1}} \dd \tau + \int_0^t \Vert u(\tau)\Vert_{\dot{B}^{s+\alpha}_{2,1}} \dd \tau \\
  &\leq C e^{C V(t)}\Big(\Vert \sigma_0\Vert_{\widetilde{B}^{s,s+\alpha-1}}+\Vert u_0\Vert_{\dot{B}^s_{2,1}}+
  \int_0^t \Vert F(\tau) \Vert_{\widetilde{B}^{s,s+\alpha-1}} \dd \tau
  + \int_0^t \Vert G(\tau)\Vert_{\dot{B}^s_{2,1}} \dd \tau \Big),
\end{aligned}
\end{equation}
with
$V(t):=\int_0^t \Vert v(\tau)\Vert_{\dot{B}^{\frac{N}{2}+1}_{2,1}}\dd \tau$ and $C= C(\alpha, N)>0$.
\end{proposition}

\begin{proof}
Denote by $\mathbb{P}:= \Id - \nabla \Delta^{-1} \Div $ the projection operator onto the divergence-free field and
\begin{align}\label{eq:d}
  d := \Lambda^{-1} \Div u.
\end{align}
We decompose $u$ into two parts
\[u = - \nabla \Lambda^{-1} d + \mathbb{P} u.\]
Then \eqref{eq.lineareqbn0} becomes
\begin{equation}\label{eq.lineareq-sig-d}
\begin{cases}
  \partial_t \sigma+  T_{v}\cdot \nabla\sigma + \lambda \Lambda d =  F,\\
  \partial_t d +\Lambda^{-1}\Div (T_{v} \cdot \nabla u) + \mu\Lambda^\alpha d - \lambda \Lambda \sigma =  \Lambda^{-1} \Div G, \\
  \partial_t \mathbb{P} u + \mathbb{P}(T_v \cdot \nabla u) +  \mu \Lambda^\alpha \mathbb{P}u =  \mathbb{P} G.
\end{cases}
\end{equation}
From \eqref{eq.lineareq-sig-d}, we see that $(\dot{\Delta}_j \sigma, \dot{\Delta}_j u)$ satisfy that
\begin{equation}\label{eq.lineareqibn0}
\begin{cases}
  \partial_t \dot{\Delta}_j \sigma + \dot S_{j-1} v\cdot\nabla \dot{\Delta}_j \sigma + \lambda \Lambda {\dot{\Delta}_j} d = f_j,\\
  \partial_t \dot{\Delta}_j d +\dot S_{j-1} v \cdot \nabla \dot{\Delta}_j d + \mu\Lambda^\alpha  \dot{\Delta}_j d - \lambda \Lambda \dot{\Delta}_j \sigma = g_j,\\
  \partial_t\dot{\Delta}_j \mathbb{P}u + \dot S_{j-1} v \cdot \nabla\dot{\Delta}_j \mathbb{P}u+ \mu\Lambda^\alpha\dot{\Delta}_j \mathbb{P}u =\widetilde{g}_j,
\end{cases}
\end{equation}
where
\begin{align}
  f_j & :=\dot{\Delta}_j F + \dot S_{j-1} v\cdot\nabla \dot{\Delta}_j \sigma - \dot{\Delta}_j (T_{v}\cdot\nabla \sigma), \label{eq:fj} \\
  g_j & :=\Lambda^{-1}\Div \dot{\Delta}_j G + \dot S_{j-1} v\cdot\nabla\dot{\Delta}_j d - \Lambda^{-1} \Div \dot{\Delta}_j (T_{v}\cdot \nabla u), \label{eq:gj} \\
  \widetilde{g}_j & : =\mathbb{P}\dot{\Delta}_j G + \dot S_{j-1} v \cdot\nabla \mathbb{P}\dot{\Delta}_j u-\mathbb{P}\dot{\Delta}_j (T_{v}\cdot \nabla u). \label{eq:tild-gj}
\end{align}

Now we estimate the compressible part $(\dot{\Delta}_j \sigma, \dot{\Delta}_j d)$ and the incompressible part $\mathbb{P}\dot{\Delta}_j u$ respectively. For the sake of simplicity, we denote $(f| g)=\int_{\R^d} f(x)\cdot g(x) \dd x$ in the sequel.

\textbf{Step 1: the compressible part.}
Taking the $L^2$ inner product of the first two equations of \eqref{eq.lineareqibn0} respectively and using the integration by parts give that
\begin{equation}\label{eq.comp.aj}
  \frac{1}{2}\frac{d}{dt} \Vert \dot{\Delta}_j \sigma\Vert^2_{L^2} + \lambda \int\dot{\Delta}_j \sigma\, \Lambda{\dot{\Delta}_j d} \dd x
  = \int f_j\, \dot{\Delta}_j \sigma \dd x + \frac{1}{2} \int (\dot{\Delta}_j \sigma)^2 \Div \dot S_{j-1} v \dd x,
\end{equation}
and
\begin{equation}\label{eq.comp.Quj}
  \frac{1}{2}\frac{d}{dt}\Vert \dot{\Delta}_j d \Vert^2_{L^2}
  + \mu\Vert\Lambda^{\frac{\alpha}{2}}\dot{\Delta}_j d \Vert^2_{L^2}
  - \lambda\int  \dot{\Delta}_j \sigma \Lambda \dot{\Delta}_j d \dd x
  = \int g_j \dot{\Delta}_j d \dd x + \frac{1}{2}\int (\dot{\Delta}_j d)^2 \Div \dot S_{j-1} v \dd x.
\end{equation}
In order to get the smoothing effect of $\sigma$, we consider the cross term $(\Lambda^{\alpha-1} \dot{\Delta}_j \sigma | \dot{\Delta}_j d)$.
Applying the operator $\Lambda^{\alpha-1}$ to the first equation of \eqref{eq.lineareqibn0} yields
\begin{equation}\label{eq:sigma-Lam}
\begin{aligned}
  \partial_t \Lambda^{\alpha-1} \dot{\Delta}_j \sigma+\Lambda^{\alpha-1}(\dot S_{j-1} v \cdot\nabla \dot{\Delta}_j \sigma)
  + \lambda \Lambda^\alpha {\dot{\Delta}_j d} = \Lambda^{\alpha-1} f_j.
\end{aligned}
\end{equation}
Together with the second equation of \eqref{eq.lineareqibn0}, we get
\begin{align}\label{eq.comp.crossj}
  &\quad\frac{d}{dt} (\dot{\Delta}_j d | \Lambda^{\alpha-1} \dot{\Delta}_j \sigma)
  - \lambda\Vert \Lambda^{\frac{\alpha}{2}} \dot{\Delta}_j \sigma\Vert^2_{L^2}
  + \lambda\Vert\Lambda^{\frac{\alpha}{2}}\dot{\Delta}_j d \Vert_{L^2}^2 + \mu\int \Lambda^\alpha \dot{\Delta}_j d \cdot \Lambda^{\alpha-1} \dot{\Delta}_j \sigma \dd x \nonumber \\
  &=\int \Lambda^{\alpha-1} f_j\cdot \dot{\Delta}_j d \dd x + \int g_j\cdot \Lambda^{\alpha-1} \dot{\Delta}_j \sigma \dd x \nonumber \\
  &\quad- \int (\dot S_{j-1} v \cdot \nabla \dot{\Delta}_j d)\cdot \Lambda^{\alpha-1} \dot{\Delta}_j \sigma \dd x - \int \Lambda^{\alpha-1} (\dot S_{j-1} v\cdot\nabla \dot{\Delta}_j \sigma ) \cdot \dot{\Delta}_j d\, \dd x.
\end{align}

To proceed, we work with low frequencies and high frequencies separately.
Define
\begin{equation}\label{eq:j0}
  j_0 := \frac{1}{\alpha -1} \log_2 \frac{4\lambda}{\mu},
  \qquad\text{i.e.}\quad
  2^{j_0(\alpha-1)} = \frac{4\lambda}{\mu}.
\end{equation}

For low frequencies $j\leq j_0$, namely
\begin{equation}\label{eq:jlow}
  2^{j(\alpha-1)} \leq \frac{4\lambda}{\mu},
\end{equation}
we set
\begin{align*}
  Y_j^2 = \Vert \dot{\Delta}_j \sigma\Vert^2_{L^2}+ \Vert \dot{\Delta}_j d\Vert^2_{L^2} - \delta \frac{\mu}{\lambda} (\dot{\Delta}_j d | \Lambda^{\alpha-1}  \dot{\Delta}_j \sigma),\quad j\leq j_0,
\end{align*}
where $\delta>0$ is a suitable small constant, e.g. $\delta=\frac{1}{300}$.
It is easy to check that
\begin{align}\label{eq:Yj-low-equ}
  Y_j\approx \Vert\dot{\Delta}_j \sigma\Vert_{L^2}+\Vert \dot{\Delta}_j d\Vert_{L^2}.
\end{align}
Gathering the equations \eqref{eq.comp.aj}, \eqref{eq.comp.Quj} and \eqref{eq.comp.crossj}, we see that $Y_j$ satisfies
\begin{equation}\label{eq.comp.low}
\begin{aligned}
  &\quad\frac{1}{2}\frac{d}{dt}Y_j^2+ \Big(1-\frac{\delta}{2}\Big)\mu\Vert\Lambda^{\frac{\alpha}{2}} \dot{\Delta}_j d\Vert^2_{L^2}+\frac{\delta\mu}{2}\Vert \Lambda^{\frac{\alpha}{2}}\dot{\Delta}_j \sigma\Vert^2_{L^2}-\frac{\delta\mu^2}{2\lambda}\int \Lambda^\alpha \dot{\Delta}_j d \cdot \Lambda^{\alpha-1} \dot{\Delta}_j \sigma \dd x \\
  &=\int f_j \Big(\dot{\Delta}_j \sigma-\frac{\delta\mu}{2\lambda}\Lambda^{\alpha-1} \dot{\Delta}_j d\Big) \dd x + \int g_j \Big(\dot{\Delta}_j d - \frac{\delta\mu}{2\lambda}\Lambda^{\alpha-1} \dot{\Delta}_j \sigma\Big) \dd x \\
  &\quad - \frac{1}{2} \int \big((\dot{\Delta}_j \sigma)^2 + (\dot{\Delta}_j d)^2 \big) \Div \dot S_{j-1} v\, \dd x  \\
  &\quad+ \frac{\delta\mu}{2\lambda}\int (\dot S_{j-1} v \cdot \nabla \dot{\Delta}_j d)\cdot \Lambda^{\alpha-1} \dot{\Delta}_j \sigma \dd x
  +  \frac{\delta\mu}{2\lambda}\int \Lambda^{\alpha-1}\big(\dot S_{j-1} v\cdot\nabla \dot{\Delta}_j \sigma\big) \cdot \dot{\Delta}_j d\, \dd x.
\end{aligned}
\end{equation}
For the terms on the right-hand side of equality \eqref{eq.comp.low}, using H\"older's inequality, we get
\begin{equation*}
\begin{aligned}
   \int f_j \Big(\dot{\Delta}_j \sigma-\frac{\delta\mu}{2\lambda}\Lambda^{\alpha-1} \dot{\Delta}_j d\Big)\dd x
  +\int g_j \Big(\dot{\Delta}_j d - \frac{\delta\mu}{2\lambda}\Lambda^{\alpha-1} \dot{\Delta}_j \sigma\Big)\dd x \le C(\Vert f_j\Vert_{L^2}+\Vert g_j\Vert_{L^2})Y_j, &
\end{aligned}
\end{equation*}
\begin{equation*}
\begin{aligned}
  \Big| \int \big((\dot{\Delta}_j \sigma)^2 + (\dot{\Delta}_j d)^2\big) \Div \dot S_{j-1} v \dd x \Big|
  \le C\Vert \nabla \dot S_{j-1} v\Vert_{L^{\infty}}Y_j^2,
\end{aligned}
\end{equation*}
and thanks to the integration by parts and Lemma \ref{corollary.commutator_S_j}, we find
\begin{align}\label{eq.commutator_S_j}
  &\quad \int (\dot S_{j-1} v \cdot \nabla \dot{\Delta}_j d)\cdot \Lambda^{\alpha-1} \dot{\Delta}_j \sigma \dd x
  + \int \Lambda^{\alpha-1} ( \dot S_{j-1} v\cdot\nabla \dot{\Delta}_j \sigma) \cdot \dot{\Delta}_j d\, \dd x \nonumber \\
  &=  - \int_{\R^N}\, (\Div \dot S_{j-1} v) \dot{\Delta}_j d\cdot \Lambda^{\alpha-1} \dot{\Delta}_j \sigma \dd x + \int_{\R^N} [\Lambda^{\alpha-1},\dot S_{j-1}v\cdot\nabla] \dot{\Delta}_j \sigma\, \cdot \dot{\Delta}_j d\, \dd x \nonumber\\
  &\le C2^{j_0(\alpha-1)}\Vert \nabla \dot S_{j-1} v\Vert_{L^{\infty}} \Vert \dot{\Delta}_j d\Vert_{L^2} \Vert \dot{\Delta}_j \sigma\Vert_{L^2}.
\end{align}
For the terms on the left-hand side of \eqref{eq.comp.low}, by virtue of H\"older's and Young's inequalities, we have that for all $j\le j_0$,
\begin{equation*}
\begin{aligned}
  \frac{\delta\mu^2}{2\lambda}\int \Lambda^\alpha \dot{\Delta}_j d \cdot \Lambda^{\alpha-1} \dot{\Delta}_j \sigma \dd x
  &\le \frac{\delta \mu^3}{4\lambda^2} \|\Lambda^{\frac{3\alpha}{2}-1} \dot{\Delta}_j d\|_{L^2}^2  + \frac{\delta \mu}{4} \|\Lambda^{\frac{\alpha}{2}}\dot{\Delta}_j \sigma\|_{L^2}^2  \\
  &\le \frac{\delta\mu^3}{\lambda^2}\big(\frac83\big)^{2(\alpha-1)} 2^{2j_0(\alpha-1)} \Vert\Lambda^{\frac{\alpha}{2}}\dot{\Delta}_j d\Vert^2_{L^2} + \frac{\delta\mu}{4} \Vert\Lambda^{\frac{\alpha}{2}} \dot{\Delta}_j \sigma \Vert^2_{L^2}.
\end{aligned}
\end{equation*}
Choosing $\delta $ small enough (e.g. $\delta =\frac{1}{300}$) and using \eqref{eq:jlow}, we obtain
\begin{align*}
  \Big(1-\frac{\delta}{2}\Big)\mu -\frac{\delta\mu^3}{\lambda^2}\Big(\frac83\Big)^{2(\alpha-1)} 2^{2j_0(\alpha-1)} > (1 - 115\delta)\mu  > \frac{\mu}{2}.
\end{align*}
Thus the equation \eqref{eq.comp.low} can be rewritten as
\begin{equation}\label{eq.Yj.pro.low}
\begin{aligned}
  \frac{d}{dt}Y_j+\bar{\mu} 2^{j\alpha}Y_j\le C\Big(\Vert f_j\Vert_{L^2} + \Vert g_j\Vert_{L^2}+\Vert \nabla v\Vert_{L^{\infty}}Y_j\Big),
\end{aligned}
\end{equation}
where $\bar{\mu} := \frac{\delta\mu}{8}$ and we have used the following estimate
\begin{align*}
  \frac{\mu}{2}\Vert\Lambda^{\frac{\alpha}{2}} \dot{\Delta}_j d\Vert^2_{L^2} + \frac{\delta\mu}{4} \Vert \Lambda^{\frac{\alpha}{2}}\dot{\Delta}_j \sigma\Vert^2_{L^2}
  \geq \frac{\delta \mu}{4} \big(\frac{3}{4}\big)^\alpha 2^{j\alpha} \big( \|\dot\Delta_j d\|_{L^2}^2 + \|\dot\Delta_j \sigma\|_{L^2}^2\big) \geq \bar\mu 2^{j\alpha} Y_j^2 .
\end{align*}

Now we consider the high frequency case $j>j_0$, namely
\begin{equation}\label{eq:jhigh}
  2^{j(\alpha-1)} > \frac{4\lambda}{\mu}.
\end{equation}
We define another energy, still denoted by $Y_j$, as follows
\begin{equation*}
  Y_j^2= \Vert \Lambda^{\alpha-1} \dot{\Delta}_j \sigma\Vert^2_{L^2}
  + 2 \frac{\lambda^2}{\mu^2}\Vert \dot{\Delta}_j d\Vert^2_{L^2}
  - 2\frac{\lambda}{\mu}(\dot{\Delta}_j d | \Lambda^{\alpha-1} \dot{\Delta}_j \sigma),\qquad j>j_0.
\end{equation*}
From \eqref{eq:sigma-Lam} we can gain the following $L^2$-estimate 
\begin{equation*}
\begin{aligned}
  \frac{1}{2}\frac{d}{dt}\Vert \Lambda^{\alpha-1} \dot{\Delta}_j \sigma\Vert^2_{L^2}
  &+\int \Lambda^{\alpha-1} \big(\dot S_{j-1} v\cdot\nabla \dot{\Delta}_j \sigma\big)\cdot \Lambda^{\alpha-1} \dot{\Delta}_j \sigma \dd x \\
  &+ \lambda\int \Lambda^\alpha \dot{\Delta}_j d \cdot \Lambda^{\alpha-1} \dot{\Delta}_j \sigma \dd x
  =\int \Lambda^{\alpha-1} f_j \cdot \Lambda^{\alpha-1} \dot{\Delta}_j \sigma \dd x.
\end{aligned}
\end{equation*}
Noting that $2\frac{\lambda}{\mu}(\dot{\Delta}_j d | \Lambda^{\alpha-1} \dot{\Delta}_j \sigma) \leq \frac{3}{2} \frac{\lambda^2}{\mu^2}\|\dot \Delta_j d\|_{L^2}^2 + \frac{2}{3}\|\Lambda^{\alpha-1}\dot\Delta_j \sigma\|_{L^2}^2 $, we find that
\begin{align}\label{eq:Yj-high-equ}
  Y_j \approx \Vert \Lambda^{\alpha-1} \dot{\Delta}_j \sigma\Vert_{L^2} + \frac{\lambda^2}{\mu^2}\Vert \dot{\Delta}_j d\Vert_{L^2}
\end{align}
and $Y_j$ satisfies
\begin{equation*}
\begin{aligned}
  &\quad\frac{1}{2}\frac{d}{dt} Y_j^2
  + \frac{\lambda^2}{\mu}\Vert\Lambda^{\frac{\alpha}{2}} \dot{\Delta}_j d\Vert^2_{L^2}
  + \frac{\lambda^2}{\mu}\Vert \Lambda^{\frac{\alpha}{2}} \dot{\Delta}_j \sigma\Vert^2_{L^2}
  - \frac{2\lambda^3}{\mu^2} \int \dot{\Delta}_j \sigma\, \Lambda \dot{\Delta}_j d \,\dd x \\
  &= \int \Lambda^{\alpha-1} f_j \Big( \Lambda^{\alpha-1} \dot{\Delta}_j \sigma -\frac{\lambda}{\mu} \dot{\Delta}_j d \Big)\,\dd x
  + \frac{\lambda}{\mu^2} \int g_j ( 2\lambda \dot{\Delta}_j d - \mu\Lambda^{\alpha-1} \dot{\Delta}_j \sigma)\,\dd x \\
  &\quad+ \frac{\lambda}{\mu}  \int (\dot S_{j-1} v \cdot \nabla \dot{\Delta}_j d)\cdot \Lambda^{\alpha-1} \dot{\Delta}_j \sigma\, \dd x
  + \frac{\lambda}{\mu} \int \Lambda^{\alpha-1} \big(\dot S_{j-1} v\cdot\nabla \dot{\Delta}_j \sigma\big) \cdot \dot{\Delta}_j d\, \dd x \\
  &\quad- \int \Lambda^{\alpha-1} (\dot S_{j-1} v\cdot \nabla \dot{\Delta}_j \sigma)\cdot \Lambda^{\alpha-1} \dot{\Delta}_j \sigma \,\dd x
  + \frac{\lambda^2}{\mu^2} \int (\dot{\Delta}_j d)^2 \Div \dot S_{j-1} v\, \dd x .
\end{aligned}
\end{equation*}
For the left-hand side of the above equality, applying \eqref{eq:jhigh} we have
\begin{equation*}
  \Vert\Lambda^{\frac{\alpha}{2}}\dot{\Delta}_j d\Vert_{L^2}^2
  \ge (\frac{3}{4})^{2\alpha-2} 2^{j(2\alpha-2)}\Vert\Lambda^{1-\frac{\alpha}{2}}\dot{\Delta}_j d\Vert^2_{L^2}
  \ge 8 \frac{\lambda^2}{\mu^2} \Vert\Lambda^{1-\frac{\alpha}{2}}\dot{\Delta}_j d\Vert^2_{L^2},
\end{equation*}
and
\begin{equation*}
\begin{aligned}
  \frac{2\lambda^3}{\mu^2} \int \dot{\Delta}_j \sigma \,\Lambda \dot{\Delta}_j d \dd x
  &\le \frac{\lambda^2}{2\mu}\Vert \Lambda^{\frac{\alpha}{2}} \dot{\Delta}_j \sigma\Vert_{L^2}^2
  +\frac{2\lambda^4}{\mu^3} \Vert \Lambda^{1-\frac{\alpha}{2}}\dot{\Delta}_j d \Vert_{L^2}^2 \\
  & \leq \frac{\lambda^2}{2\mu}\Vert \Lambda^{\frac{\alpha}{2}} \dot{\Delta}_j \sigma\Vert_{L^2}^2
  + \frac{\lambda^2}{2\mu} \|\Lambda^{\frac{\alpha}{2}} \dot{\Delta}_j d \|_{L^2}^2,
\end{aligned}
\end{equation*}
and
\begin{align*}
  \frac{\lambda^2}{2\mu} \big(\Vert \Lambda^{\frac{\alpha}{2}} \dot{\Delta}_j \sigma\Vert_{L^2}^2
  + \|\Lambda^{\frac{\alpha}{2}} \dot{\Delta}_j d \|_{L^2}^2\big)
  & \geq \frac{\lambda^2}{2\mu} \Big(\Vert \Lambda^{\frac{\alpha}{2}} \dot{\Delta}_j \sigma\Vert_{L^2}^2
  + 8 \frac{\lambda^2}{\mu^2} \|\Lambda^{1-\frac{\alpha}{2}} \dot{\Delta}_j d \|_{L^2}^2 \Big) \\
  & \geq \frac{\lambda^2}{4 \mu} 2^{j(2-\alpha)} \Big( \|\Lambda^{\alpha -1} \dot{\Delta}_j \sigma \|_{L^2}^2
  + \frac{8\lambda^2}{\mu^2} \|\dot{\Delta}_j d\|_{L^2}^2 \Big).
\end{align*}
On the other hand, for the terms on the right-hand side, owing to H\"older's and Bernstein's inequalities, we infer that
\begin{equation*}
  \int \Lambda^{\alpha-1} f_j \Big( \Lambda^{\alpha-1} \dot{\Delta}_j \sigma -\frac{\lambda}{\mu} \dot{\Delta}_j d \Big)\,\dd x
  \le C\Vert \Lambda^{\alpha-1}f_j \Vert_{L^2}Y_j,
\end{equation*}
\begin{equation*}
\begin{aligned}
  \frac{\lambda}{\mu^2} \int g_j \big( 2\lambda \dot{\Delta}_j d - \mu\Lambda^{\alpha-1} \dot{\Delta}_j \sigma\big)\,\dd x \le C\Vert g_j\Vert_{L^2}Y_j,
\end{aligned}
\end{equation*}
and
\begin{equation*}
  \frac{\lambda^2}{\mu^2} \int (\dot{\Delta}_j d)^2 \Div \dot S_{j-1} v\, \dd x
  \le C\Vert \Div \dot S_{j-1} v\Vert_{L^\infty}\Vert \dot{\Delta}_j d\Vert_{L^2}^2,
\end{equation*}
and using the integration by parts and Corollary \ref{corollary.commutator_S_j},
\begin{align*}
  &\quad \int \Lambda^{\alpha-1} (\dot S_{j-1} v\cdot \nabla \dot{\Delta}_j \sigma)\cdot \Lambda^{\alpha-1} \dot{\Delta}_j \sigma \, \dd x
  \\&=\int  (\dot S_{j-1} v\cdot \nabla \Lambda^{\alpha-1}\dot{\Delta}_j \sigma)\, \Lambda^{\alpha-1} \dot{\Delta}_j \sigma \, \dd x
  + \int ([\Lambda^{\alpha-1}, \dot S_{j-1} v\cdot \nabla] \dot{\Delta}_j \sigma)\, \Lambda^{\alpha-1} \dot{\Delta}_j \sigma \, \dd x \\
  &= - \frac{1}{2}\int (\Div \dot S_{j-1} v )\cdot (\Lambda^{\alpha-1}\dot{\Delta}_j \sigma)^2 \, \dd x
  + \int ([\Lambda^{\alpha-1}, \dot S_{j-1} v\cdot \nabla] \dot{\Delta}_j \sigma)\, \Lambda^{\alpha-1} \dot{\Delta}_j \sigma \, \dd x \\
  &\le C 2^{j(\alpha-1)}\Vert \nabla \dot S_{j-1} v\Vert_{L^{\infty}} \Vert \Lambda^{\alpha-1}\dot{\Delta}_j \sigma\Vert_{L^2} \Vert \dot{\Delta}_j \sigma\Vert_{L^2}.
\end{align*}
Similarly as obtaining inequality \eqref{eq.commutator_S_j}, we also see that
\begin{equation*}
\begin{aligned}
  \frac{\lambda}{\mu}  \int (\dot S_{j-1} v \cdot \nabla \dot{\Delta}_j d)\cdot \Lambda^{\alpha-1} \dot{\Delta}_j \sigma \dd x
  + \frac{\lambda}{\mu} \int \Lambda^{\alpha-1} \big(\dot S_{j-1} v\cdot\nabla \dot{\Delta}_j \sigma\big) \cdot \dot{\Delta}_j d\, \dd x &\\
  \le C2^{j(\alpha-1)}\Vert \nabla \dot S_{j-1} v\Vert_{L^{\infty}} \Vert \dot{\Delta}_j d\Vert_{L^2} \Vert \dot{\Delta}_j \sigma\Vert_{L^2}.&
\end{aligned}
\end{equation*}
Hence, by letting $\bar{\nu} := \frac{\lambda^2}{4\mu}$ and gathering the above estimates, we have the following inequality
\begin{equation}\label{eq.Yj.pro.high}
  \frac{d}{dt}Y_j+\bar{\nu}2^{j(2-{\alpha})}Y_j\le C 2^{j(\alpha-1)} \Vert f_j \Vert_{L^2}
  + C \Vert g_j\Vert_{L^2}+ C\Vert \nabla v\Vert_{L^\infty}Y_j.
\end{equation}

Combining the estimates on both low and high frequencies \eqref{eq.Yj.pro.low} and \eqref{eq.Yj.pro.high}, we obtain
\begin{equation*}
  \frac{d}{dt}{Y_j}+ \min\{\bar{\mu} 2^{j\alpha},\bar{\nu}2^{j(2-\alpha)}\}Y_j
  \le
  C \max\{1,2^{j(\alpha-1)}\}\Vert f_j \Vert_{L^2}
  + C \Vert g_j\Vert_{L^2}+ C\Vert \nabla v\Vert_{L^\infty}Y_j.
\end{equation*}
Integrating this with respect to the time variable and using \eqref{eq:Yj-low-equ}, \eqref{eq:Yj-high-equ}, we obtain
\begin{equation}\label{eq.priori.comp.Xj}
\begin{aligned}
  & \quad Y_j(t)+\min\{\bar{\mu} 2^{j\alpha},\bar{\nu}2^j\}\int_0^t \|\dot{\Delta}_j \sigma\|_{L^2} \dd \tau + \min\{\bar{\mu} 2^{j\alpha},\bar{\nu}2^{j(2-\alpha)}\} \int_0^t \|\dot{\Delta}_j d\|_{L^2} \dd \tau \\
  & \le Y_j(0)+C\int_0^t \Big(\max\{1,2^{j(\alpha-1)}\}\Vert f_j \Vert_{L^2} +\Vert g_j\Vert_{L^2}+\Vert \nabla v\Vert_{L^\infty}Y_j\Big) \dd \tau.
\end{aligned}
\end{equation}

In the above estimate, the smoothing effect of $d$ in the high-frequency can be improved. 
Indeed, taking the $L^2$-inner product with $d$ in the second equation of \eqref{eq.lineareqibn0}, we deduce that
\begin{equation}\label{eq.dj.pro}
  \frac{d}{dt}\Vert \dot{\Delta}_j d\Vert_{L^2}
  + (\frac{3}{4})^\alpha \mu2^{j\alpha}\Vert\dot{\Delta}_j d\Vert_{L^2} \le \frac{8}{3} \lambda 2^j\Vert  \dot{\Delta}_j \sigma\Vert_{L^2}
  + \Vert g_j\Vert_{L^2}+\Vert \nabla v\Vert_{L^{\infty}}\Vert \dot{\Delta}_j d\Vert_{L^2},
\end{equation}
and integrating on the time variable leads to
\begin{equation*}
  \mu 2^{j\alpha}\int_0^t \Vert\dot{\Delta}_j d(\tau)\Vert_{L^2}\dd \tau - \lambda 2^j\int_0^t \Vert \dot{\Delta}_j \sigma(\tau)\Vert_{L^2}\dd \tau
  \le CY_j(0) + C\int_0^t \Big(\Vert g_j\Vert_{L^2} +\Vert\nabla v \Vert_{L^\infty}Y_j\Big) \dd \tau.
\end{equation*}
Inserting the above inequality into \eqref{eq.priori.comp.Xj} yields
\begin{equation}\label{eq.lin.comp.pro}
\begin{aligned}
  & \quad Y_j +\min\{\bar{\mu} 2^{j\alpha},\bar{\nu}2^j \}\int_0^t \Vert \dot{\Delta}_j \sigma\Vert_{L^2}\dd \tau
  +2^{j\alpha}\int_0^t \Vert\dot{\Delta}_j d\Vert_{L^2}\dd \tau \\
  & \leq C Y_j(0) + C \int_0^t \Big(\max\{1,2^{j(\alpha-1)}\}\Vert f_j \Vert_{L^2}
  +\Vert g_j\Vert_{L^2}+\Vert \nabla v\Vert_{L^\infty}Y_j \Big) \dd \tau.
\end{aligned}
\end{equation}

\textit{\bf Step 2: the incompressible part.}
By taking the $L^2$-estimate of equation $\eqref{eq.lineareqibn0}_3$, the incompressible part satisfies
\begin{equation*}%
\begin{aligned}
  \frac{1}{2}\frac{d}{dt}\Vert \dot\Delta_j\mathbb{P}u\Vert_{L^2}^2 + \mu \Vert \Lambda^{\frac{\alpha}{2}}\dot\Delta_j\mathbb{P}u\Vert_{L^2}^2
  \le C \Vert \dot\Delta_j \mathbb{P} u\Vert_{L^2} \big(\Vert \Div \dot S_{j-1} v\Vert_{L^{\infty}}\Vert \dot\Delta_j\mathbb{P} u \Vert_{L^2}
  + \Vert \widetilde{g}_j\Vert_{L^2}\big),
\end{aligned}
\end{equation*}
which implies that
\begin{align}\label{eq.Puj.pro}
  \frac{d}{dt}\Vert \dot\Delta_j\mathbb{P}u\Vert_{L^2} + \mu (\frac{3}{4})^\alpha 2^{j\alpha}\Vert \dot\Delta_j\mathbb{P}u\Vert_{L^2}^2
  \le C \big(\Vert \Div \dot S_{j-1} v\Vert_{L^{\infty}}\Vert \dot\Delta_j\mathbb{P} u \Vert_{L^2} + \Vert \widetilde{g}_j\Vert_{L^2}\big).
\end{align}
Integrating with respect to the time variable, we have
\begin{equation}\label{eq.lin.incomp.pro}
\begin{aligned}
  \Vert \dot\Delta_j \mathbb{P}u(t)\Vert_{L^2} + \mu 2^{j\alpha} \int_0^t\Vert \dot \Delta_j\mathbb{P} u\Vert_{L^2}\dd \tau
  \le 2 \Vert \dot \Delta_j\mathbb{P}u_0\Vert_{L^2} + C \int_0^t \Big(\Vert \widetilde{g}_j\Vert_{L^2} + \Vert \nabla v\Vert_{L^{\infty}} \Vert \dot\Delta_j\mathbb{P}u\Vert_{L^2} \Big)\dd \tau.
\end{aligned}
\end{equation}

\textit{\bf Step 3: the \textit{a priori} estimate for $(\sigma,u)$.}
We need to combine the compressible and incompressible estimates to show the \textit{a priori} estimate of $(\sigma,u)$.
Multiplying $2^{js}$ on both sides of \eqref{eq.lin.comp.pro} and \eqref{eq.lin.incomp.pro}, taking the $\ell^1$-norm with regard to $j$,
and noting that $\|u\|_{\dot{B}^s_{2,1}} \leq \|d\|_{\dot{B}^s_{2,1}} + \|\mathbb{P} u\|_{\dot{B}^s_{2,1}}$, we obtain
\begin{equation}\label{eq:sig-u-es}
\begin{aligned}
  &\quad\Vert \sigma\Vert_{\widetilde{L}_t^{\infty}(\widetilde{B}^{s,s+\alpha-1})} + \Vert u\Vert_{\widetilde{L}_t^{\infty}(\dot{B}^s_{2,1})}
  +\int_0^t \Big(\Vert \sigma(\tau)\Vert_{\widetilde{B}^{s+\alpha,s+1}} + \Vert u(\tau)\Vert_{\dot{B}^{s+\alpha}_{2,1}}\Big) \dd \tau\\
  &\lesssim \Vert \sigma_0\Vert_{\widetilde{B}^{s,s+\alpha-1}}+\Vert u_0\Vert_{\dot{B}^s_{2,1}}
  +\int_0^t \Vert \nabla v(\tau)\Vert_{L^\infty} \big(\Vert \sigma(\tau)\Vert_{\widetilde{B}^{s,s+\alpha-1}} + \Vert u(\tau)\Vert_{\dot{B}^s_{2,1}} \big) \dd \tau \\
  &\quad+ \int_0^t \Big(\sum_{j\in\mathbb{Z}}\max\{2^{js},2^{j(s+\alpha-1)}\}\Vert f_j \Vert_{L^2}
  + \sum_{j\in\mathbb{Z}}2^{js} \big(\Vert g_j\Vert_{L^2} + \Vert \widetilde{g}_j\Vert_{L^2}\big)\Big) \dd \tau.
\end{aligned}
\end{equation}
Recalling that $f_j$, $g_j$ and $\widetilde{g}_j$ are given by \eqref{eq:fj}--\eqref{eq:tild-gj}, and by virtue of Lemma \ref{lemma.Remark2.103}, we see that
\begin{equation}\label{eq.fi}
  \sum_{j\in\mathbb{Z}}\max\{2^{js},2^{j(s+\alpha-1)}\}\Vert f_j \Vert_{L^2} \le C \Vert F\Vert_{\widetilde{B}^{s,s+\alpha-1}}+C\Vert \nabla v\Vert_{\dot{B}^{\frac{N}{2}}_{2,1}}\Vert \sigma\Vert_{\widetilde{B}^{s,s+\alpha-1}},
\end{equation}
and
\begin{equation}\label{eq.gi}
  \sum_{j\in\mathbb{Z}}2^{js}(\Vert g_j\Vert_{L^2}+\Vert \widetilde{g}_j\Vert_{L^2})\le C \Vert G \Vert_{\dot{B}^s_{2,1}} + C\Vert \nabla v\Vert_{\dot{B}^{\frac{N}{2}}_{2,1}}\Vert u\Vert_{\dot{B}^s_{2,1}}.
\end{equation}
Inserting \eqref{eq.fi}-\eqref{eq.gi} into \eqref{eq:sig-u-es} yields
\begin{equation}\label{eq:sig-u-es2}
\begin{aligned}
  &\quad\Vert \sigma\Vert_{\widetilde{L}_t^{\infty}(\widetilde{B}^{s,s+\alpha-1})} + \Vert u\Vert_{\widetilde{L}_t^{\infty}(\dot{B}^s_{2,1})}
  +\int_0^t \Big(\Vert \sigma(\tau)\Vert_{\widetilde{B}^{s+\alpha,s+1}} + \Vert u(\tau)\Vert_{\dot{B}^{s+\alpha}_{2,1}}\Big) \dd \tau \\
  &\lesssim \Vert \sigma_0\Vert_{\widetilde{B}^{s,s+\alpha-1}}+\Vert u_0\Vert_{\dot{B}^s_{2,1}}
  +\int_0^t \Vert \nabla v(\tau)\Vert_{\dot{B}^{\frac{N}{2}}_{2,1}}  \big(\Vert \sigma(\tau)\Vert_{\widetilde{B}^{s,s+\alpha-1}} + \Vert u(\tau)\Vert_{\dot{B}^s_{2,1}} \big) \dd \tau  \\
  &\quad+ \Vert F\Vert_{L^1_t(\widetilde{B}^{s,s+\alpha-1})} +  \Vert G \Vert_{L^1_t(\dot{B}^s_{2,1})}.
\end{aligned}
\end{equation}
With the help of Gronwall's inequality, we conclude the proof of \eqref{eq.lin.pro}.
\end{proof}

\subsection{A priori estimates for the nonlinear system}\label{subsec:apri-nl}
Let us turn our discussion to the nonlinear system \eqref{eq.EAsigma.nl}. It can be viewed as \eqref{eq.lineareqbn0} with $v=u$ and
\begin{equation}\label{eq:FG}
\begin{aligned}
  &F:=-(\gamma-1)\sigma\Div u-T_{\nabla\sigma}\cdot u-R(u,{\nabla\sigma}),\\
  &G := -\mu\Lambda^\alpha (uh(\sigma)) + \mu\, u\Lambda^\alpha h(\sigma)-T_{\nabla {u}}\cdot u-R(u,\nabla {u}).
\end{aligned}
\end{equation}
Define
\begin{equation}\label{eq:XT}
  X(T) :=\Vert \sigma\Vert_{\widetilde{L}^{\infty}_{T}(\widetilde{B}^{\frac{N}{2}+1-\alpha,\frac{N}{2}})}
  + \Vert u\Vert_{\widetilde{L}^{\infty}_{T}(\dot{B}^{\frac{N}{2}+1-\alpha}_{2,1})}
  + \Vert \sigma\Vert_{L^{1}_{T}(\widetilde{B}^{\frac{N}{2}+1,\frac{N}{2}+2-\alpha})}+\Vert u\Vert_{L^{1}_{T}(\dot{B}^{\frac{N}{2}+1}_{2,1})}
\end{equation}
and
\begin{equation}\label{eq:X0}
  X_0 :=\Vert \sigma_0\Vert_{\widetilde{B}^{\frac{N}{2}+1-\alpha,\frac{N}{2}}} +\Vert u_0\Vert_{\dot{B}^{\frac{N}{2}+1-\alpha}_{2,1}}.
\end{equation}

We state the following \emph{a priori} estimate on $X(T)$, assuming $X_0$ is sufficiently small. The uniform bound on $X(T)$ will play an important role in the global well-posedness theory for the system \eqref{eq.EAsigma.nl}.

\begin{proposition}\label{prop.nonlin.pro.nl}
Assume that $(\sigma,u)$ is a smooth solution for the system \eqref{eq.EAsigma.nl} with $\sigma_0 \in \widetilde{B}^{\frac{N}{2}+1-\alpha,\frac{N}{2}}$ and $u_0 \in \dot{B}^{\frac{N}{2}+1-\alpha}_{2,1}$. Then there exist constants $\varepsilon_0=\varepsilon_0(\alpha,N)>0$ and  $C_*= C_*(\alpha, N) > 0$ such that, if $X_0\leq \varepsilon_0$ then
\begin{equation}\label{es:XT-glob}
  X(T)\le C_* \,X_0 ,\qquad \forall~T\ge0.
\end{equation}
\end{proposition}

\begin{proof}[Proof of Proposition \ref{prop.nonlin.pro.nl}]
Apply Proposition \ref{prop.lin.pro} with $s =\frac{N}{2}+1-\alpha$ and get
\[
X(T)\leq C e^{C X(T)} \Big(X_0+ \Vert F\Vert_{L^1_T(\widetilde{B}^{\frac{N}{2}+1-\alpha,\frac{N}{2}})} + \Vert G\Vert_{ L^1_T(\dot{B}^{\frac{N}{2}+1-\alpha}_{2,1})} \Big).
\]

Now, we estimate the source terms $F$ and $G$. It follows from Lemma \ref{lemma.Bilinear} that
\begin{align*}
   &\Vert\sigma\Div u\Vert_{L^1_T(\widetilde{B}^{\frac{N}{2}+1-\alpha,\frac{N}{2}})}
  + \Vert T_{\nabla\sigma}\cdot u\Vert_{L^1_T(\widetilde{B}^{\frac{N}{2}+1-\alpha,\frac{N}{2}})}
  + \Vert R(u,{\nabla\sigma})\Vert_{L^1_T(\widetilde{B}^{\frac{N}{2}+1-\alpha,\frac{N}{2}})}\\
  & \leq C \Vert \sigma\Vert_{\widetilde{L}^\infty_T(\widetilde{B}^{\frac{N}{2}+1-\alpha,\frac{N}{2}})}
  \Vert u\Vert_{L^1_T(\dot{B}^{\frac{N}{2}+1}_{2,1})} \le C X^2(T)
\end{align*}
and
\[
   \Vert T_{\nabla u}\cdot u\Vert_{L^1_T(\dot{B}^{\frac{N}{2}+1-\alpha}_{2,1})}
  + \Vert R(u,\nabla u)\Vert_{L^1_T(\dot{B}^{\frac{N}{2}+1-\alpha}_{2,1})}  \le C \Vert u\Vert_{\widetilde{L}^\infty_T(\dot{B}^{\frac{N}{2}+1-\alpha}_{2,1})}
  \Vert u \Vert_{L^1_T(\dot{B}^{\frac{N}{2}+1}_{2,1})}
  \le CX^2(T).
\]
Taking advantage of Corollary \ref{corollary.kato_ponce} and Lemma \ref{lemma.composition}, we obtain
\begin{equation}\label{eq.local.G2}
\begin{split}
  \Vert \Lambda^\alpha (uh(\sigma)) - u \Lambda^\alpha h(\sigma)\Vert_{L^1_T(\dot{B}^{\frac{N}{2}+1-\alpha}_{2,1})} & \le C\Vert u\Vert_{L^1_T(\dot{B}^{\frac{N}{2}+1}_{2,1})}
  \Vert h(\sigma) \Vert_{\widetilde{L}^\infty_T(\dot{B}^{\frac{N}{2}}_{2,1})} \\
  & \leq C \|u \|_{L^1_T(\dot B^{\frac{N}{2}+1}_{2,1})} \|\sigma\|_{\widetilde{L}^\infty_T (\dot B^{\frac{N}{2}}_{2,1})}
  \le C X^2(T).
\end{split}
\end{equation}
Note that when applying \eqref{eq:compest1} in Lemma \ref{lemma.composition}, the constant $C$ depends on $\|\sigma\|_{L^\infty}$. Moreover, when $\gamma>1$, $h$ is defined in $I = [-\frac{\sqrt{\kappa\gamma}}{\gamma-1},+\infty)$, and it is not necessarily smooth at the endpoint $\sigma=-\frac{\sqrt{\kappa\gamma}}{\gamma-1}$ (The endpoint corresponds to vacuum $\rho=0$, which needs to be avoided). Therefore, we may assume e.g. 
\begin{equation}\label{eq:sigmasmall}
\|\sigma\|_{\widetilde{L}^\infty_T(L^\infty)} \leq \tfrac{\sqrt{\kappa\gamma}}{2(\gamma-1)}
\end{equation}
to make sure $C$ is a universal constant. \eqref{eq:sigmasmall} can be enforced by an appropriately chosen smallness condition.

Combining the above estimates with H\"older's inequality yields
\begin{equation}\label{eq:FG-es}
  \Vert F\Vert_{L^{1}_T(\widetilde{B}^{\frac{N}{2}+1-\alpha,\frac{N}{2}})} + \Vert G\Vert_{L^1_T(\dot{B}^{\frac{N}{2}+1-\alpha}_{2,1})} \le C X^2(T).
\end{equation}
Hence, collecting the above estimates we deduce that
\begin{equation}\label{eq:XT}
  X(T)\le Ce^{CX(T)}\big(X_0 + X^2(T)\big),
\end{equation}
where the constant $C=C(\alpha,N)$. Pick $C_*=\max\{4C, 1\}$ and $\varepsilon_0=C_*^{-2}$. We now show \eqref{es:XT-glob} by contradiction. Suppose \eqref{es:XT-glob} is false. By continuity of $X$, there exists a $T>0$ such that $X(T)=C_*X_0$. Then \eqref{eq:XT} implies
\[
X(T)\leq C e^{CC_*X_0}(X_0+C_*^2X_0^2)\leq Ce^{1/4}\cdot 2X_0<
C_*X_0.\]
This leads to a contradiction.

We can further choose a smaller $\varepsilon_0$ to get a smaller $X(T)$ when needed. In particular, since $\widetilde{B}^{\frac{N}{2}+1-\alpha,\frac{N}{2}}\hookrightarrow L^\infty$, we can choose $\varepsilon$ small enough to guarantee \eqref{eq:sigmasmall}.
\end{proof}

Next, we present an improved \emph{a priori} estimate on $X(T)$ for some positive time $T$, without a smallness assumption on $u_0$. This allows us to obtain a stronger local well-posedness result (see Theorem \ref{th.local.solution.nl}). This method has been used on the barotropic compressible Navier-Stokes system, see e.g. \cite[Corollary 10.4]{bahouri2011fourier}.

\begin{proposition}\label{prop.priori.local}
Let $1<\alpha<2$, $\sigma_0 \in \widetilde{B}^{\frac{N}{2}+1-\alpha,\frac{N}{2}}$ and $u_0 \in \dot{B}^{\frac{N}{2}+1-\alpha}_{2,1}$. Assume that $(\sigma,u)$ is a smooth solution for the system \eqref{eq.EAsigma.nl}.
Then there exist a positive time $T$ and a small enough constant $\eta >0$, such that, under the assumption
\begin{equation}\label{assum:sig0}
  \Vert \sigma_0\Vert_{\widetilde{B}^{\frac{N}{2}+1-\alpha,\frac{N}{2}}}\le \eta,
\end{equation}
we have
\begin{align}\label{eq:sig-loc}
  \Vert \sigma\Vert_{\widetilde{L}^{\infty}_{T}(\widetilde{B}^{\frac{N}{2}+1-\alpha,\frac{N}{2}})}
  + \Vert \sigma\Vert_{L^1_T(\widetilde{B}^{\frac{N}{2}+1,\frac{N}{2}+2-\alpha})} \leq C \Vert \sigma_0\Vert_{\widetilde{B}^{\frac{N}{2}+1-\alpha,\frac{N}{2}}},
\end{align}
and
\begin{equation}\label{eq:u-loc}
\begin{aligned}
  X(T) \le C X_0,
\end{aligned}
\end{equation}
with $C = C(\alpha,N)>0$.
\end{proposition}

\begin{proof}[Proof of Proposition \ref{prop.priori.local}]
We split $u$ into  $u_L+\widetilde{u}$, with $u_L$ satisfying
\begin{equation}\label{eq.fractional.heat}
  \partial_t u_L +\mu\Lambda^\alpha u_L=0, \qquad u_L|_{t=0} = u_0.
\end{equation}
It is obvious that $u_L = e^{- \mu t\Lambda^{\alpha}}u_0$, where $e^{ - \mu t \Lambda^\alpha} = \mathcal{F}^{-1}(e^{-\mu t |\xi|^\alpha})*$ is the fractional heat semigroup operator.
By using Bernstein's inequality and \cite[Proposition 2.2]{hmidi2007global}, we have
\begin{align}\label{eq:UL}
  U_L(t) := \int_0^{t}\Vert u_L(\tau)\Vert_{\dot{B}^{\frac{N}{2}+1}_{2,1}} \dd \tau \le C\sum_{j\in\mathbb{Z}}2^{j(\frac{N}{2}+1-\alpha)}(1-e^{-\mu t2^{j\alpha}})\Vert\dot{\Delta}_ju_0\Vert_{L^2},
\end{align}
and
\begin{align}\label{eq:UL2}
  \Vert u_L \Vert_{\widetilde{L}^\infty(\R^+;\dot{B}^{\frac{N}{2}+1-\alpha}_{2,1})} \le \Vert u_0\Vert_{\dot{B}^{\frac{N}{2}+1-\alpha}_{2,1}}.
\end{align}
Observe that $(\sigma,\widetilde{u})$ satisfies
\begin{equation}
\begin{cases}
  \partial_t {\sigma}+T_u \cdot\nabla\sigma +\lambda\Div{\widetilde{u}} = -\lambda\Div{u_L} - (\gamma-1)\sigma\Div u-T_{\nabla\sigma}\cdot u-R(u,{\nabla\sigma})
  =: \widetilde{F},\\
  \partial_t \widetilde{u}+T_{u}\cdot\nabla \widetilde{u}+\mu\Lambda^\alpha \widetilde{u} + \lambda\nabla {\sigma}
  =\\
  \qquad\qquad\qquad-u\cdot\nabla u_L-\mu \big(\Lambda^\alpha (uh(\sigma)) - u \Lambda^\alpha h(\sigma)\big) -T_{\nabla \widetilde{u}}\cdot u-R(u,\nabla \widetilde{u})=:\widetilde{G}, \\
  ({\sigma}, \widetilde{u})_{t=0} = (\sigma_0,0),
\end{cases}
\end{equation}
with $u= u_L + \widetilde{u}$.
Denote by
\begin{equation}
  \widetilde{X}(T) :=\Vert \sigma\Vert_{\widetilde{L}^{\infty}_{T}(\widetilde{B}^{\frac{N}{2}+1-\alpha,\frac{N}{2}})}
  +\Vert \widetilde{u}\Vert_{\widetilde{L}^{\infty}_{T}(\dot{B}^{\frac{N}{2}+1-\alpha}_{2,1})}
  +\Vert \sigma\Vert_{L^1_{T}(\widetilde{B}^{\frac{N}{2}+1,\frac{N}{2}+2-\alpha})} + \Vert \widetilde{u}\Vert_{L^1_T (\dot{B}^{\frac{N}{2}+1}_{2,1})}.
\end{equation}
According to \eqref{eq:sig-u-es2} of Proposition \ref{prop.lin.pro} with $s= \frac{N}{2} + 1-\alpha$, we infer that
\begin{equation}\label{eq.local.X}
  \widetilde{X}(t) \lesssim \Vert \sigma_0\Vert_{\widetilde{B}^{\frac{N}{2}+1-\alpha,\frac{N}{2}}} + \int_0^t \Vert \nabla u(\tau)\Vert_{\dot{B}^{\frac{N}{2}}_{2,1}} \widetilde{X}(\tau) \dd \tau
  +\Vert \widetilde{F}\Vert_{L^1_t(\widetilde{B}^{\frac{N}{2}+1-\alpha,\frac{N}{2}})} + \Vert \widetilde{G} \Vert_{L^1_t(\dot{B}^{\frac{N}{2}+1-\alpha}_{2,1})}.
\end{equation}
H\"older's inequality implies
\begin{equation}\label{eq.local.R1}
  \int_0^t \Vert \nabla u(\tau)\Vert_{\dot{B}^{\frac{N}{2}}_{2,1}} \widetilde{X}(\tau) \dd \tau \le U_L(t)\widetilde{X}(t) + \widetilde{X}^2(t) \le U_L^2(t)+ 2 \widetilde{X}^2(t).
\end{equation}
Now we calculate the terms involving $\widetilde{F}$ and $\widetilde{G}$. By virtue of H\"older's inequality and the interpolation, we get
\begin{align*}
  \Vert \Div u_L\Vert_{L^1_t(\widetilde{B}^{\frac{N}{2}+1-\alpha,\frac{N}{2}})}&\le \Vert u_L\Vert_{L^1_t(\dot{B}^{\frac{N}{2}+2-\alpha}_{2,1})} +  \Vert u_L \Vert_{L^1_t(\dot{B}^{\frac{N}{2}+1}_{2,1})} \nonumber \\
  & \le \Vert u_L \Vert_{L^1_t (\dot{B}^{\frac{N}{2}+1-\alpha}_{2,1} \cap \dot{B}^{\frac{N}{2}+1}_{2,1})} + \Vert u_L\Vert_{L^1_t(\dot{B}^{\frac{N}{2}+1}_{2,1})} \nonumber \\
  & \le t \Vert u_L\Vert_{\widetilde{L}^\infty_t(\dot{B}^{\frac{N}{2}+1-\alpha}_{2,1})} + 2\Vert u_L\Vert_{L^1_t(\dot{B}^{\frac{N}{2}+1}_{2,1})}\le t\Vert u_0\Vert_{\dot{B}^{\frac{N}{2}+1-\alpha}_{2,1}}+2U_L(t).
\end{align*}
It follows from Lemma \ref{lemma.Bilinear} that
\begin{equation*}
\begin{aligned}
  & \Vert\sigma\Div u\Vert_{L^1_t(\widetilde{B}^{\frac{N}{2}+1-\alpha,\frac{N}{2}})}
  + \Vert T_{\nabla\sigma}\cdot u\Vert_{L^1_t(\widetilde{B}^{\frac{N}{2}+1-\alpha,\frac{N}{2}})}
  + \Vert R(u,{\nabla\sigma})\Vert_{L^1_t(\widetilde{B}^{\frac{N}{2}+1-\alpha,\frac{N}{2}})} \\
  & \leq C \Vert \sigma\Vert_{\widetilde{L}^\infty_t(\widetilde{B}^{\frac{N}{2}+1-\alpha,\frac{N}{2}})}
  \Vert u\Vert_{L^1_t(\dot{B}^{\frac{N}{2}+1}_{2,1})} \le C \big(U_L^2(t) + \widetilde{X}^2(t) \big),
\end{aligned}
\end{equation*}
and
\begin{equation*}
  \Vert u\cdot\nabla u_L\Vert_{L^1_t(\dot{B}^{\frac{N}{2}+1-\alpha}_{2,1})}\le C \Vert u\Vert_{\widetilde{L}^\infty_t(\dot{B}^{\frac{N}{2}+1-\alpha}_{2,1})}\Vert u_L\Vert_{L^1_t(\dot{B}^{\frac{N}{2}+1}_{2,1})}
  \le C\big(U_L^2(t)+\widetilde{X}^2(t) \big),
\end{equation*}
and
\begin{equation*}
\begin{aligned}
  & \quad \Vert T_{\nabla \widetilde{u}}\cdot u\Vert_{L^1_t(\dot{B}^{\frac{N}{2}+1-\alpha}_{2,1})}
  + \Vert R(u,\nabla \widetilde{u})\Vert_{L^1_t(\dot{B}^{\frac{N}{2}+1-\alpha}_{2,1})} \\
  & \le C \Vert \widetilde{u}\Vert_{\widetilde{L}^\infty_t(\dot{B}^{\frac{N}{2}+1-\alpha}_{2,1})}
  \Vert u \Vert_{L^1_t(\dot{B}^{\frac{N}{2}+1}_{2,1})}
  \le C\big(U_L^2(t)+\widetilde{X}^2(t) \big).
\end{aligned}
\end{equation*}
From \eqref{eq.local.G2} we get
\[
  \Vert \Lambda^\alpha (uh(\sigma)) - u \Lambda^\alpha h(\sigma)\Vert_{L^1_t(\dot{B}^{\frac{N}{2}+1-\alpha}_{2,1})}\leq C \|u \|_{L^1_t(\dot B^{\frac{N}{2}+1}_{2,1})} \|\sigma\|_{\widetilde{L}^\infty_t (\dot B^{\frac{N}{2}}_{2,1})}
  \le C\big(U_L^2(t)+\widetilde{X}^2(t)\big) .
\]
Plugging the estimates above into \eqref{eq.local.X}, we have
\begin{equation*}
  \widetilde{X}(t) \le C_1 \Big(\Vert \sigma_0\Vert_{\widetilde{B}^{\frac{N}{2}+1-\alpha,\frac{N}{2}}} +
  t \Vert u_0\Vert_{\dot{B}^{\frac{N}{2}+1-\alpha}_{2,1}} + U_L(t) + U_L^2(t) \Big) + C_2 \widetilde{X}^2(t) ,
\end{equation*}
with $C_1,C_2>0$. Since $\lim_{t\to0}U_L(t)=0$, by letting $T>0$ small enough,
we infer that for every $t\in [0,T]$,
\begin{align}
  \widetilde{X}(t) \leq 2C_1 \Vert \sigma_0\Vert_{\widetilde{B}^{\frac{N}{2}+1-\alpha,\frac{N}{2}}}  + C_2 \widetilde{X}^2(t) .
\end{align}
Let $\eta = \frac{1}{8 C_1 C_2}$ and $C = 4C_1$, a similar continuity argument as in Proposition \ref{prop.nonlin.pro.nl} yields
\[\widetilde{X}(t)< C \Vert \sigma_0\Vert_{\widetilde{B}^{\frac{N}{2}+1-\alpha,\frac{N}{2}}},\quad\forall~t\in[0,T].\]

Finally, combined with \eqref{eq:UL}-\eqref{eq:UL2}, we conclude the proof of \eqref{eq:u-loc}.
\end{proof}

\section{Proof of Theorem \ref{cor.global.solution}}\label{global_solution}
This section is devoted to the proof of Theorem \ref{cor.global.solution}. As a byproduct, we also include the local well-posedness result.

We will mainly prove the following global well-posedess result to the system \eqref{eq.EAsigma.nl}.
\begin{theorem}\label{th.global.solution.nl}
Let $1<\alpha<2$. Consider the system \eqref{eq.EAsigma.nl} with initial data $\sigma_0\in \widetilde{B}^{\frac{N}{2}+1-\alpha,\frac{N}{2}}(\R^N)$ and $u_0\in \dot{B}^{\frac{N}{2}+1-\alpha}_{2,1}(\R^N)$.
There exists a small constant $\varepsilon'>0$, such that if
\begin{equation}\label{assum:sig-u0}
  \Vert \sigma_0\Vert_{ \widetilde{B}^{\frac{N}{2}+1-\alpha,\frac{N}{2}}}
  +\Vert u_0\Vert_{\dot{B}^{\frac{N}{2}+1-\alpha}_{2,1}}\leq \varepsilon',
\end{equation}
then the system \eqref{eq.EAsigma.nl} has a global unique solution $(\sigma,u)$ satisfying
\begin{equation}\label{eq:sig-u-ubdd}
  \Vert \sigma \Vert_{\widetilde{L}^\infty(\R^+;\widetilde{B}^{\frac{N}{2}+1-\alpha,\frac{N}{2}})} + \Vert \sigma\Vert_{L^1(\R^+;\widetilde{B}^{\frac{N}{2}+1,\frac{N}{2}+2-\alpha})}
  + \Vert u\Vert_{\widetilde{L}^\infty(\R^+;\dot{B}^{\frac{N}{2}+1-\alpha}_{2,1})} + \Vert u\Vert_{L^1(\dot{B}^{\frac{N}{2}+1}_{2,1})} \leq C \varepsilon',
\end{equation}
with $C = C(\alpha,N)>0$.
\end{theorem}

Based on Theorem \ref{th.global.solution.nl}, we can immediately present the proof of Theorem \ref{cor.global.solution}. Indeed, noticing that $\rho-1 =h(\sigma)$ and $\sigma = h^{-1}(\rho-1)$ with $h^{-1}$ the inverse function of $h$, the assumption \eqref{eq:cond0} and Lemma \ref{lemma.composition} ensure the condition \eqref{assum:sig-u0}, and thus according to \eqref{eq:sig-u-ubdd} in Theorem \ref{th.global.solution.nl} and Lemma \ref{lemma.composition}, we can conclude \eqref{eq:rho>0} and \eqref{eq:rho-u-bdd0} by letting $\varepsilon$ small enough. Similarly, the higher regularity \eqref{eq:rho-u-hreg} of $(\rho,u)$ can be obtained from the corresponding estimates on $(\sigma,u)$, presented in Proposition \ref{prop.a.4}.


%

\subsection{Local existence}
Let us start with the statement of the following local well-posedness result.

\begin{theorem}[Local well-posedness]\label{th.local.solution.nl}
Let $1<\alpha<2$. Consider the system \eqref{eq.EAsigma.nl} with initial data $\sigma_0\in \widetilde{B}^{\frac{N}{2}+1-\alpha,\frac{N}{2}}(\R^N)$ and $u_0\in \dot{B}^{\frac{N}{2}+1-\alpha}_{2,1}(\R^N)$. There exists a constant $\eta>0$ such that if
\begin{equation}\label{assum:sig1}
  \Vert \sigma_0\Vert_{\widetilde{B}^{\frac{N}{2}+1-\alpha,\frac{N}{2}}}\le \eta,
\end{equation}
then there exists a positive time $T$ such that system \eqref{eq.EAsigma.nl} has a unique solution $(\sigma,u)$ on $[0,T[\times\mathbb{R}^N$ which satisfy for every $T'<T$,
\begin{equation}\label{eq:sig-u-bdd}
\begin{split}
  & \qquad u\in C_b([0,T'];\dot{B}^{\frac{N}{2}+1-\alpha}_{2,1})\cap L^1([0,T'];{\dot{B}^{\frac{N}{2}+1}_{2,1}}),\quad \textrm{and} \\
  &  \Vert \sigma\Vert_{\widetilde{L}^{\infty}_{T}(\widetilde{B}^{\frac{N}{2}+1-\alpha,\frac{N}{2}})}
  + \Vert \sigma\Vert_{L^1_T(\widetilde{B}^{\frac{N}{2}+1,\frac{N}{2}+2-\alpha})} \leq C \Vert \sigma_0\Vert_{\widetilde{B}^{\frac{N}{2}+1-\alpha,\frac{N}{2}}}.
\end{split}
\end{equation}
\end{theorem}

The local solution is constructed through a standard approximation by the paralinearized equations \eqref{eq.lineareqbn0} and applying \emph{a priori} estimates to pass to the limit. We sketch the proof in below, with special attention to the nonlocal alignment term that needs a careful treatment.

The assumption \eqref{assum:sig1} only requires smallness on $\sigma_0$, thanks to Proposition \ref{prop.priori.local}. It is possible to obtain local well-posedness without such condition (see \cite{danchin2014local} on the compressible Navier-Stokes system). Since our global well-posedness result requires a stronger smallness assumption  \eqref{assum:sig-u0}, we do not make an effort to remove this smallness condition.

\begin{proof}[Proof of Theorem \ref{th.local.solution.nl}: existence]
We first construct the following approximate system:
\begin{equation}\label{eq:app-sys}
  \frac{d}{dt}
  \left (\begin{aligned} &\sigma^n\\ & \bar{u}^n
  \end{aligned}\right )
  =\left (\begin{aligned} &F_n( \sigma^n, \bar{u}^n) \\
  & G_n(\sigma^n,\bar{u}^n)
  \end{aligned}\right ),\ \
  \left (\begin{aligned}
  & \sigma^n\\ &\bar{u}^n
  \end{aligned} \right )_{t=0}
  =\left ( \begin{aligned} &\mathcal{J}_n \sigma_0\\ & \;\;0
\end{aligned}\right ),
\end{equation}
where $\mathcal{J}_n $ is the Friedrichs projector defined by
\[\widehat{\mathcal{J}_n f}(\xi) := 1_{\mathcal{C}_n}(\xi) \widehat{f}(\xi),\quad\text{with}\quad
\mathcal{C}_n:=\{\xi\in \mathbb{R}^N;\ n^{-1}\le |\xi|\le n\},\]
and $u^n = \bar{u}^n +u^n _{L}$ with
$u^n _L$ satisfying
\begin{equation*}
  \partial_t u_L^n + \mu\Lambda^\alpha u_L^n = 0,\quad u_L^n|_{t=0} = \mathcal{J}_n u_0,
\end{equation*}
and
\begin{equation*}
  F_n(\sigma^n,\bar{u}^n) = - \lambda \mathcal{J}_n \Div \bar{u}^n - \lambda \mathcal{J}_n \Div u_L^n - \mathcal{J}_n (u^n\cdot\nabla \sigma^n)
  - (\gamma -1)\mathcal{J}_n(\sigma^n \Div u^n ),
\end{equation*}
\begin{equation*}
\begin{aligned}
  G_n(\sigma^n,\bar{u}^n) = - \lambda \mathcal{J}_n \big(\nabla \sigma^n \big) - \mu\mathcal{J}_n \Lambda^\alpha \bar{u}^n
  -\mathcal{J}_n (u^n \cdot \nabla u^n )- \mu \mathcal{J}_n \Big(\Lambda^\alpha \big(u^nh(\sigma^n)\big)-u^n\Lambda^\alpha \big(h(\sigma^n)\big)\Big).
\end{aligned}
\end{equation*}
In addition, we define the set 
\[\dot{L}^2_n :=\{f\in L^2(\mathbb{R}^N );\  \supp \widehat{f}\subset \mathcal{C}_n\}.\]
It is easy to see that the map
$(\sigma^n ,\bar{u}^n )\mapsto \big(F_n (\sigma^n ,\bar{u}^n ),G_n (\sigma^n ,\bar{u}^n )\big)$
is locally Lipschitz continuous and also continuous with respect to $t$ in $(\dot{L}^2_n)^{N+1}$.
Via the Cauchy-Lipschitz theorem, there is a unique local solution $(\sigma^n ,\bar{u}^n )\in C^1([0,T^{*}_n );(\dot{L}^2_n)^{N+1})$ to \eqref{eq:app-sys}.

Since $\mathcal{J}_n $ is uniformly bounded from $L^2$ to $L^2$, from \eqref{eq:UL} and \eqref{eq:UL2},
we see that $u^n_L$ is bounded uniformly in $n$, and satisfies
\[\Vert u^n_L\Vert_{\widetilde{L}^{\infty}_{T}(\dot{B}^{N/2+1-\alpha}_{2,1})}
  + \Vert u^n_L\Vert_{L^1_T(\dot{B}^{N/2+1}_{2,1})}\le C \|u_0\|_{\dot B^{N/2+1-\alpha}_{2,1}}.\]
Arguing as obtaining the \textit{a priori} estimate in Proposition \ref{prop.priori.local},
under the assumption \eqref{assum:sig1} with small $\eta>0$, we infer that
$(\sigma^n,u^n)$ is uniformly-in-$n$ bounded and satisfies that for some $T>0$,
\begin{align}\label{eq:sig-unif-loc}
  \Vert \sigma^n\Vert_{\widetilde{L}^{\infty}_{T}(\widetilde{B}^{\frac{N}{2}+1-\alpha,\frac{N}{2}})}
  + \Vert \sigma^n\Vert_{L^1_T(\widetilde{B}^{\frac{N}{2}+1,\frac{N}{2}+2-\alpha})} \leq C \Vert \sigma_0\Vert_{\widetilde{B}^{\frac{N}{2}+1-\alpha,\frac{N}{2}}},
\end{align}
\begin{equation}\label{eq:sig-u-unif}
\begin{aligned}
  \Vert \sigma^n\Vert_{\widetilde{L}^\infty_T(\widetilde{B}^{\frac{N}{2}+1-\alpha,\frac{N}{2}})}+\Vert u^n\Vert_{\widetilde{L}^\infty_T(\dot{B}^{\frac{N}{2}+1-\alpha}_{2,1})}
  +\Vert \sigma^n\Vert_{L^1_T(\widetilde{B}^{\frac{N}{2}+1,\frac{N}{2}+2-\alpha})}
  + \Vert u^n\Vert_{L^{1}_{T}(\dot{B}^{\frac{N}{2}+1}_{2,1})}\le C X_0,
\end{aligned}
\end{equation}
where $X_0$ is given by \eqref{eq:X0}.
As a result, we have
\[\sigma^n \in \widetilde{L}^{\infty}_T(\widetilde{B}^{\frac{N}{2}+1-\alpha,\frac{N}{2}})\cap L^{1}_T({\widetilde{B}^{\frac{N}{2}+1,\frac{N}{2}+2-\alpha}})\quad\text{and}\quad
 u^n \in  \widetilde{L}^{\infty}_T(\dot{B}^{\frac{N}{2}+1-\alpha}_{2,1})\cap L^1_{T}(\dot{B}^{\frac{N}{2}+1}_{2,1}),\]
uniformly in $n$.
The interpolation also implies that $u^n \in \widetilde{L}^r_T(\dot B^{\frac{N}{2}+1 - \frac{(r-1)\alpha}{r}}_{2,1} )$
for every $r\in [1,\infty]$.

Let $(\tilde{\sigma}^n,\tilde{u}^n)$ be the solution of the following linear system
\begin{align}\label{eq:sig-u-tild}
\begin{cases}
  \partial_t \tilde{\sigma}^n + \lambda \Div \tilde{u}^n = 0, \\
  \partial_t \tilde{u}^n + \mu \Lambda^\alpha \tilde{u}^n + \lambda\nabla \tilde{\sigma}^n =0, \\
  (\tilde{\sigma}^n , \tilde{u}^n) = (\mathcal{J}_n \sigma_0, 0).
\end{cases}
\end{align}
According to Proposition \ref{prop.lin.pro} with $(v,F,G)\equiv (0,0,0)$, we deduce
\begin{align*}
  \Vert \tilde\sigma^n\Vert_{\widetilde{L}^\infty_T(\widetilde{B}^{\frac{N}{2}+1-\alpha,\frac{N}{2}})}
  + \Vert \tilde{u}^n\Vert_{\widetilde{L}^\infty_T(\dot{B}^{\frac{N}{2}+1-\alpha}_{2,1})}
  +\Vert \tilde\sigma^n\Vert_{L^1_T(\widetilde{B}^{\frac{N}{2}+1,\frac{N}{2}+2-\alpha})}
  + \Vert \tilde{u}^n\Vert_{L^{1}_{T}(\dot{B}^{\frac{N}{2}+1}_{2,1})}
  \leq C \|\sigma_0\|_{\widetilde{B}^{\frac{N}{2}+1-\alpha,\frac{N}{2}}}.
\end{align*}
Due to the fact that $\{(\tilde{\sigma}^n, \tilde{u}^n)\}_{n\in\N}$ is a Cauchy sequence in the considered space, there exists functions $(\tilde{\sigma},\tilde{u})$ such that $\tilde\sigma^n \rightarrow \tilde\sigma$ in $\widetilde{L}^\infty_T(\widetilde{B}^{\frac{N}{2}+1-\alpha,\frac{N}{2}}) \cap L^1_T(\widetilde{B}^{\frac{N}{2}+1,\frac{N}{2}+2-\alpha})$ and $\tilde{u}^n\rightarrow \tilde{u}$ in $\widetilde{L}^\infty_T(\dot B^{\frac{N}{2}+1-\alpha}_{2,1})\cap L^1_T(\dot B^{\frac{N}{2}+1}_{2,1})$.

Now denote by $\bar{\sigma}^n := \sigma^n - \tilde{\sigma}^n$ for every $n\in \N$.
In order to gain more time-continuity information, we consider $(\partial_t \bar{\sigma}^n, \partial_t \bar{u}^n)$.
Note that
\[ \partial_t \bar{\sigma}^n = - \lambda \mathcal{J}_n \Div u^n - \lambda \Div \tilde{u}^n - \mathcal{J}_n (u^n\cdot\nabla \sigma^n)
  - (\gamma -1)\mathcal{J}_n(\sigma^n \Div u^n ),\]
and by virtue of the above uniform estimates and Lemma \eqref{lemma.Bilinear}, we have
\begin{align}\label{eq:par-t-sig}
  \Vert\partial_t \bar{\sigma}^n \Vert_{L^{\frac{\alpha}{\alpha-1}}_T(\dot{B}^{\frac{N}{2}-1}_{2,1})}
  \lesssim \Vert u^n \Vert_{L^{\frac{\alpha}{\alpha-1}}_T(\dot{B}^{\frac{N}{2}}_{2,1})}
  + \|\tilde{u}^n\|_{L^{\frac{\alpha}{\alpha-1}}_T(\dot B^{\frac{N}{2}}_{2,1})}
  + \Vert u^n\Vert_{L^{\frac{\alpha}{\alpha-1}}_T(\dot{B}^{\frac{N}{2}}_{2,1})}
  \Vert \sigma^n\Vert_{L^{\infty}_T(\dot{B}^{\frac{N}{2}}_{2,1})} < \infty.
\end{align}
Thus, $\bar{\sigma}^n$ is uniformly bounded in
\begin{equation}\label{eq.b-td.nl}
  C([0,T];\dot{B}^{\frac{N}{2}}_{2,1})\cap C^{\frac{1}{\alpha}}([0,T];\dot{B}^{\frac{N}{2}-1}_{2,1}).
\end{equation}
Meanwhile, recalling the equation of $\bar{u}^n $ in \eqref{eq:app-sys},
and with the help of Lemma \ref{lemma.Bilinear} and Corollary \ref{corollary.kato_ponce}, we obtain
\begin{equation}\label{eq:par-t-u}
\begin{aligned}
  \Vert\partial_t\bar{u}^n\Vert_{L^{\frac{\alpha}{\alpha-1}}_T
  (\widetilde{B}^{\frac{N}{2}-1, \frac{N}{2} -\alpha})} 
  & \lesssim \Vert \sigma^n \Vert_{L^{\frac{\alpha}{\alpha-1}}_T(\dot{B}^{\frac{N}{2}}_{2,1})}
  + \Vert \bar{u}^n\Vert_{L^{\frac{\alpha}{\alpha-1}}_T(\dot{B}^{\frac{N}{2}}_{2,1})} \\
  &\quad + \Vert u^n\Vert_{L^{\frac{\alpha}{\alpha-1}}_T(\dot{B}^{\frac{N}{2}}_{2,1})}
  \Big(\Vert \sigma^n\Vert_{L^{\infty}_T(\dot{B}^{\frac{N}{2}}_{2,1})}
  + \Vert  u^n\Vert_{L^{\infty}_T(\dot{B}^{\frac{N}{2}+1-\alpha}_{2,1})}\Big) < \infty.
\end{aligned}
\end{equation}
So, $\bar{u}^n $ is uniformly bounded in
\begin{equation}\label{eq.b-tu.nl}
  C([0,T];\dot{B}^{\frac{N}{2}+1-\alpha}_{2,1})\cap
  C^{\frac{1}{\alpha}}([0,T];\widetilde{B}^{\frac{N}{2}-1, \frac{N}{2} -\alpha}).
\end{equation}

Let $\{\phi_k\}_{k\in \mathbb{N}}\in C^\infty_c(\R^N)$ be a sequence of bump functions which are supported in
$B(0, k+1)$ and equal to $1$ on $B(0, k)$.
According to \eqref{eq.b-td.nl}-\eqref{eq.b-tu.nl} and Lemma \ref{lem:Bes-K},
we have that $\{\phi_k \bar{\sigma}^n\}_{n\ge 1}$ is uniformly bounded in
$C([0,T];{B}^{\frac{N}{2}}_{2,1})\cap C^{\frac{1}{\alpha}}([0,T];{B}^{\frac{N}{2}-1}_{2,1})$
for all $k \in \mathbb{N}$,
and $\{\phi_k \bar{u}^n\}_{n\ge 1}$ is uniformly bounded in
$C([0,T];{B}^{\frac{N}{2}+1-\alpha}_{2,1})\cap C^{\frac{1}{\alpha}}([0,T];{B}^{\frac{N}{2}-\alpha}_{2,1})$
for all $k \in \mathbb{N}$.
Note that the map $f \mapsto \phi_k f$ is compact from ${B}^{\frac{N}{2}}_{2,1}$ to ${B}^{\frac{N}{2}-1}_{2,1}$,
and also from ${B}^{\frac{N}{2}+1-\alpha}_{2,1}$ to ${B}^{\frac{N}{2}-\alpha}_{2,1}$.
Thanks to Ascoli's theorem and the diagonal process, we can find a subsequence of $\{(\phi_k\bar{\sigma}^n ,\phi_k\bar{u}^n)\}_{n\ge 1}$
(still uses this notation)
satisfying that for all $k\in \mathbb{N}$ and as $n\rightarrow \infty$,
$\phi_k \bar{\sigma}^n$ converges to $\bar{\sigma}_k$ in  $C([0,T];{B}^{\frac{N}{2}-1}_{2,1})$,
and $\phi_k\bar{u}^n $  converges to $\bar{u}_k$ in $ C([0,T];{B}^{\frac{N}{2}-\alpha}_{2,1}) $.
Define
$(\bar{\sigma},\bar{u}):= (\bar{\sigma}_k,\bar{u}_k)$ for $x\in B(0,k)$.
Since $\phi_k \phi_{k+1}=\phi_k$, we infer that $(\bar{\sigma},\bar{u})$ is well-defined.
Hence, for all $\phi\in C_c^\infty(\mathbb{R}^N)$, we have
\begin{align*}
  (\phi \bar{\sigma}^n ,\phi\bar{u}^n) \rightarrow (\phi \bar{\sigma}, \phi\bar{u}),\quad \textrm{in} \;\; C([0,T];{B}^{\frac{N}{2}-1}_{2,1}\times{B}^{\frac{N}{2}-\alpha}_{2,1}).
\end{align*}
Via Fatou's property for Besov space (see \cite[Thm. 2.25]{bahouri2011fourier}) and the uniform estimates of $(\phi \bar{\sigma}^n ,\phi\bar{u}^n )$, we moreover get
\begin{align}\label{eq:sig-u-bdd3}
  \sigma = \bar{\sigma} + \tilde{\sigma}  \in \widetilde{L}^{\infty}_{T}(\widetilde{B}^{\frac{N}{2}+1 -\alpha,\frac{N}{2}})
  \cap L^1_T(\widetilde{B}^{\frac{N}{2}+1,\frac{N}{2}+2-\alpha}) ,\
  u = \bar{u} + u_L \in \widetilde{L}^{\infty}_{T}(\dot{B}^{\frac{N}{2}+1-\alpha}_{2,1})
  \cap L^1_{T}(\dot{B}^{\frac{N}{2}+1}_{2,1}).
\end{align}

Now, we pass to the limit in the approximate system \eqref{eq:app-sys}. We here only deal with the term
$\mathcal{J}_n \big(\Lambda^\alpha \big(u^nh(\sigma^n)\big)-u^n\Lambda^\alpha \big(h(\sigma^n)\big)\big)$,
since the remaining terms can be treated by the standard process.
Let $\phi\in C_c^\infty(\mathbb{R}^N)$ be any fixed function with $\supp \phi \subset B(0,R)$.
Recalling that $\varphi$ is the function introduced in \eqref{eq:chi-phi}, there exists a bump function
$\chi\in C^\infty_c(\mathbb{R}^N)$ supported in the ball $B(0,\frac{4}{3})$ such that
$\chi(x) + \sum_{j\geq 0} \varphi_j(x) = 1$ for every $x\in\mathbb{R}^N$ (see \cite{bahouri2011fourier}).
Clearly, $\chi\equiv 1$ in $B(0,\frac{3}{4})$, and we can choose $k\in \N$ large enough so that
$\chi_k(\cdot) = \chi(2^{-k}\cdot) \equiv 1 $ in $B(0,2R)$ and it also holds that
$\chi_k(x) + \sum_{j\geq k} \varphi_j(x) =1$ for every $x\in \R^N$.
We write
\begin{align}
  &\quad \big\langle \mathcal{J}_n \big(\Lambda^\alpha \big(u^nh(\sigma^n)\big)-u^n\Lambda^\alpha \big(h(\sigma^n)\big)\big)-\big(\Lambda^\alpha \big(uh(\sigma)\big)
  - u\Lambda^\alpha \big(h(\sigma)\big)\big),\phi \big\rangle \nonumber \\
  & = \big\langle\Lambda^\alpha \big(u^nh(\sigma^n)\big)
  - u^n\Lambda^\alpha \big(h(\sigma^n)\big),(\mathcal{J}_n-\Id) \phi \big\rangle \nonumber \\
  & \quad + \big\langle
  \big( \Lambda^\alpha
  (\chi_k u^n\, h(\sigma^n))- u^n \Lambda^\alpha (\chi_k h(\sigma^n))\big)
  - \big( \Lambda^\alpha
  ( \chi_k u\, h(\sigma))- u \Lambda^\alpha(\chi_k h(\sigma))\big),
  \phi \big\rangle \nonumber \\
  &\quad + \sum_{j=k}^\infty \big\langle  \big( \Lambda^\alpha
  (\varphi_j u^n\, h(\sigma^n))- u^n \Lambda^\alpha (\varphi_j h(\sigma^n))\big)
  - \big( \Lambda^\alpha
  ( \varphi_j u\, h(\sigma))- u \Lambda^\alpha(\varphi_j h(\sigma))\big),
  \phi \big\rangle \nonumber \\
  & =: I_1^n + I_2^n + I_3^n.
\end{align}
Similarly as \eqref{eq.local.G2}, we know that $\Lambda^\alpha (u^nh(\sigma^n))
 - u^n\Lambda^\alpha (h(\sigma^n)) $ has a uniformly-in-$n$ bound in $L^1_T(\dot B^{\frac{N}{2}+1-\alpha}_{2,1})$. Then by virtue of the dual property and the smoothness of $\phi$, we can show that \[\lim\limits_{n\rightarrow \infty} I_1^n =0.\]
Note that $h(\sigma)$ given by \eqref{eq:h} has the same spatial support with $\sigma$.
Hence $\chi_k h(\sigma^n) = \chi_k h(\chi_{k+2} \sigma^n)$, and we can write $I_2^n$ as
\begin{align*}
  I_2^n = \big\langle  [\Lambda^\alpha, \chi_{k+2} (u^n - u) ] \,\big(\chi_k h(\chi_{k+2}\sigma^n)  \big),  \phi \big\rangle
  + \big\langle  [\Lambda^\alpha,  \chi_{k+2} u ] \,\big(\chi_k (h(\chi_{k+2}\sigma^n )
  - h(\chi_{k+2}\sigma))\big), \phi\rangle.
\end{align*}
Since $\chi_{k+2}u^n\to\chi_{k+2}u$ in $L^1_T(B^{\frac{N}{2}+1-\varepsilon}_{2,1})$
and $\chi_{k+2}\sigma^n\to\chi_{k+2}\sigma$ in $L^\infty_T(B^{\frac{N}{2}-\varepsilon}_{2,1})$ with any $\varepsilon>0$ small,
and taking advantage of Corollary \ref{corollary.kato_ponce} (with $r=\frac{N}{2}+1$ and $r_1=\frac{N}{2}+\frac{\alpha}{2}$)
and Lemma \ref{lemma.composition},
we get
\begin{align*}
  |I_2^n| & \leq C \|\chi_{k+2} (u^n-u)\|_{L^1_T(\dot B^{1+\frac{\alpha}{2}}_{2,1})}
  \|h(\chi_{k+1}\sigma^n) \|_{L^\infty_T(\dot B^{\frac{N}{2}+\frac{\alpha}{2}-1}_{2,1})} \|\phi\|_{L^2} \\
  & \quad + C \|\chi_{k+2} u\|_{L^1_T(\dot B^{1+\frac{\alpha}{2}}_{2,1})}
  \|h(\chi_{k+2}\sigma^n) - h(\chi_{k+2} \sigma)\|_{L^\infty_T(\dot B^{\frac{N}{2}+\frac{\alpha}{2}-1}_{2,1})} \|\phi\|_{L^2} \\
  & \leq C \|\chi_{k+2} (u^n-u)\|_{L^1_T(\dot B^{1+\frac{\alpha}{2}}_{2,1})}  +
  C  \|\chi_{k+2} (\sigma^n - \sigma)\|_{L^\infty_T(\dot B^{\frac{N}{2}+\frac{\alpha}{2}-1}_{2,1})} \rightarrow 0,
  \quad \textrm{as}\;\; n\rightarrow \infty.
\end{align*}
Next we consider $I_3^n$. Noting that $\varphi_j h(\sigma^n) = \varphi_j h(\chi_{j+2}\sigma^n)$ for every $j\in \N$,
and by using the integral formula of $\Lambda^\alpha $ given in \eqref{def:Lam-alp} and the support property, 
we infer that
\begin{align*}
  & \quad \sum_{j=k}^\infty \big|\big\langle \Lambda^\alpha (\varphi_j u^n h(\sigma^n)) - \Lambda^\alpha (\varphi_j u h(\sigma)) , \phi \big\rangle \big| \nonumber \\
  & \leq \sum_{j=k}^\infty \big|\big\langle \Lambda^\alpha (\varphi_j (u^n - u) h(\chi_{j+2}\sigma^n)) , \phi \big\rangle \big|
  + \sum_{j=k}^\infty \big|\big\langle \Lambda^\alpha \big(\varphi_j u \big( h(\chi_{j+2}\sigma^n) - h(\chi_{j+2}\sigma)\big)\big) , \phi \big\rangle \big| \nonumber \\
  & \lesssim \sum_{j=k}^\infty \Big|\int_0^T\int_{|x|\lesssim 2^k} \int_{|y|\sim 2^j} \frac{ \varphi_j(y) (u^n - u)(y) h(\chi_{j+2}\sigma^n) (y)}{ |x - y|^{N+\alpha}} \phi(x) \dd y \dd x \dd t \Big| \nonumber \\
  & \quad +  \sum_{j=k}^\infty \Big|\int_0^T\int_{|x|\lesssim 2^k} \int_{|y|\sim 2^j} \frac{  \varphi_j(y) u(y) \big( h(\chi_{j+2}\sigma^n) - h(\chi_{j+2}\sigma)\big) (y)}{ |x - y|^{N+\alpha}} \phi(x) \dd y \dd x \dd t \Big| \nonumber \\
  & \lesssim \sum_{j=k}^\infty 2^{-j(\alpha + \frac{N}{2})} \|\varphi_j (u^n - u)\|_{L^1_T(L^2)} \|h(\chi_{j+2}\sigma^n)\|_{L^\infty_T(L^\infty)}  \nonumber \\
  & \quad + \sum_{j=k}^\infty 2^{-j(\alpha + \frac{N}{2})} \|h(\chi_{j+2} \sigma^n) - h(\chi_{j+2}\sigma) \|_{L^\infty_T(L^2)}  \|\varphi_j u\|_{L^1_T(L^\infty)} \nonumber \\
  & \leq C \sum_{j=k}^\infty 2^{-j (\alpha + \frac{N}{2})} \Big(\|\varphi_j (u^n - u)\|_{L^1_T(L^2)} + \|\chi_{j+2} (\sigma^n - \sigma)\|_{L^\infty_T(L^2)} \Big),
\end{align*}
where in the last line we have used the uniform (in $j$ and $n$) estimates
\begin{align*}
\|h(\chi_{j+2}\sigma^n)\|_{L^\infty_T(L^\infty)} \lesssim&\, \|\chi_{j+2} \sigma^n\|_{L^\infty_T(L^\infty)} \lesssim \|\sigma^n \|_{L^\infty_T(\dot B^{\frac{N}{2}}_{2,1})} \leq C,\\
\|\varphi_j u\|_{L^1_T(L^\infty)} \lesssim&\, T^{\frac{1}{\alpha}} \|u\|_{L^{\frac{\alpha}{\alpha-1}}_T(\dot B^{\frac{N}{2}}_{2,1})} \leq C.
\end{align*}
For any $\epsilon >0$, Lemma \ref{lem:Bes-K} and \eqref{eq:sig-unif-loc}, \eqref{eq:sig-u-bdd3} ensure that
\begin{align*}
  \|\varphi_j (u^n -u)\|_{L^1_T(L^2)} \leq&\, C 2^{j (\frac{1}{2} + \frac{1-\alpha}{N})} T\,\|u^n -u\|_{L^\infty_T(\dot B^{\frac{N}{2}+1-\alpha}_{2,1})}  \leq C 2^{j\frac{1}{2}}, \\
  \|\chi_{j+2} (\sigma^n -\sigma)\|_{L^\infty_T(L^2)} \leq&\, C 2^{j \frac{1}{2}} \|\sigma^n -\sigma\|_{L^\infty_T(\dot B^{\frac{N}{2}}_{2,1})}  \leq C 2^{j\frac{1}{2}},
\end{align*}
with $C>0$ independent of $j$ and $n$, then there exists a large number $J\in \N$ so that
\begin{align*}
  C \sum_{j\geq J}  2^{-j (\alpha + \frac{N}{2})} \Big(\|\varphi_j (u^n - u)\|_{L^1_T(L^2)} + \|\chi_{j+2} (\sigma^n - \sigma)\|_{L^\infty_T(L^2)} \Big)
  \leq C \sum_{j\geq J} 2^{-j (\alpha + \frac{N-1}{2})} \leq \frac{\epsilon}{4}.
\end{align*}
On the other hand, the above established convergence result guarantees that there exists $n_0=n_0(\alpha,N,J,\epsilon)\in \N$ so that for any $n\geq n_0$,
\begin{align*}
   C \sum_{k\leq j\leq J} 2^{-j (\alpha + \frac{N}{2})} \Big(\|\varphi_j (u^n - u)\|_{L^1_T(L^2)} + \|\chi_{j+2} (\sigma^n - \sigma)\|_{L^\infty_T(L^2)} \Big)\leq \frac{\epsilon}{4}.
\end{align*}
Thus for any $\epsilon>0$, there exists $n_0\in \N$ so that for any $n\geq n_0$,
\begin{align}\label{eq:I3-1}
  \sum_{j=k}^\infty \big|\big\langle \Lambda^\alpha (\varphi_j u^n h(\sigma^n)) - \Lambda^\alpha (\varphi_j u h(\sigma)) , \phi \big\rangle \big|\leq \frac{\epsilon}{2}.
\end{align}
Similarly, for the remaining term in $I_3^n$ we have
\begin{align*}
  & \quad \sum_{j=k}^\infty \big|\big\langle u^n \Lambda^\alpha (\varphi_j h(\sigma^n)) - u \Lambda^\alpha (\varphi_j h(\sigma)) , \phi \big\rangle \big| \\
  & \leq \sum_{j=k}^\infty \big|\big\langle \chi_k (u^n -u)\, \Lambda^\alpha (\varphi_j h(\sigma^n)) , \phi \big\rangle \big|
  + \sum_{j=k}^\infty \big|\big\langle \chi_k u \,\Lambda^\alpha \big(\varphi_j \big( h(\chi_{j+2}\sigma^n) - h(\chi_{j+2}\sigma)\big)\big) , \phi \big\rangle \big| \\
  & \lesssim \sum_{j=k}^\infty \Big|\int_0^T\int_{|x|\lesssim 2^k} \int_{|y|\sim 2^j} \frac{\varphi_j(y) h(\sigma^n) (y)}{ |x - y|^{N+\alpha}} \chi_k(x) (u^n -u )(x)\phi(x) \dd y \dd x \dd t \Big| \\
  & \quad +  \sum_{j=k}^\infty \Big|\int_0^T\int_{|x|\lesssim 2^k} \int_{|y|\sim 2^j}
  \frac{ \varphi_j(y) \big( h(\chi_{j+2}\sigma^n) - h(\chi_{j+2}\sigma)\big) (y)}{ |x - y|^{N+\alpha}} \chi_k(x) u(x) \phi(x) \dd y \dd x \dd t \Big| \\
  & \leq C \|\chi_k (u^n - u)\|_{L^1_T(L^2)} \sum_{j=k}^\infty 2^{-j\alpha} \|h(\sigma^n)\|_{L^\infty_T(L^\infty)}  \\
  & \quad + C \sum_{j=k}^\infty 2^{-j(\alpha + \frac{N}{2})} \|h(\chi_{j+2} \sigma^n) - h(\chi_{j+2}\sigma) \|_{L^\infty_T(L^2)}
  \|\chi_k u\|_{L^1_T(L^\infty)} .
\end{align*}
Arguing as above, we can deduce that there exists a constant $n_1 \in\N$ so that for any $n\geq n_1$,
\begin{align}\label{eq:I3-2}
  \sum_{j=k}^\infty \big|\big\langle u^n \Lambda^\alpha (\varphi_j h(\sigma^n)) - u \Lambda^\alpha (\varphi_j h(\sigma)) , \phi \big\rangle \big| \leq \frac{\epsilon}{2} .
\end{align}
Hence, combining \eqref{eq:I3-1} with \eqref{eq:I3-2}, for any $\epsilon>0$ we have that $|I_3^n| \leq \epsilon$ for every $n\geq \max\{n_0,n_1\}$, which implies
$\lim\limits_{n\rightarrow \infty} I_3^n = 0$. Therefore, gathering the above convergence result of $I_1^n$-$I^n_3$,
we conclude the convergence of the term $\mathcal{J}_n \big(\Lambda^\alpha \big(u^nh(\sigma^n)\big)-u^n\Lambda^\alpha \big(h(\sigma^n)\big)\big)$.

Finally, we show the time continuity property of $(\sigma, u)$.
Note that $\sigma\in \widetilde{L}^{\infty}_{T}(\widetilde{B}^{\frac{N}{2}+1 -\alpha,\frac{N}{2}})$ and
$u\in \widetilde{L}^\infty_T(\dot B^{\frac{N}{2}+1-\alpha}_{2,1})$,
and using the fact that $\partial_t \sigma \in L^{\frac{\alpha}{\alpha-1}}_T(\dot{B}^{\frac{N}{2}-1}_{2,1})$
and $\partial_t u \in L^{\frac{\alpha}{\alpha-1}}_T (\widetilde{B}^{\frac{N}{2}-1, \frac{N}{2} -\alpha})$ (similarly as obtaining \eqref{eq:par-t-sig}, \eqref{eq:par-t-u}),
we can apply a simple argument of high-low frequency decomposition (e.g. see \cite[Thm. 3.3.1]{danchin2005fourier}) to conclude that $\sigma\in C([0,T];\widetilde{B}^{\frac{N}{2}+1-\alpha,\frac{N}{2}})$ and $u\in C([0,T];\dot{B}^{\frac{N}{2}+1-\alpha}_{2,1})$.
\end{proof}

\subsection{Uniqueness}
We continue with a uniqueness argument for our constructed solution.

\begin{proof}[Proof of Theorem \ref{th.local.solution.nl}: uniqueness]
Assume that $(\sigma_1,u_1)$ and $(\sigma_2,u_2)$ satisfying
\begin{align}\label{assum:uni}
  u_i \in L^1_T(\dot B^{\frac{N}{2}+1}_{2,1}),
  \quad \|\sigma_i\|_{L^\infty_T(\widetilde{B}^{\frac{N}{2}+1-\alpha,\frac{N}{2}})} \leq C\eta,
  \quad \textrm{with $\eta>0$ small enough},
\end{align}
are two regular solutions of system \eqref{eq.EAsigma.nl} associated with the same initial data $(\sigma_0,u_0)$.
Denote $(\delta \sigma, \delta u) :=(\sigma_1-\sigma_2,u_1-u_2)$, then
\begin{equation*}
\begin{cases}
  \partial_t \delta \sigma + T_{u_1}\cdot\nabla \delta \sigma + \lambda\Div \delta u = \delta F,\\
  \partial_t \delta u + T_{u_1} \cdot \nabla \delta u + \Lambda^\alpha \delta u + \lambda\nabla \delta \sigma = \delta G, \\
  (\delta \sigma, \delta u)|_{t=0} = (0,0),
\end{cases}
\end{equation*}
where
\begin{equation*}
\begin{aligned}
  \delta F := -\delta u\cdot\nabla \sigma_2 - (\gamma-1)(\delta \sigma \Div u_1 + \sigma_2\Div \delta u) - T_{\nabla \delta\sigma}\cdot u_1 - R(u_1,\nabla \delta\sigma),
\end{aligned}
\end{equation*}
\begin{equation*}
  \delta G := -\delta u\cdot\nabla u_2 - \mu \big(\Lambda^\alpha (\delta u\,h(\sigma_1)) -
  \delta u\, \Lambda^\alpha (h(\sigma_1)) \big) - \mu \big(\Lambda^\alpha (u_2\delta h) - u_2\Lambda^\alpha (\delta h)\big)
  - T_{\nabla \delta u}\cdot u_1 - R(u_1, \nabla \delta u),
\end{equation*}
with $\delta h = h(\sigma_1)-h(\sigma_2)$.
Applying \eqref{eq:sig-u-es2} in Propositions \ref{prop.lin.pro} leads to that
\begin{equation*}
\begin{aligned}
  & \quad \Vert \delta \sigma\Vert_{\widetilde{L}_t^\infty (\widetilde{B}^{\frac{N}{2}-\alpha,\frac{N}{2}-1})}
  + \Vert \delta u\Vert_{\widetilde{L}_t^\infty (\dot{B}^{\frac{N}{2}-\alpha}_{2,1})}
  + \int_0^t \Big(\Vert \delta \sigma(\tau)\Vert_{\widetilde{B}^{\frac{N}{2},\frac{N}{2}+1-\alpha}}
  + \Vert \delta u(\tau)\Vert_{\dot{B}^{\frac{N}{2}}_{2,1}}\Big) \dd \tau\\
  & \leq C \int_0^t \Vert u_1 \Vert_{\dot{B}^{\frac{N}{2}+1}_{2,1}}
  \Big(\Vert \delta \sigma \Vert_{\widetilde{B}^{\frac{N}{2}-\alpha,\frac{N}{2}-1}}
  + \Vert \delta u\Vert_{\dot{B}^{\frac{N}{2}-\alpha}_{2,1}} \Big) \dd \tau + C \Vert \delta F \Vert_{L^1_t (\widetilde{B}^{\frac{N}{2}-\alpha,\frac{N}{2}-1})}
  +C \Vert \delta G\Vert_{L^1_t(\dot{B}^{\frac{N}{2}-\alpha}_{2,1})} .
\end{aligned}
\end{equation*}
Applying \eqref{eq:compest2} in Lemma \ref{lemma.composition}, we get
\begin{align*}
  \Vert \delta h\Vert_{\dot{B}^{\frac{N}{2}-1}_{2,1}} = \|h(\sigma_1) -h(\sigma_2)\|_{\dot{B}^{\frac{N}{2}-1}_{2,1}} \lesssim \big(1+ \|(\sigma_1,\sigma_2)\|_{L^\infty_t(\dot{B}^{\frac{N}{2}}_{2,1})}\big)^{[\frac{N}{2}+1]}
  \Vert \delta \sigma\Vert_{\dot{B}^{\frac{N}{2}}_{2,1}}
  \le C \Vert \delta \sigma\Vert_{\dot{B}^{\frac{N}{2}}_{2,1}}.
\end{align*}
Taking advantage of Lemma \ref{lemma.Bilinear}, Corollary \ref{corollary.kato_ponce} and the above inequality, we infer that
\begin{equation*}
  \Vert \delta F \Vert_{L^1_t(\widetilde{B}^{\frac{N}{2}-\alpha,\frac{N}{2}-1})}
  \leq C \int_0^t\Vert \delta \sigma \Vert_{\widetilde{B}^{\frac{N}{2}-\alpha,\frac{N}{2}-1}}
  \Vert u_1\Vert_{\dot{B}^{\frac{N}{2}+1}_{2,1}}\dd\tau
  + C \Vert \sigma_2\Vert_{L^\infty_t(\widetilde{B}^{\frac{N}{2}+1-\alpha,\frac{N}{2}})}
  \Vert \delta u\Vert_{L^1_t(\dot{B}^{\frac{N}{2}}_{2,1})},
\end{equation*}
\begin{equation*}
\begin{aligned}
  \Vert \delta G \Vert_{L^1_t(\dot{B}^{\frac{N}{2}-\alpha}_{2,1})} \le
  C\int_0^t \Big( \Vert (u_1,u_2)\Vert_{\dot{B}^{\frac{N}{2}+1}_{2,1}} \Vert \delta u\Vert_{\dot{B}^{\frac{N}{2}-\alpha}_{2,1}}
  + \Vert \delta u\Vert_{\dot{B}^{\frac{N}{2}}_{2,1}} \Vert h(\sigma_1)\Vert_{\dot{B}^{\frac{N}{2}}_{2,1}}
  + \Vert u_2\Vert_{\dot{B}^{\frac{N}{2}+1}_{2,1}} \Vert \delta h\Vert_{\dot{B}^{\frac{N}{2}-1}_{2,1}}\Big)\dd\tau & \\
  \le C\int_0^t \Big( \Vert (u_1,u_2)\Vert_{\dot{B}^{\frac{N}{2}+1}_{2,1}}\Vert \delta u\Vert_{\dot{B}^{\frac{N}{2}-\alpha}_{2,1}}
  + \Vert u_2\Vert_{\dot{B}^{\frac{N}{2}+1}_{2,1}} \Vert \delta \sigma \Vert_{\dot{B}^{\frac{N}{2}-1}_{2,1}}\Big) \dd\tau
  + C \Vert \delta u \Vert_{L^1_t(\dot{B}^{\frac{N}{2}}_{2,1})} \Vert \sigma_1\Vert_{L^\infty_t(\dot{B}^{\frac{N}{2}}_{2,1})}.&
\end{aligned}
\end{equation*}
By setting $\eta >0$ in \eqref{assum:uni} be small enough so that
$C\big(\Vert \sigma_1\Vert_{L^\infty_T(\dot{B}^{\frac{N}{2}}_{2,1})} + \Vert \sigma_2\Vert_{L^\infty_T(\widetilde{B}^{\frac{N}{2}+1-\alpha,\frac{N}{2}})} \big)\le \frac{1}{4 C}$,
we gather the above estimates to get
\begin{equation*}
\begin{aligned}
  &\quad \Vert \delta \sigma\Vert_{\widetilde{L}_t^{\infty}(\widetilde{B}^{\frac{N}{2}-\alpha,\frac{N}{2}-1})} +
  \Vert \delta u\Vert_{\widetilde{L}_t^{\infty}(\dot{B}^{\frac{N}{2}-\alpha}_{2,1})} \\
  &\le C\int_0^t \Vert (u_1,u_2)(\tau)\Vert_{\dot{B}^{\frac{N}{2}+1}_{2,1}}
  \Big(\Vert \delta \sigma(\tau)\Vert_{\widetilde{B}^{\frac{N}{2}-\alpha,\frac{N}{2}-1}}
  + \Vert \delta u(\tau)\Vert_{\dot{B}^{\frac{N}{2}-\alpha}_{2,1}} \Big)\dd \tau.
\end{aligned}
\end{equation*}
Gronwall's inequality guarantees that
$\Vert \delta \sigma \Vert_{\widetilde{L}^\infty_t(\widetilde{B}^{\frac{N}{2}-\alpha,\frac{N}{2}-1})}
+ \Vert \delta u\Vert_{\widetilde{L}^\infty_t(\dot{B}^{\frac{N}{2}-\alpha}_{2,1})}\equiv0$ for every $t\in [0,T]$.
In other words, $\sigma_1\equiv \sigma_2$ and $u_1\equiv u_2$ on $[0,T]\times\mathbb{R}^N$.
\end{proof}

\subsection{Global existence}
Now we are ready to prove Theorem \ref{th.global.solution.nl}. Under the smallness assumption on $X_0$ in \eqref{assum:sig-u0}, Proposition \ref{prop.nonlin.pro.nl} ensures the smallness of $X(T)$. Hence, we can extend the solution using the local existence result in Theorem \ref{th.local.solution.nl}, viewing $T$ as the initial time. Repeating the process, we obtain a global solution.

\begin{proof}[Proof of Theorem \ref{th.global.solution.nl}]
Take $\varepsilon'=\min\{\varepsilon_0, \eta/C_*\}$, where $(\varepsilon_0, C_*)$ are the constants in Proposition \ref{prop.nonlin.pro.nl} and $\eta$ is the constant in Theorem \ref{th.local.solution.nl}.
Let $T_*$ be the maximal existence time of the solution $(\sigma,u)$ to \eqref{eq.EAsigma.nl}, namely
\[T_*=\sup\{T\geq0 : \exists\text{ a solution }(\sigma,u) \text{ on } [0,T]\times\R^N\}.\]
We will show $T_*=\infty$ by contradiction.

Suppose $T_*$ is finite. A direct application of Proposition \ref{prop.nonlin.pro.nl} yields
\[X(T)\leq C_* X_0\leq \eta, \quad \forall~T<T_*.\]
The continuity of $X$ then implies $X(T_*)\leq\eta$. We then apply Theorem \ref{th.local.solution.nl} to the system \eqref{eq.EAsigma.nl} initiated at time $T$. There exists a time $T>0$ such that $(\sigma,u)$ exists in $[T_*, T_*+T]$. This contradicts the definition of $T_*$.
\end{proof}


\subsection{Propagation of smooth initial data}
In this subsection, we show that if the initial data is smoother (also known as subcritical), the solution will inherit the initial regularity.
\begin{proposition}\label{prop.a.4}
Under the assumption of Theorem \ref{th.global.solution.nl}, and additionally assuming
\begin{equation*}
  \Vert \sigma_0\Vert_{ \widetilde{B}^{s,s+\alpha-1}}
  +\Vert u_0\Vert_{\dot{B}^{s}_{2,1}}< +\infty,\quad\textrm{for}\;\; s>\frac{N}{2}+1-\alpha,
\end{equation*}
The solution $(\sigma,u)$ to \eqref{eq.EAsigma.nl} satisfies
\begin{equation}\label{eq.A.4.0}
\Vert \sigma\Vert_{\widetilde{L}_T^{\infty} (\widetilde{B}^{s,s+\alpha-1})} + \Vert u\Vert_{\widetilde{L}_T^\infty (\dot{B}^s_{2,1})}
  + \int_0^T \Vert \sigma(\tau)\Vert_{\widetilde{B}^{s+\alpha,s+1}} \dd \tau + \int_0^T \Vert u(\tau)\Vert_{\dot{B}^{s+\alpha}_{2,1}} \dd \tau <+\infty,
\end{equation}
for all $T>0$.
\end{proposition}

\begin{proof}[Proof of Proposition \ref{prop.a.4}]
Let $F$ and $G$ be defined as in \eqref{eq:FG}. The system \eqref{eq.EAsigma.nl} can be seen as in the form of \eqref{eq.lineareqbn0} with $v$ replaced by $u$. Then according to Proposition \ref{prop.lin.pro}, we have
\begin{equation*}
\begin{aligned}
  &\quad \Vert \sigma\Vert_{\widetilde{L}_T^{\infty} (\widetilde{B}^{s,s+\alpha-1})} + \Vert u\Vert_{\widetilde{L}_T^\infty (\dot{B}^s_{2,1})}
  + \int_0^T \Vert \sigma(\tau)\Vert_{\widetilde{B}^{s+\alpha,s+1}} \dd \tau + \int_0^T \Vert u(\tau)\Vert_{\dot{B}^{s+\alpha}_{2,1}} \dd \tau\\
  &\leq C e^{C U(T)} \Big(\Vert \sigma_0\Vert_{ \widetilde{B}^{s,s+\alpha-1}}
  +\Vert u_0\Vert_{\dot{B}^{s}_{2,1}}+ \Vert F\Vert_{L^1_T(\widetilde{B}^{s,s+\alpha-1})}
  + \Vert G\Vert_{ L^1_T(\dot{B}^{s}_{2,1})} \Big),
\end{aligned}
\end{equation*}
with
$U(T):=\int_0^T \Vert u(\tau)\Vert_{\dot{B}^{N/2+1}_{2,1}}\dd \tau$.
Owing to Lemmas \ref{lemma.Bilinear} and \ref{cor.K-P.Besov}, and using \eqref{eq:comp-Bes1} and \eqref{eq:sig-u-ubdd}, we infer that
\begin{equation*}
\begin{aligned}
  \Vert F\Vert_{L^{1}_T(\widetilde{B}^{s,s+\alpha-1})}
  &\lesssim \Vert \sigma\Vert_{L^{\infty}_{T}(\widetilde{B}^{s,s+\alpha-1})}\Vert u\Vert_{L^{1}_{T}(\dot{B}^{\frac{N}{2}+1}_{2,1})}+\Vert \sigma\Vert_{L^{\infty}_{T}(\widetilde{B}^{\frac{N}{2}+1-\alpha,\frac{N}{2}})}\Vert u\Vert_{L^{1}_{T}(\dot{B}^{s+\alpha}_{2,1})}\\
  &\leq C \varepsilon'\big(\Vert \sigma\Vert_{L^{\infty}_{T}(\widetilde{B}^{s,s+\alpha-1})}+\Vert u\Vert_{L^{1}_{T}(\dot{B}^{s+\alpha}_{2,1})}\big),
\end{aligned}
\end{equation*}
and
\begin{equation*}
\begin{aligned}
  \Vert G\Vert_{L^1_T(\dot{B}^{s}_{2,1})}
  &\lesssim \Vert u\Vert_{L^1_T(\dot{B}^{\frac{N}{2}+1}_{2,1})} \Vert u \Vert_{L^\infty_T(\dot{B}^{s}_{2,1})}+\Vert u\Vert_{L^1_T(\dot{B}^{s+\alpha}_{2,1})} \Vert \sigma \Vert_{L^\infty_T(\dot{B}^{\frac{N}{2}}_{2,1})}+\Vert u\Vert_{L^1_T(\dot{B}^{\frac{N}{2}+1}_{2,1})} \Vert \sigma \Vert_{L^\infty_T(\dot{B}^{s+\alpha-1}_{2,1})}\\
  &\leq C \varepsilon'\big(\Vert u \Vert_{L^\infty_T(\dot{B}^{s}_{2,1})}+\Vert u\Vert_{L^1_T(\dot{B}^{s+\alpha}_{2,1})}
  + \Vert \sigma \Vert_{L^\infty_T(\widetilde{B}^{s,s+\alpha-1})} \big).
\end{aligned}
\end{equation*}
Since $\varepsilon'$ is small enough, we obtain the desired estimate \eqref{eq.A.4.0}.
\end{proof}

\section{Asymptotic behavior: proof of Theorem \ref{thm:decay}}\label{decay}

This section aims at proving Theorem \ref{thm:decay}. 
We shall mainly prove the following asymptotic behavior of the global solution $(\sigma,u)$ for system \eqref{eq.EAsigma.nl}.
\begin{proposition}\label{th.decay.intro}
Let $1<\alpha<2$. Assume that $(\sigma,u)$ is a global solution of system \eqref{eq.EAsigma.nl} such that
\begin{equation}\label{eq:sig-u-bdd2}
\begin{split}
  \sigma\in \widetilde{L}^{\infty}(\mathbb{R}^{+};\widetilde{B}^{\frac{N}{2}+1-\alpha,\frac{N}{2}}),\quad 
  u\in \widetilde{L}^{\infty}(\mathbb{R}^{+};\dot{B}^{\frac{N}{2}+1-\alpha}_{2,1})\cap L^{1}(\mathbb{R}^{+};{\dot{B}^{\frac{N}{2}+1}_{2,1}}).
\end{split}
\end{equation}
Then we have for every $0< s< 1- \frac{1}{\alpha}$,
\begin{equation}\label{eq.Decay.X.t+1}
  \Vert (\sigma,u)^\ell(t)\Vert_{\dot{B}^{\frac{N}{2}+1-\alpha+s\alpha}_{2,1}}+\Vert \sigma^h(t)\Vert_{ \dot{B}^{\frac{N}{2}}_{2,1}}
  +\Vert u^h(t)\Vert_{\dot{B}^{\frac{N}{2}+1-\alpha}_{2,1}} \le C(1+t)^{-s}.
\end{equation}
where $C$ depends on the norms of $(\sigma,u)$ in \eqref{eq:sig-u-bdd2}.
Besides, we have
\begin{equation}\label{eq.asymptotic}
  \lim_{t\to \infty} \Big(\Vert \sigma(t)\Vert_{\widetilde{B}^{\frac{N}{2}+1-\alpha,\frac{N}{2}}}+\Vert u(t)\Vert_{\dot{B}^{\frac{N}{2}+1-\alpha}_{2,1}}\Big) = 0.
\end{equation}
If we assume, in addition, $(\sigma_0,u_0)^\ell\in \dot{B}^{-s_0}_{2,\infty}(\R^N)$ with $s_0\in (\alpha-\frac{N}{2}-1,\frac{N}{2})$,
then for all $s_1\in [-s_0, \frac{N}{2}+1-\alpha]$,
\begin{equation}\label{eq.decay}
  \Vert \sigma(t)\Vert_{\widetilde{B}^{s_1,\frac{N}{2}}}+\Vert u(t)\Vert_{\widetilde{B}^{s_1,\frac{N}{2}+1-\alpha}} \le C(1+t)^{- \frac{s_1 + s_0}{\alpha}}.
\end{equation}
\end{proposition}

With Proposition \ref{th.decay.intro} at our disposal, we can finish the proof of Theorem \ref{thm:decay}. Indeed, under condition \eqref{eq:rho>0}, the Euler-alignment system \eqref{eq.EA} is equivalent to \eqref{eq.EAsigma.nl}. Recall $\rho-1=h(\sigma)$. Since both $h$ and $h^{-1}$ are smooth, Proposition \ref{lemma.composition} implies that the regularity on $\rho$ and $\sigma$ are equivalent. Hence, the uniform bound \eqref{eq:rho-u-bdd} implies \eqref{eq:sig-u-bdd2}, and the estimates \eqref{eq.Decay.X.t+1}--\eqref{eq.decay} lead to \eqref{eq:rho-u-decay1}--\eqref{eq:rho-u-decay3}.

Our task remains to prove Proposition \ref{th.decay.intro}.
Let us denote
\begin{align}\label{eq:Xj}
  X_j(t) : = Y_j(t) + \Vert \dot\Delta_j\mathbb{P}u(t)\Vert_{L^2} \approx
  \begin{cases}
    \Vert\dot{\Delta}_j \sigma(t)\Vert_{L^2} + \Vert \dot{\Delta}_j u(t)\Vert_{L^2} ,& \quad \textrm{for}\;\; j\leq j_0, \\
    \Vert \Lambda^{\alpha-1} \dot{\Delta}_j \sigma(t)\Vert_{L^2} + \Vert \dot{\Delta}_j u(t)\Vert_{L^2}, & \quad \textrm{for}\;\; j> j_0.
  \end{cases}
\end{align}
Recalling the estimates \eqref{eq.Yj.pro.low}, \eqref{eq.Yj.pro.high} and \eqref{eq.Puj.pro}, we have
\begin{equation}\label{eq.Decay.xj}
\left\{
\begin{aligned}
  &\frac{d}{dt}X_j+ \bar{\mu} 2^{j\alpha}X_j\le C \Big(\Vert f_j\Vert_{L^2}+\Vert g_j\Vert_{L^2} + \Vert \widetilde{g}_j\Vert_{L^2}+
  \Vert \nabla \dot S_{j-1}u\Vert_{L^{\infty}}X_j\Big), & \; \textrm{for}\;j\le j_0,\\
  &\frac{d}{dt}X_j+\mu_h X_j\le C \Big(2^{j(\alpha-1)}\Vert f_j \Vert_{L^2} + \Vert g_j\Vert_{L^2} + \Vert \widetilde{g}_j\Vert_{L^2}
  + \Vert \nabla \dot S_{j-1}u\Vert_{L^\infty} X_j \Big),& \; \textrm{for}\;j>j_0,
\end{aligned}
\right.
\end{equation}
where $\bar{\mu} := \frac{\delta\mu}{8}$, $\mu_h :=\min(\bar{\nu}2^{j_0(2-{\alpha})},\mu2^{j_0 \alpha })$ and $f_j$, $g_j$, $\widetilde{g}_j$ are defined by \eqref{eq:fj}-\eqref{eq:tild-gj}
with $v=u$, and
\begin{equation}\label{eq:FGd}
\begin{aligned}
  &F=-(\gamma-1)\sigma\Div u-T_{\nabla\sigma}\cdot u-R(u,{\nabla\sigma}),\\
  &G= -\mu\Lambda^\alpha (uh(\sigma)) + \mu\, u\Lambda^\alpha h(\sigma)-T_{\nabla {u}}\cdot u-R(u,\nabla {u}).
\end{aligned}
\end{equation}
Note that $X_j$ for low-frequency part $j\le j_0$ has a dissipation effect analogous to the fractional heat operator,
while for high-frequency part $j>j_0$ it has a damping effect.
Thus one may expect that $(\sigma,u)$ altogether will have a polynomial decay by developing the dissipation/damping effect.
In the sequel we will treat the low-frequency part and high-frequency part separately to show the desired decay estimates.

Before proceeding forward, we introduce the following notations: for $-s_0 \leq  \bar{s}\leq \frac{N}{2}+1-\alpha $ and $s\geq 0$,
\begin{align}\label{eq:Zl}
  Z_{s,\bar{s}}^\ell(t) := t^s \Vert (\sigma,u)^\ell(t)\Vert_{\dot{B}^{\bar{s}+s\alpha}_{2,1}}
  = t^s \sum_{j\leq j_0} 2^{j(\bar{s}+s\alpha)} \|(\dot\Delta_j \sigma, \dot \Delta_j u)(t)\|_{L^2}
  = t^s \sum_{j\leq j_0} 2^{j(\bar{s}+ s\alpha)} X_j(t),
\end{align}
\begin{align}\label{eq:Zh}
  Z_s^h(t) := t^s\Big(\Vert \sigma^h(t)\Vert_{ \dot{B}^{\frac{N}{2}}_{2,1}} + \Vert u^h(t)\Vert_{\dot{B}^{\frac{N}{2}+1-\alpha}_{2,1}}\Big)
  = t^s \sum_{j > j_0} \Big(2^{j\frac{N}{2}} \|\dot \Delta_j \sigma\|_{L^2} + 2^{j(\frac{N}{2}+1-\alpha )} \|\dot \Delta_j u\|_{L^2} \Big),
\end{align}
and
\begin{align}\label{eq:Zt}
  Z_{s,\bar{s}}(t) := Z_{s,\bar{s}}^\ell(t) + Z_s^h(t).
\end{align}

\begin{remark}\label{th.Z.lh}
Notice that for $\bar{s}+s\alpha\leq \frac{N}{2}$,
\begin{equation}\label{eq.Z.sig.l+h}
  t^s \sum_{j\in\mathbb{Z}} 2^{j(\bar{s}+s\alpha)} \Vert \dot\Delta_j \sigma \Vert_{L^2}
  +  t^s \sum_{j\in\mathbb{Z}} 2^{j\frac{N}{2}} \Vert \dot \Delta_j\sigma \Vert_{L^2}  \le C Z_{s,\bar{s}}(t),
\end{equation}
and for $\bar{s}+s\alpha\leq \frac{N}{2} + 1 - \alpha$,
\begin{equation}\label{eq.Z.u.l+h}
  t^s \sum_{j\in \mathbb{Z}}2^{j(\bar{s} + s \alpha)} \Vert \dot \Delta_j u\Vert_{L^2}
  + t^s \sum_{j\in\mathbb{Z}}2^{j(\frac{N}{2}+1-\alpha)} \Vert \dot \Delta_j u\Vert_{L^2}\le C Z_{s,\bar{s}}(t).
\end{equation}
\end{remark}

\subsection{Low-frequency estimates}
From \eqref{eq.Decay.xj}, it is reasonable to expect that $(\sigma^\ell,u^\ell)$
present the polynomial decay estimate in suitable functional spaces.
\begin{lemma}\label{th.low-fre-est}
Under the assumption of Proposition \ref{th.decay.intro}, if $-s_0\leq \bar{s} \leq \frac{N}{2}+1-\alpha $ and $s\in[0,1)$ satisfy
$\bar{s}+s\alpha < \frac{N}{2}$,
then we have
\begin{equation}\label{eq.Decay.X.l.result}
\begin{aligned}
  Z_{s,\bar{s}}^\ell(t)\le C \Vert (\sigma_0,u_0) \Vert^\ell_{ \dot{B}^{\bar{s}}_{2,1}} + C\int_0^t \Phi(t,\tau) Z_{s,\bar{s}}(\tau) \dd \tau,
\end{aligned}
\end{equation}
where
\begin{equation}\label{eq.Phi.decay}
\begin{aligned}
  \Phi(t,\tau)& :=
  \sum_{j\le j_0} \sum_{j'\le j+4}2^{j'(\frac{N}{2} +1)} \Vert \dot\Delta_{j'}u(\tau) \Vert_{L^2}\psi_j(t,\tau)
  \\&\quad
  +\sum_{j\le j_0}\sum_{j'>j-4}2^{(j-j')(\frac{N}{2}+\bar{s} + \delta s \alpha)}2^{j'(\frac{N}{2}+1)} \Vert \dot\Delta_{j'}u(\tau) \Vert_{L^2}\psi_j(t,\tau) ,
\end{aligned}
\end{equation}
with
\begin{align}\label{eq:psi-j}
  \psi_j (t,\tau) := e^{-\bar{\mu} 2^{j\alpha}(t-\tau)}t^s\tau^{-s}
\end{align}
and
\begin{equation}\label{def:del}
\delta:=
\begin{cases}
  1, & \quad \textrm{if}\;\; \bar{s}+ s\alpha \leq \frac{N}{2} +1 -\alpha, \\
  \frac{N/2 -\bar{s} -\alpha +1}{s\alpha}, & \quad \textrm{if}\;\; \bar{s}  + s\alpha > \frac{N}{2} + 1-\alpha.
\end{cases}
\end{equation}
\end{lemma}
\begin{proof}[Proof of Lemma \ref{th.low-fre-est}]
By multiplying the first inequality of \eqref{eq.Decay.xj} with $e^{\bar{\mu} 2^{j\alpha}t}$ and integrating over the time interval $[0,t]$,
we infer that
\begin{equation*}
\begin{aligned}
  X_j(t)\le e^{-\bar{\mu} 2^{j\alpha}t} X_j(0) + C\int_0^t e^{-\bar{\mu} 2^{j\alpha}(t-\tau)} R_j(\tau) \dd \tau,
\end{aligned}
\end{equation*}
with
\begin{align}\label{eq:Rj1}
  R_j(t) := \Vert f_j(t)\Vert_{L^2}+\Vert g_j(t)\Vert_{L^2} + \Vert \widetilde{g}_j(t)\Vert_{L^2} +
  \Vert \nabla \dot S_{j-1}u(t)\Vert_{L^{\infty}}X_j(t).
\end{align}
Using the fact that
$\sup\limits_{t\ge 0}\sum\limits_{j\in\mathbb{Z}} 2^{\alpha js} t^s e^{-ct2^{\alpha j}}<+\infty $, we have that for every $s\in [0,1)$,
\begin{equation*}
  t^s X_j(t) \le C 2^{-js\alpha } X_j(0) + C\int_0^t \psi_j (t,\tau)\tau^sR_j(\tau) \dd \tau.
\end{equation*}
Multiplying both sides of the above equation with $2^{j(\bar{s}+s\alpha)}$ and taking the $\ell^1$-norm with respect to $j \leq j_0$ lead to
\begin{equation}\label{eq.Decay.X.l}
\begin{aligned}
  Z_{s,\bar{s}}^\ell (t)\le C\Vert (\sigma^\ell_0,u^\ell_0)\Vert_{ \dot{B}^{\bar{s}}_{2,1}}
  & + C \int_0^t \tau^s \sum_{j\le j_0}\psi_j(t,\tau) 2^{j(\bar{s}+s\alpha)}R_j(\tau) \dd \tau,
\end{aligned}
\end{equation}
where we have abbreviated $\psi_j(t,\tau)$ as $\psi_j$.

Next we treat the integral term of the inequality \eqref{eq.Decay.X.l}.
For the last term in $R_j$ given by \eqref{eq:Rj1}, by H\"older's inequality and \eqref{eq:Zl}, we see that
\begin{equation}\label{eq.5.1.r1}
\begin{aligned}
  \tau^s \sum_{j\le j_0} 2^{j(\bar{s}+s\alpha)}\psi_j \Vert \nabla \dot S_{j-1} u(\tau)\Vert_{L^{\infty}}X_j(\tau)
  \le C Z_{s,\bar{s}}(\tau)  \Big( \sup_{j\le j_0} \|\nabla \dot S_{j-1} u(\tau)\|_{L^\infty} \psi_j  \Big).
\end{aligned}
\end{equation}
We turn to the estimation of terms in \eqref{eq:Rj1} containing $f_j$, $g_j$ and $\widetilde{g}_j$.
Noting that $f_j$ is given by
\begin{align}\label{def:fj}
  f_j =(\gamma-1)\dot{\Delta}_j(\sigma\Div u) + \dot S_{j-1} u\cdot\nabla \dot{\Delta}_j \sigma - \dot{\Delta}_j ( u\cdot\nabla \sigma),
\end{align}
and using Bony's decomposition, we have the following splitting $f_j = \sum_{k=1}^6 f_j^k$ with
\begin{align*}
  & f_j^1 := (\gamma-1) \dot \Delta_j (T_\sigma \Div u), \quad f_j^2:= (\gamma-1) \dot \Delta_j (T_{\Div u}\sigma),
  \quad f_j^3 := (\gamma-1) \dot \Delta_j (R(\Div u,\sigma)) , \\
  & f_j^4  := S_{j-1} u\cdot\nabla \dot{\Delta}_j \sigma - \dot{\Delta}_j ( T_u \cdot\nabla\sigma), \quad f_j^5 := - \dot{\Delta}_j (T_{\nabla \sigma}\cdot u), \quad  f_j^6 := - \dot{\Delta}_j R(u\cdot, \nabla \sigma).
\end{align*}
Thanks to inequalities \eqref{eq.weight.1} and \eqref{eq.Z.sig.l+h}, we have that for every $\bar{s}+s\alpha < \frac{N}{2}$,
\begin{align}
  & \quad \tau^s\sum_{j\le j_0} \psi_j 2^{j(\bar{s}+s\alpha)} \big( \Vert f_j^1 \Vert_{L^2} + \Vert f_j^5 \Vert_{L^2} \big)  \nonumber \\
  & \lesssim \Big( \tau^s \sup_{j'\leq j_0+4} 2^{j'(\bar{s} +s\alpha)} \|\dot\Delta_{j'}\sigma\|_{L^2}  \Big)
  \Big( \sum_{j\le j_0}\sum_{\vert j-j'\vert \le 4}2^{j'(\frac{N}{2}+1)}\Vert \dot\Delta_{j'}u \Vert_{L^2}\psi_j \Big) \nonumber \\
  & \lesssim  Z_{s,\bar{s}}(\tau)\sum_{j\le j_0}\sum_{\vert j-j'\vert \le 4}2^{j'(\frac{N}{2}+1)}\Vert \dot\Delta_{j'}u \Vert_{L^2}\psi_j.
\end{align}
Using \eqref{eq.weight.2} and \eqref{eq.Z.sig.l+h}, we deduce that
\begin{align}
  \tau^s\sum_{j\le j_0} \psi_j 2^{j(\bar{s}+s\alpha)}\Vert f_j^2\Vert_{L^2}
  & \lesssim \Big(\tau^s \sum_{j'\leq j_0+4} 2^{j'(\bar{s} +s\alpha)} \|\dot\Delta_{j'} \sigma\|_{L^2} \Big)
  \sum_{j\le j_0}\sum_{j'\le j+4}2^{j'}\Vert \dot\Delta_{j'}u \Vert_{L^\infty}\psi_j   \nonumber \\
  &\lesssim  Z_{s,\bar{s}}(\tau)\sum_{j\le j_0}\sum_{j'\le j+4}2^{j'}\Vert \dot\Delta_{j'}u \Vert_{L^\infty}\psi_j .
\end{align}
By virtue of \eqref{eq:LP-rem} and \eqref{eq.Z.sig.l+h}, we see that for every $- \frac{N}{2} < \bar{s}+s \alpha < \frac{N}{2}$,
\begin{align}
  & \quad \tau^s \sum_{j\le j_0} \psi_j 2^{j(\bar{s}+s\alpha)} \big( \Vert f_j^3\Vert_{L^2} + \Vert f_j^6 \Vert_{L^2} \big)  \nonumber \\
  & \lesssim \Big(\tau^s \sup_{j'\in\mathbb{Z}} 2^{j'(\bar{s}+s\alpha)} \|\dot \Delta_{j'} \sigma\|_{L^2}  \Big)
  \Big(\sum_{j\le j_0}\sum_{j'> j-4}2^{(j-j')(\frac{N}{2}+\bar{s}+s\alpha)}2^{j'(\frac{N}{2}+1)}\Vert \dot\Delta_{j'} u
  \Vert_{L^2}\psi_j\Big)\nonumber  \\
  & \lesssim Z_{s,\bar{s}}(\tau)\sum_{j\le j_0}\sum_{j'> j-4}2^{(j-j')(\frac{N}{2}+\bar{s}+s\alpha)}2^{j'(\frac{N}{2}+1)}
  \Vert \dot\Delta_{j'} u \Vert_{L^2}\psi_j.
\end{align}
Applying Lemma \ref{lemma.commutator_estimate} to the term $f_j^4$ yields
\begin{equation*}
  \Vert f_j^4 \Vert_{L^2}
  \le C2^{-j}\sum_{j',j''=j-4}^{j+4} \big(\Vert \nabla \dot{\Delta}_{j''}u\Vert_{L^{\infty}}+\Vert \nabla \dot{S}_{j-1}u\Vert_{L^{\infty}}\big)\Vert \dot{\Delta}_{j'}\nabla \sigma\Vert_{L^2},
\end{equation*}
and in combination with inequality \eqref{eq.Z.sig.l+h}, we obtain
\begin{align}
  & \quad \tau^{s}\sum_{j\le j_0} \psi_j 2^{j(\bar{s}+s\alpha)} \Vert f_j^4 \Vert_{L^2} \nonumber \\
  & \lesssim Z_{s,\bar{s}}(\tau)\Big(\sum_{j\le j_0}\sum_{\vert j-j'\vert \le 4}2^{j'(\frac{N}{2}+1)}\Vert \Delta_{j'}u \Vert_{L^2}\psi_j
  + \sum_{j\le j_0}\sum_{j'\le j+4}2^{j'} \Vert \Delta_{j'} u \Vert_{L^\infty}\psi_j \Big).
\end{align}

For the terms $g_j$ and $\tilde{g}_j$ given by \eqref{eq:gj}-\eqref{eq:tild-gj} and \eqref{eq:FGd}, note that they both can be expressed as
\begin{align}\label{exp:gj}
  -\mu \mathcal{P}\dot{\Delta}_j \big(\Lambda^\alpha (uh(\sigma))- u\Lambda^\alpha h(\sigma)\big)
  + S_{j-1} u \cdot\nabla \mathcal{P} \dot{\Delta}_j u - \mathcal{P}\dot{\Delta}_j (u \cdot \nabla u),
\end{align}
where $\mathcal{P}$ is composed of smooth zero-order pseudo-differential operators, then $g_j$ and $\widetilde{g}_j$ both have the following decomposition $\sum_{k=1}^9 g_j^k$, with
\begin{equation*}
\begin{aligned}
  & g_j^1 : = - \mu \mathcal{P} \dot \Delta_j  \Lambda^\alpha (T_{u}h(\sigma)),\quad g_j^2 := - \mu \mathcal{P} \dot \Delta_j \Lambda^\alpha (T_{h(\sigma)}u),\quad
  g_j^3 : = - \mu \mathcal{P} \dot \Delta_j \Lambda^\alpha R(h(\sigma),u),\\
  & g_j^4 : = \mu \mathcal{P} \dot \Delta_j  (T_{u}\Lambda^\alpha h(\sigma)),\quad
  g_j^5 := \mu \mathcal{P} \dot \Delta_j T_{\Lambda^\alpha h(\sigma)}u,\quad g_j^6 := \mu \mathcal{P} \dot \Delta_j  R(\Lambda^\alpha h(\sigma),u), \\
  &  g_j^7 := S_{j-1} u\cdot\nabla \dot{\Delta}_j\mathcal{P}u - \mathcal{P}\dot{\Delta}_j ( T_u \cdot\nabla u),\quad g_j^8:=  \mathcal{P}\dot{\Delta}_j (T_{\nabla u}\cdot u),\quad g_j^9:=  \mathcal{P}\dot{\Delta}_j R(u\cdot,\nabla u).
\end{aligned}
\end{equation*}
For the terms $g^1_j + g^4_j$ and $g_j^8$,  
taking advantage of inequalities \eqref{eq:Commu-est-Lp}, \eqref{eq.Z.sig.l+h} and Lemma \ref{lemma.composition}, we obtain
\begin{align}
  & \quad \tau^s\sum_{j\le j_0}\psi_j 2^{j(\bar{s}+s\alpha)} \Vert g^1_j+g^4_j \Vert_{L^2}
  \lesssim \tau^s \sum_{j\leq j_0} \psi_j 2^{j(\bar{s}+s\alpha)} \|\dot \Delta_j\big(\Lambda^{\alpha}(T_u  h(\sigma))-T_u \Lambda^{\alpha}h(\sigma)\big) \|_{L^2}  \nonumber \\
  & \lesssim \tau^s \sum_{j\leq j_0} \psi_j 2^{j(\bar{s}+s\alpha)} \sum_{|j'-j|\leq 4}
  \|\Lambda^\alpha (\dot S_{j'-1} u\,\dot\Delta_j h(\sigma))
  - \dot S_{j'-1} u\, \Lambda^\alpha \dot \Delta_{j'} h(\sigma)\|_{L^2} \nonumber \\
  & \lesssim \tau^s \sum_{j\leq j_0} \psi_j 2^{j(\bar{s}+s\alpha)} \sum_{|j'-j|\leq 4}
  \Big( \|\nabla \dot S_{j'-1} u\|_{L^\infty} \|\Lambda^{\alpha-1} \dot \Delta_{j'}h(\sigma)\|_{L^2}
  + \|\Lambda^\alpha \dot S_{j-1} u\|_{L^\infty} \|\dot \Delta_{j'} h(\sigma)\|_{L^2} \Big)\nonumber \\
  & \lesssim_{j_0} \Big(\tau^s \sup_{j'\leq j_0+4} 2^{j'(\bar{s}+s\alpha)} \|\dot\Delta_{j'} h(\sigma)\|_{L^2} \Big) \Big(\sum_{j\le j_0}\sum_{j'\le j+4} 2^{j'}  \Vert \dot \Delta_{j'}u \Vert_{L^\infty} \psi_j \Big) \nonumber \\
  & \lesssim Z_{s,\bar{s}}(\tau)\sum_{j\le j_0}\sum_{j'\le j+4}2^{j'} \Vert \dot \Delta_{j'}u \Vert_{L^\infty}\psi_j ,
\end{align}
and
\begin{align}
  \tau^s \sum_{j\le j_0} \psi_j 2^{j(\bar{s}+s\alpha)} \Vert g_j^8 \Vert_{L^2}
  & \lesssim \Big(\tau^s \sup_{j'\leq j_0 + 4} 2^{j'(\bar{s}+ s\alpha)} \|\dot \Delta_{j'} u \|_{L^2} \Big)
  \Big(\sum_{j\le j_0}\sum_{j'\le j+4}2^{j'} \Vert \dot \Delta_{j'}u \Vert_{L^\infty}\psi_j  \Big) \nonumber \\
  & \lesssim Z_{s,\bar{s}}(\tau)\sum_{j\le j_0}\sum_{j'\le j+4}2^{j'} \Vert \dot\Delta_{j'}u \Vert_{L^\infty} \psi_j .
\end{align}
For the terms $g_j^2$ and $g_j^5$, using \eqref{eq.weight.1}, \eqref{eq.Z.sig.l+h} and Lemma \ref{lemma.composition},
we deduce that
\begin{align}
  & \quad \tau^s \sum_{j\le j_0}\psi_j 2^{j(\bar{s}+s\alpha)} \Vert g_j^2 \Vert_{L^2}
  \lesssim_{j_0} \tau^s \sum_{j\leq j_0} \psi_j 2^{j(\bar{s}+1 + s\alpha)} \|\dot\Delta_j (T_{h(\sigma)}u)\|_{L^2} \nonumber  \\
  & \lesssim \Big(\tau^s \sup_{j'\leq j_0+4} 2^{j'(\bar{s}+s\alpha)} \|\dot\Delta_{j'} h(\sigma)\|_{L^2} \Big) \Big(\sum_{j\le j_0}\sum_{\vert j-j'\vert \le 4} 2^{j'(\frac{N}{2}+1)} \Vert \dot \Delta_{j'}u \Vert_{L^2} \psi_j \Big) \nonumber \\
  & \lesssim Z_{s,\bar{s}}(\tau)\sum_{j\le j_0}\sum_{\vert j-j'\vert \le 4} 2^{j'(\frac{N}{2}+1)} \Vert \dot\Delta_{j'}u \Vert_{L^2} \psi_j,
\end{align}
and
\begin{align}
  & \quad \tau^s \sum_{j\le j_0} \psi_j 2^{j(\bar{s}+s\alpha)} \|g_j^5\|_{L^2} \nonumber \\
  & \lesssim \Big(\tau^s \sup_{j'\leq j_0 + 4} 2^{j'(\bar{s}-\frac{N}{2}-1 + s\alpha)} \|\dot \Delta_{j'} \Lambda^\alpha h(\sigma)\|_{L^\infty} \Big)
  \Big(\sum_{j\le j_0}\sum_{\vert j-j'\vert \le 4}2^{j'(\frac{N}{2}+1)} \Vert \dot \Delta_{j'}u \Vert_{L^2}\psi_j \Big) \nonumber  \\
  & \lesssim_{j_0} \Big(\tau^s \sup_{j'\leq j_0 + 4} 2^{j'(\bar{s}+ s\alpha)} \|\dot \Delta_{j'} h(\sigma)\|_{L^2} \Big)
  \Big(\sum_{j\le j_0}\sum_{\vert j-j'\vert \le 4}2^{j'(\frac{N}{2}+1)} \Vert \dot \Delta_{j'}u \Vert_{L^2}\psi_j \Big) \nonumber \\
  & \lesssim Z_{s,\bar{s}}(\tau)\sum_{j\le j_0}\sum_{\vert j-j'\vert \le 4}2^{j'(\frac{N}{2}+1)} \Vert \dot \Delta_{j'}u \Vert_{L^2}\psi_j.
\end{align}
For the terms $g_j^3$ and $g_j^6$, 
thanks to inequality \eqref{eq:LP-rem}, \eqref{eq.Z.sig.l+h} and Lemma \ref{lemma.composition} again, we infer that
\begin{align}
  & \quad \tau^s \sum_{j\le j_0}\psi_j 2^{j(\bar{s}+s\alpha)} \Vert g_j^3 \Vert_{L^2}
  \lesssim_{j_0} \tau^s \sum_{j\leq j_0} \psi_j 2^{j(\bar{s}+1 + s\alpha)} \|\dot\Delta_j R(h(\sigma),u)\|_{L^2} \nonumber  \\
  & \lesssim  \Big(\tau^s \sup_{j'\in \mathbb{Z}} 2^{j'(\bar{s} + s\alpha )} \|\dot \Delta_{j'} h(\sigma)\|_{L^2} \Big) \nonumber
     \Big(\sum_{j\le j_0} \sum_{j'>j-4}2^{(j-j')(\frac{N}{2}+1+\bar{s}+s\alpha )} 2^{j'(\frac{N}{2}+1)} \Vert \dot\Delta_{j'} u \Vert_{L^2}\psi_j \Big) \nonumber  \\
  & \lesssim_{j_0} Z_{s,\bar{s}}(\tau) \sum_{j\le j_0} \sum_{j'>j-4}2^{(j-j')(\frac{N}{2}+1+\bar{s}+s\alpha )} 2^{j'(\frac{N}{2}+1)} \Vert \dot\Delta_{j'} u \Vert_{L^2}\psi_j,
\end{align}
and
\begin{align}
  & \quad \tau^s \sum_{j\le j_0}\psi_j 2^{j(\bar{s}+s\alpha)} \Vert g_j^6 \Vert_{L^2} \nonumber \\
  & \lesssim  \Big(\tau^s \sup_{j'\in \mathbb{Z}} 2^{j'(\bar{s}-1+\alpha+ \delta s\alpha )} \|\dot \Delta_{j'} h(\sigma)\|_{L^2} \Big)
  \Big(\sum_{j\le j_0} \sum_{j'>j-4}2^{(j-j')(\frac{N}{2}+\bar{s}+\delta s\alpha )}
  2^{j'(\frac{N}{2}+1)} \Vert \dot\Delta_{j'} u \Vert_{L^2}\psi_j \Big) \nonumber  \\
  & \lesssim_{j_0} Z_{s,\bar{s}}(\tau)\sum_{j\le j_0}\sum_{j'>j-4} 2^{(j-j')(\frac{N}{2}+\bar{s}+ \delta s\alpha )}
  2^{j'(\frac{N}{2}+1)} \Vert \dot\Delta_{j'} u \Vert_{L^2}\psi_j ,
\end{align}
where $\delta$ is given by \eqref{def:del} so that $\bar{s}-1+\alpha+ \delta s\alpha \leq \frac{N}{2}$ (to fit the norm $Z_{s,\bar{s}}(\tau)$).
For the term $g_j^7$, 
with the help of Lemma \ref{lemma.commutator_estimate}, we find
\begin{align}\label{eq.5.1.rl}
  &\quad \tau^s\sum_{j\le j_0} \psi_j 2^{j(\bar{s}+s\alpha)} \Vert g_j^7 \Vert_{L^2} \nonumber \\
  & \lesssim \tau^s \sum_{j\leq j_0} \psi_j 2^{j(\bar{s}+s\alpha)}
  \Big(2^{-j}\sum_{j',j''=j-4}^{j+4} \big(\Vert \nabla \dot{\Delta}_{j''}u\Vert_{L^{\infty}} + \Vert \nabla \dot{S}_{j-1}u\Vert_{L^\infty}\big)\Vert \dot{\Delta}_{j'}\nabla u\Vert_{L^2} \Big) \nonumber \\
  &\lesssim \Big(\tau^s \sum_{j' \leq j_0+ 4} 2^{j'(\bar{s} +s\alpha )} \|\dot\Delta_{j'} u \|_{L^2} \Big)
  \Big(\sum_{j\le j_0}\sum_{\vert j-j'\vert \le 4}2^{j'(\frac{N}{2}+1)}\Vert \dot \Delta_{j'} u \Vert_{L^2}\psi_j
  + \sum_{j\le j_0}\sum_{j'\le j+4} 2^{j'} \Vert \dot\Delta_{j'}u \Vert_{L^\infty}\psi_j \Big) \nonumber \\
  &\lesssim Z_{s,\bar{s}}(\tau) \sum_{j\le j_0}\sum_{j'\le j+4} 2^{j'(\frac{N}{2}+1)} \Vert \dot\Delta_{j'}u \Vert_{L^2}\psi_j .
\end{align}
For the last term $g_j^9$, thanks to the spectral support property of dyadic operators and \eqref{eq.Z.u.l+h}, we see that
\begin{align}\label{eq.5.1.r.end}
  &\quad\tau^s \sum_{j\le j_0}\psi_j 2^{j(\bar{s}+s\alpha)} \Vert g^9_j  \Vert_{L^2}\nonumber \\
  & \lesssim \tau^s \sum_{j\leq j_0} \psi_j 2^{j(\bar{s} + s\alpha)} \Big( \sum_{j'\geq j-3} 2^{j\frac{N}{2}} \|\dot\Delta_{j'} u\|_{L^2}
  2^{j'} \|\widetilde{\dot \Delta}_j u\|_{L^2} \Big) \nonumber \\
  &\lesssim  \Big(\tau^s \sup_{j'\in \mathbb{Z}} 2^{j'(\bar{s} +\delta s\alpha + (1-\delta) s\alpha 1_{\{j'\leq j_0\}})} \|\dot \Delta_{j'}u \|_{L^2} \Big)\times \nonumber \\
  &\quad \times \Big(\sum_{j\le j_0} \sum_{j'>j-4}2^{(j-j')(\frac{N}{2}+\bar{s}+ \delta s\alpha + (1-\delta)s \alpha 1_{\{j'\leq j_0\}})} 2^{j'(\frac{N}{2}+1)}
  \Vert \dot\Delta_{j'} u \Vert_{L^2}\psi_j \Big) \nonumber  \\
  & \lesssim Z_{s,\bar{s}}(\tau) \sum_{j\le j_0}\sum_{j'>j-4}
  2^{(j-j')(\frac{N}{2}+\bar{s}+ s\alpha)}2^{j'(\frac{N}{2}+1)} \Vert \dot\Delta_{j'}u \Vert_{L^2}\psi_j,
\end{align}
where in the above we have used the fact that $2^{j(1-\delta)1_{\{j'\geq j_0\}}} \leq C$ for every $j\leq j_0$.

Inserting \eqref{eq.5.1.r1}--\eqref{eq.5.1.r.end} into \eqref{eq.Decay.X.l}, we thus conclude the inequality \eqref{eq.Decay.X.l.result}.
\end{proof}

\begin{lemma}\label{th.low-fre-est-n}
Under the assumption of Proposition \ref{th.decay.intro},
if $\bar{s}=-s_0\in (-\frac{N}{2}, \frac{N}{2}+1-\alpha)$ and $s\in [n,n+1)$ ($n\in \mathbb{N}$) satisfy $-\frac{N}{2} <-s_0+s\alpha\leq \frac{N}{2}+1-\alpha$,
we have
\begin{equation}\label{eq.Decay.X.l.result.s}
\begin{aligned}
  Z^\ell_{s,\bar{s}}(t)\le C Z_{s-1,\bar{s}}^\ell(t) + C\int_0^t  Z_{s,\bar{s}}(\tau) \Psi(t,\tau) \dd \tau,
\end{aligned}
\end{equation}
where
\begin{equation}\label{eq.Psi.decay}
\begin{aligned}
  \Psi(t,\tau) & :=
  \sum_{j\le j_0} \sum_{j'\le j+4}2^{j'(\frac{N}{2} +1)} \Vert \dot\Delta_{j'}u(\tau) \Vert_{L^2} e^{-\bar{\mu} 2^{j\alpha}(t-\tau)} \\ & \quad
  +\sum_{j\le j_0}\sum_{j'>j-4}2^{(j-j')(\frac{N}{2} + \bar{s} + s \alpha)}2^{j'(\frac{N}{2}+1)} \Vert \dot\Delta_{j'}u(\tau) \Vert_{L^2} e^{-\bar{\mu} 2^{j\alpha}(t-\tau)} .
\end{aligned}
\end{equation}
\end{lemma}
\begin{proof}[Proof of Lemma \ref{th.low-fre-est-n}]
From the inequality \eqref{eq.Decay.xj}, we have that for every $j\leq j_0$,
\begin{align}
\frac{d}{dt}\big(t^sX_j(t)\big)+ \bar{\mu} 2^{j\alpha}t^sX_j(t) \le C t^s R_j(t) + s t^{s-1} X_j(t),
\end{align}
where $R_j$ is given by \eqref{eq:Rj1}.
Multiplying the above inequality with $e^{\bar{\mu} 2^{j\alpha}t}$ and integrating over $[0,t]$, we obtain
\begin{equation*}
\begin{aligned}
  t^s X_j(t)\le  \int_0^t e^{-\bar{\mu} 2^{j\alpha}(t-\tau)} \tau^{s-1} X_j(\tau) \dd \tau
  + C\int_0^t e^{-\bar{\mu} 2^{j\alpha}(t-\tau)} \tau^s R_j(\tau) \dd \tau.
\end{aligned}
\end{equation*}
Multiplying the above equation with $2^{j(\bar{s}+s\alpha)}$ and taking the $\ell^1$-norm with respect to $j \leq j_0$,
we find
\begin{equation}\label{eq.Decay.Zs.l}
\begin{aligned}
  Z_{s,\bar{s}}^\ell(t) & \le Z_{s-1,\bar{s}}^\ell(t) 2^{j\alpha}\int_0^t e^{-\bar{\mu} 2^{j\alpha}(t-\tau)} \dd \tau
  + C\int_0^t \tau^s \sum_{j\le j_0}2^{j(\bar{s}+s\alpha)}R_j(\tau)e^{-\bar{\mu} 2^{j\alpha}(t-\tau)}  \dd \tau\\
  &\le CZ_{s-1,\bar{s}}^\ell(t) + C\int_0^t\tau^s \sum_{j\le j_0}2^{j(\bar{s}+s\alpha)}R_j(\tau) e^{-\bar{\mu} 2^{j\alpha}(t-\tau)} \dd \tau.
\end{aligned}
\end{equation}
Similarly as \eqref{eq.Decay.X.l}--\eqref{eq.5.1.r.end}, we can estimate the last term of the above inequality as follows
\begin{equation}\label{eq.Decay.X.l.s}
\begin{aligned}
  \int_0^t\tau^s \sum_{j\le j_0}2^{j(\bar{s}+s\alpha)}R_j(\tau) e^{-\bar{\mu} 2^{j\alpha}(t-\tau)} \dd \tau
  \le C\int_0^t Z_{s,\bar{s}}(\tau)\Psi(t,\tau)\dd \tau.
\end{aligned}
\end{equation}
Inserting \eqref{eq.Decay.X.l.s} into \eqref{eq.Decay.Zs.l}, we obtain the desired inequality \eqref{eq.Decay.X.l.result.s}.
\end{proof}

\subsection{High-frequency estimates}
Since $X_j$ exhibits a damping effect in the second inequality of \eqref{eq.Decay.xj},
one can generally expect to derive an exponential decay for $X_j$. But in order to be coincident with the low-frequency case,
we instead will prove a polynomial decay estimate.
\begin{lemma}\label{th.high-freq-est}
Under the assumption of Proposition \ref{th.decay.intro}, for every $-s_0 \leq \bar{s} \leq \frac{N}{2}+1-\alpha$ and $-\frac{N}{2}< \bar{s} + s \alpha < \frac{N}{2}$,
we have
\begin{equation}\label{eq.Decay.X.h.result}
\begin{aligned}
  Z_s^h(t) \le C\Big(\Vert \sigma_0\Vert_{ \widetilde{B}^{\frac{N}{2}+1-\alpha,\frac{N}{2}}_{2,1}}
  + \Vert u_0 \Vert_{\dot{B}^{\frac{N}{2}+1-\alpha}_{2,1}} \Big) + C\int_0^t \Vert u(\tau)\Vert_{\dot{B}^{\frac{N}{2}+1}_{2,1}} Z_{s,\bar{s}}(\tau)\dd \tau.
\end{aligned}
\end{equation}
\end{lemma}
\begin{proof}[Proof of Lemma \ref{th.high-freq-est}]
Starting from \eqref{eq.Decay.xj} and letting $t_0 = s\mu_h^{-1}$, we have that for every $t\ge t_0$,
\begin{equation*}
\begin{aligned}
  \frac{d}{dt}X_j(t) + st^{-1} X_j(t) \le C \widetilde{R}_j(t),
\end{aligned}
\end{equation*}
with 
\begin{align}\label{eq:tild-Rj}
  \widetilde{R}_j(t) := 2^{j(\alpha-1)}\Vert f_j(t) \Vert_{L^2} + \Vert g_j(t)\Vert_{L^2} + \Vert \widetilde{g}_j(t)\Vert_{L^2}
  + \Vert \nabla \dot S_{j-1}u(t)\Vert_{L^\infty} X_j(t).
\end{align}
Thus, we deduce that
\begin{equation}
  t^s X_j(t) \le t_0^s X_j(t_0) +  C \int_0^t \tau^s \widetilde{R}_j(\tau) \dd \tau,\ \quad \forall t\ge t_0.
\end{equation}
Recalling notations \eqref{eq:Xj} and \eqref{eq:Zh}, we multiply the above inequality by $2^{j(\frac{N}{2}+1-\alpha)}$ and take the $\ell^1$-norm over $j > j_0 $ to get that for every $t\geq t_0$,
\begin{equation}\label{eq:Zh-es1}
  Z_s^h(t) \le C \Big(\Vert \sigma\Vert_{L^\infty_{t_0}(\widetilde{B}^{\frac{N}{2}+1-\alpha,\frac{N}{2}}_{2,1})}
  + \Vert u\Vert_{L^\infty_{t_0}(\dot{B}^{\frac{N}{2}+1-\alpha}_{2,1})} \Big) + C\sum_{j> j_0}2^{j(\frac{N}{2}+1-\alpha)} \int_0^t \tau^s \widetilde{R}_j(\tau) \dd \tau.
\end{equation}
The above inequality also holds when $t\le t_0$: 
\begin{equation*}
\begin{aligned}
  Z_s^h(t) =  t^s \Big(\Vert \sigma^h(t)\Vert_{ \dot{B}^{\frac{N}{2}}_{2,1}} + \Vert u^h(t)\Vert_{\dot{B}^{\frac{N}{2}+1-\alpha}_{2,1}} \Big)
  \le t_0^s \Big(\Vert \sigma\Vert_{L^\infty_{t_0}(\widetilde{B}^{\frac{N}{2}+1-\alpha,\frac{N}{2}}_{2,1})}
  + \Vert u \Vert_{L^\infty_{t_0}(\dot{B}^{\frac{N}{2}+1-\alpha}_{2,1})} \Big).
\end{aligned}
\end{equation*}
Note that in view of \eqref{eq:sig-u-bdd2}, the above right-hand term is under control.

Next, we deal with the term in \eqref{eq:Zh-es1} containing $\widetilde{R}_j$ (given by \eqref{eq:tild-Rj}). The last term in $\widetilde{R}_j$ can be estimated as
\begin{equation*}
\begin{aligned}
  \sum_{j> j_0}\tau^s2^{j(\frac{N}{2}+1-\alpha)} \Vert \nabla \dot S_{j-1} u(\tau)\Vert_{L^{\infty}} X_j(\tau) \le C Z_s^h(\tau) \Vert u(\tau)\Vert_{\dot{B}^{\frac{N}{2}+1}_{2,1}}.
\end{aligned}
\end{equation*}
Recalling that $f_j$ is given by \eqref{def:fj}, and by virtue of Lemmas \ref{lemma.Bilinear} and \ref{lemma.Remark2.103}, we infer that
\begin{equation*}
\begin{aligned}
  &\quad\tau^s\sum_{j> j_0}2^{j\frac{N}{2}} \Vert f_j(\tau)\Vert_{L^2}
 \\ &\le C\tau^s\sum_{j> j_0}2^{j\frac{N}{2}} \Big(\Vert \dot{\Delta}_j(\sigma\Div u)\Vert_{L^2}
  +\Vert S_{j-1} u\cdot\nabla \dot{\Delta}_j \sigma - \dot{\Delta}_j ( u\cdot\nabla \sigma)\Vert_{L^2}+
  \Vert R(u,{\nabla\sigma})\Vert_{L^2}\Big)\\
  &\le C\Vert u(\tau)\Vert_{\dot{B}^{\frac{N}{2}+1}_{2,1}}\tau^s\Vert \sigma(\tau)\Vert_{\dot{B}^{\frac{N}{2}}_{2,1}}
  \le C\Vert u(\tau)\Vert_{\dot{B}^{\frac{N}{2}+1}_{2,1}} Z_{s,\bar{s}}(\tau).
\end{aligned}
\end{equation*}
Noting that $g_j$ and $\widetilde{g}_j$ have the same expression formula \eqref{exp:gj}, and thanks to Lemmas \ref{lemma.Bilinear}, \ref{lemma.composition} and \ref{lemma.commutator_estimate}, we find
\begin{align*}
  & \quad \tau^s\sum_{j> j_0}2^{j(\frac{N}{2}+1-\alpha)} \big( \Vert g_j(\tau)\Vert_{L^2} + \|\widetilde{g}_j(\tau)\|_{L^2} \big) \\
  & \le C\tau^s\sum_{j> j_0}2^{j(\frac{N}{2}+1-\alpha)} \Vert \dot{\Delta}_j \big(\Lambda^\alpha (uh(\sigma))- u\Lambda^\alpha h(\sigma)\big)\Vert_{L^2} \\
  & \quad+C\tau^s\sum_{j> j_0}2^{j(\frac{N}{2}+1-\alpha)} \Vert S_{j-1} u\cdot\nabla\dot{\Delta}_j \mathcal{P} u - \mathcal{P} \dot{\Delta}_j (T_u\cdot \nabla u)\Vert_{L^2}\\
  &\quad+C\tau^s\sum_{j> j_0}2^{j(\frac{N}{2}+1-\alpha)} \Vert \mathcal{P} \dot{\Delta}_j (T_{\nabla u} u)+ \mathcal{P} \dot{\Delta}_j R(u, \nabla u)\Vert_{L^2}\\
  &\le C\Vert u\Vert_{\dot{B}^{\frac{N}{2}+1}_{2,1}}\tau^s\Vert h(\sigma)\Vert_{\dot{B}^{\frac{N}{2}}_{2,1}} + C\Vert u\Vert_{\dot{B}^{\frac{N}{2}+1}_{2,1}}\tau^s\Vert u^h\Vert_{\dot{B}^{\frac{N}{2}+1-\alpha}_{2,1}}
  \le C\Vert u(\tau)\Vert_{\dot{B}^{\frac{N}{2}+1}_{2,1}} Z_{s,\bar{s}}(\tau).
\end{align*}
Hence collecting the above estimates yields \eqref{eq.Decay.X.h.result}, as desired.
\end{proof}

\subsection{Proof of Proposition \ref{th.decay.intro}}
Firstly, we prove  a sketchy version of the decay estimate:
\begin{align}\label{eq:Zt-bdd}
  Z_{s,\bar{s}}(t) = Z_{s,\bar{s}}^\ell(t)+Z_s^h(t) \le C,
\end{align}
where $(Z_{s,\bar{s}}^\ell, Z_s^h)$ is given by \eqref{eq:Zl}--\eqref{eq:Zh}, and $(s,\bar{s})$ satisfies that
\begin{equation}\label{cond:s-bars}
\begin{cases}
  s\in [0,1-\frac{1}{\alpha}),\quad & \textrm{if}\;\; \bar{s}= \frac{N}{2}+1-\alpha, \\
  s\geq 0, \;\bar{s} + s\alpha \leq \frac{N}{2}+1-\alpha,\quad & \textrm{if}\;\; \bar{s} = -s_0 \in (-\frac{N}{2},\frac{N}{2}+1-\alpha).
\end{cases}
\end{equation}

We first consider the proof of \eqref{eq:Zt-bdd} with an additional condition $0\leq s<1$.
Through combining \eqref{eq.Decay.X.l.result} with \eqref{eq.Decay.X.h.result}, we have
\begin{equation*}
\begin{aligned}
  Z_{s,\bar{s}}(t)\le C + C\int_0^t \Big(\Phi(t,\tau) + \Vert u(\tau)\Vert_{\dot{B}^{\frac{N}{2}+1}_{2,1}}\Big) Z_{s,\bar{s}}(\tau) \dd \tau,
\end{aligned}
\end{equation*}
where $C$ depends on $\Vert \sigma_0\Vert_{ \widetilde{B}^{\bar{s},\frac{N}{2}}}$ and
$\Vert u_0 \Vert_{\widetilde{B}^{\bar{s},\frac{N}{2}+1-\alpha}}$ (noting that if $\bar{s}=-s_0$, the additional assumption
$(\sigma_0^\ell,u_0^\ell) \in \dot{B}^{-s_0}_{2,\infty}$ is needed).
Gronwall's inequality guarantees that
\begin{equation}\label{eq.Decay.X}
\begin{aligned}
  Z_{s,\bar{s}}(t) \le C \exp \Big\{C\int_0^t \Big(\Phi(t,\tau)+\Vert u(\tau)\Vert_{\dot{B}^{\frac{N}{2}+1}_{2,1}}\Big)\dd \tau \Big\}.
\end{aligned}
\end{equation}
Recalling the definition of $\Phi$ in \eqref{eq.Phi.decay} and using H\"older's inequality, we infer that for every $\frac{1}{1-s}<r<\infty$ and $\frac{1}{r}+\frac{1}{r'}=1$,
\begin{align}\label{es:Phi1}
  & \quad \int_0^t \Phi(t,\tau) \dd \tau \nonumber \\
  & \le C\bigg(\sum_{j\le j_0}\sum_{j'\le j+4}2^{-(j-j')\frac{\alpha}{r'}} 2^{j'(\frac{N}{2} +1 -\frac{\alpha}{r'})}\Vert \dot\Delta_{j'} u\Vert_{L^{r}_t(L^2)} \nonumber \\
  & \quad  + \sum_{j\le j_0}\sum_{j'>j-4}2^{(j-j')(\frac{N}{2}+\bar{s} + \delta s \alpha -\frac{\alpha}{r'})}
  2^{j'(\frac{N}{2}+1-\frac{\alpha}{r'})} \Vert \dot\Delta_{j'} u \Vert_{L^r_t (L^2)}\bigg)
  \Big(\sup_{j\in\mathbb{Z}}\int_0^t \psi^{r'}_j (t,\tau)2^{j\alpha}\dd \tau\Big)^{\frac{1}{r'}} \nonumber \\
  & =: C (I_1 + I_2 ) \Big( \sup_{j\in \mathbb{Z}} \int_0^t \psi^{r'}_j (t,\tau)2^{j\alpha}\dd \tau\Big)^{\frac{1}{r'}}.
\end{align}

In order to treat the integral term on the right-hand of \eqref{es:Phi1}, we recall the following elementary result
(it can be easily deduced by respectively considering the contribution from $[0,t/2]$ and $[t/2,t]$,
e.g. see \cite[Lemma 3.4]{xue2017differentiability}).
\begin{lemma}\label{lemma.gamma.split}
Let $\alpha\in (0,2)$, $c>0$ and $j\in \mathbb{Z}$. Then for every $0\le s<1$ we have
\begin{align*}
  \int_0^t e^{-c2^{j\alpha}(t-\tau)}t^s\tau^{-s}\dd \tau\le C2^{-j\alpha}.
\end{align*}
\end{lemma}

Thanks to Lemma \ref{lemma.gamma.split}, we have that for every $r'\in (1,\frac{1}{s})$,
\begin{align}
  \sup_{j\in \mathbb{Z}} \int_0^t \psi^{r'}_j (t,\tau)2^{j\alpha}\dd \tau = \sup_{j\in\mathbb{Z}} 2^{j\alpha} \int_0^t e^{- r' \bar{\mu} 2^{j\alpha}(t-\tau)} t^{s r'} \tau^{-s r'} \dd \tau \le C.
\end{align}
On the other hand, if $\bar{s}=\frac{N}{2} + 1-\alpha$ and $s\in (0,1- \frac{1}{\alpha})$,
we have $\delta =0$ (recalling $\delta$ is given by \eqref{def:del}) and $\frac{N}{2} + \bar{s} + \delta s\alpha - s\alpha > N+ 2(1-\alpha)>0 $;
if $\bar{s}=-s_0\in (-\frac{N}{2}, \frac{N}{2}+1-\alpha)$ and $s \in (0, 1)$ satisfy
$\bar{s} + s\alpha \leq \frac{N}{2} + 1-\alpha$,
it is evident that $\delta =1$ and $\frac{N}{2} + \bar{s} + \delta s\alpha - s\alpha= \frac{N}{2} + \bar{s} >0$.
Moreover, due to $\frac{1}{r'}\in(s,1)$, we can let $\frac{1}{r'}$ be sufficiently close to $s$ so that $\frac{N}{2} + \bar{s} + \delta s\alpha - \frac{\alpha}{r'} >0$.
Consequently, the discrete Young's inequality implies that
\begin{equation}\label{eq.I1I2}
\begin{aligned}
  I_1 + I_2 \le C \Vert u\Vert_{\widetilde{L}^r_t (\dot{B}^{\frac{N}{2}+1-\alpha+\frac{\alpha}{r}}_{2,1})} \leq C \Vert u\Vert_{\widetilde{L}^\infty_t (\dot{B}^{\frac{N}{2}+1-\alpha}_{2,1})} + C \Vert u\Vert_{\widetilde{L}^1_t (\dot{B}^{\frac{N}{2}+1}_{2,1})}.
\end{aligned}
\end{equation}
Hence gathering the above estimates with \eqref{eq:sig-u-bdd2} we deduce that for every $(s,\bar{s})$ satisfying \eqref{cond:s-bars} and $s\in (0,1)$,
\begin{equation*}
\begin{aligned}
  Z_{s,\bar{s}}(t)
  \le C \exp\Big\{ C \Vert u\Vert_{\widetilde{L}^\infty_t (\dot{B}^{\frac{N}{2}+1-\alpha}_{2,1})}
  + C \Vert u\Vert_{\widetilde{L}^1_t (\dot{B}^{\frac{N}{2}+1}_{2,1})} \Big\} \leq C .
\end{aligned}
\end{equation*}

Then we prove \eqref{eq:Zt-bdd} for general $s\in [n,n+1)$ with $n\in \mathbb{Z}^{+}$. 
We only need to treat the case $\bar{s}=-s_0$ and $\bar{s} + s\alpha \leq \frac{N}{2} +1 -\alpha$.
Combining \eqref{eq.Decay.X.l.result.s} with \eqref{eq.Decay.X.h.result} leads to
\begin{equation*}
\begin{aligned}
  Z_{s,\bar{s}}(t)\le C + C Z_{s-1,\bar{s}}(t) + C\int_0^t \Big(\Psi(t,\tau) +
  \Vert u(\tau)\Vert_{\dot{B}^{\frac{N}{2}+1}_{2,1}}\Big) Z_{s,\bar{s}}(\tau) \dd \tau.
\end{aligned}
\end{equation*}
Gronwall's inequality gives that
\begin{equation}
\begin{aligned}
  Z_{s,\bar{s}}(t)\le C\big(1+Z_{s-1,\bar{s}}(t)\big)
  \exp \Big\{C\int_0^t \Big(\Psi(t,\tau)+\Vert u(\tau)\Vert_{\dot{B}^{\frac{N}{2}+1}_{2,1}}\Big)\dd \tau \Big\}.
\end{aligned}
\end{equation}
%
Recalling that $\Psi(t,\tau)$ is given by \eqref{eq.Psi.decay}, and by using H\"older's inequality, Lemma \ref{lemma.gamma.split} and \eqref{eq.I1I2}, we see that for every $r'\in(1,+\infty)$,
\begin{align}
  &\quad\int_0^t \Psi(t,\tau)\dd \tau \nonumber \\
  & \le C\Big(\sup_{j\in\mathbb{Z}}\int_0^t e^{-\bar{\mu}r' 2^{j\alpha}(t-\tau)}2^{j\alpha}\dd \tau\Big)^{\frac{1}{r'}}\bigg(\sum_{j\le j_0}\sum_{j'\le j+4}2^{(j'-j)\frac{\alpha}{r'}} 2^{j'(\frac{N}{2} +1 -\frac{\alpha}{r'})}\Vert \dot\Delta_{j'} u\Vert_{L^{r}_t(L^2)} \nonumber \\
  & \quad  + \sum_{j\le j_0}\sum_{j'>j-4}2^{(j-j')(\frac{N}{2}+\bar{s} +  s \alpha -\frac{\alpha}{r'})}
  2^{j'(\frac{N}{2}+1-\frac{\alpha}{r'})} \Vert \dot\Delta_{j'} u \Vert_{L^r_t (L^2)}\bigg) \nonumber
   \\
  & \le C\Big(\Vert u\Vert_{\widetilde{L}^\infty_t (\dot{B}^{\frac{N}{2}+1-\alpha}_{2,1})} +\Vert u\Vert_{\widetilde{L}^1_t (\dot{B}^{\frac{N}{2}+1}_{2,1})}\Big),
\end{align}
where in the last line we have used the fact $\frac{N}{2} +\bar{s} + s\alpha -\alpha >0$.
Thus the above two inequalities guarantee that for every $\bar{s}=-s_0$, $s\in [n,n+1)$ and $\bar{s} + s\alpha \leq \frac{N}{2} + 1-\alpha$,
\begin{align}
  Z_{s,\bar{s}}(t)\le C + C Z_{s-1,\bar{s}}(t).
\end{align}
Hence, after an iteration of finite times and using the above $0\leq s <1$ result,
we conclude the proof of \eqref{eq:Zt-bdd} for every $(s,\bar{s})$ satisfying \eqref{cond:s-bars}.
\vskip1mm

Next, we prove inequalities \eqref{eq.Decay.X.t+1} and \eqref{eq.decay} from \eqref{eq:Zt-bdd}.
When $t\le 1$, noticing that $(t+1)^s \le 2^s$ and $2^{js\alpha}\le 2^{j_0s\alpha}$ for every $j\le j_0$,  we thus have
\begin{equation}\label{eq.Decay.X.t+1.le}
\begin{aligned}
  (t+1)^s\Big(\Vert &(\sigma,u)^\ell(t)\Vert_{\dot{B}^{\bar{s}+s\alpha}_{2,1}}+\Vert \sigma^h(t)\Vert_{ \dot{B}^{\frac{N}{2}}_{2,1}}
  +\Vert u^h(t)\Vert_{\dot{B}^{\frac{N}{2}+1-\alpha}_{2,1}}\Big) \\
  &\le C(\Vert (\sigma,u)^\ell(t)\Vert_{\dot{B}^{\bar{s}}_{2,1}}+\Vert \sigma^h(t)\Vert_{ \dot{B}^{\frac{N}{2}}_{2,1}} + \Vert u^h(t)\Vert_{\dot{B}^{\frac{N}{2}+1-\alpha}_{2,1}}) \leq C.
\end{aligned}
\end{equation}
The last inequality of \eqref{eq.Decay.X.t+1.le} is provided by \eqref{eq:Zt-bdd} with $s=0$.
When $t> 1$, it is clear that $(t+1)^st^{-s}\le 2^s$, and consequently,
\begin{equation}\label{eq.Decay.X.t+1.ge}
  (t+1)^s\Big(\Vert (\sigma,u)^\ell(t)\Vert_{\dot{B}^{\bar{s}+s\alpha}_{2,1}}+\Vert \sigma^h(t)\Vert_{ \dot{B}^{\frac{N}{2}}_{2,1}} +\Vert u^h(t)\Vert_{\dot{B}^{\frac{N}{2}+1-\alpha}_{2,1}}\Big)
  \le 2^s Z_{s,\bar{s}}(t)\le C.
\end{equation}
$\bullet$ If $(\sigma_0, u_0)\in \widetilde{B}^{\frac{N}{2}+1-\alpha,\frac{N}{2}}\times \dot B^{\frac{N}{2} +1-\alpha} $,
we take $\bar{s}= \frac{N}{2} +1-\alpha$ and $s\in [0, 1-\frac{1}{\alpha})$,
thus the desired inequality \eqref{eq.Decay.X.t+1} follows from \eqref{eq.Decay.X.t+1.le} and \eqref{eq.Decay.X.t+1.ge}.

\noindent
$\bullet$ If the additional condition $(\sigma_0^\ell,u_0^\ell) \in \dot{B}^{-s_0}_{2,\infty}$ with $s_0\in (\alpha-\frac{N}{2}-1,\frac{N}{2}]$ is assumed,
we have $ (\sigma_0,u_0)\in \widetilde{B}^{\bar{s},\frac{N}{2}}\times \widetilde{B}^{\bar{s},\frac{N}{2}+1-\alpha}$ with $\bar{s}= -s_0$,
and by letting $s_1=\bar{s}+s\alpha$, we see that $s_1\in [-s_0, \frac{N}{2} +1-\alpha ]$,
and \eqref{eq.Decay.X.t+1.le}-\eqref{eq.Decay.X.t+1.ge} yield the decay estimate \eqref{eq.decay}.
\vskip1mm

Besides, we prove the asymptotic behavior \eqref{eq.asymptotic}.
For any $\varepsilon >0$, in view of \eqref{eq:sig-u-bdd2}, there exists an integer $N_0>\vert j_0\vert$ (recalling $j_0$ is given by \eqref{eq:j0}) such that
\begin{align*}
  \sup_{t\ge 0} \sum_{j\le -N_0} 2^{j(\frac{N}{2}+1-\alpha)} \big(\Vert \dot\Delta_j \sigma(t) \Vert_{L^2} + \Vert \dot\Delta_j u(t)\Vert_{L^2} \big) \le \frac{\varepsilon}{3}.
\end{align*}
In addition, according to the decay estimate \eqref{eq.Decay.X.t+1}, we infer that for $0<s <1-\frac{1}{\alpha}$,
\begin{equation*}
\begin{aligned}
  & \quad \sum_{-N_0<j\le j_0 }2^{j(\frac{N}{2}+1-\alpha)} (\Vert \dot\Delta_j \sigma(t)\Vert_{L^2}+ \Vert \dot \Delta_j u(t)\Vert_{L^2}) \\
  & \le 2^{N_0s\alpha} \sum_{-N_0<j\le j_0 } 2^{j(\frac{N}{2}+1-\alpha+s\alpha)} (\Vert \dot\Delta_j\sigma(t)\Vert_{L^2} + \Vert \dot\Delta_j u(t)\Vert_{L^2}) \le C(1+t)^{-s},
\end{aligned}
\end{equation*}
and
\begin{align*}
  \sum_{j> j_0} \big(2^{j\frac{N}{2}} \Vert \dot\Delta_j\sigma(t)\Vert_{L^2} + 2^{j(\frac{N}{2}+1-\alpha)} \Vert \dot\Delta_j u(t)\Vert_{L^2}\big)\le C(1+t)^{-s}.
\end{align*}
Then there exists a positive real number $T_0$ such that for every $t\geq T_0$,
\begin{equation*}
  \sum_{j>-N_0} 2^{j(\frac{N}{2}+1-\alpha)} \big( \Vert \dot\Delta_j\sigma(t) \Vert_{L^2} + \Vert \dot\Delta_j u(t)\Vert_{L^2}) \le \frac{2\varepsilon}{3}.
\end{equation*}
Hence for any $\varepsilon >0$, we have that for every $t\geq T_0$,
\begin{equation*}
\begin{aligned}
  \Vert \sigma(t)\Vert_{\widetilde{B}^{\frac{N}{2}+1-\alpha,\frac{N}{2}}}+\Vert u(t)\Vert_{\dot{B}^{\frac{N}{2}+1-\alpha}_{2,1}} \le \varepsilon.
\end{aligned}
\end{equation*}
In other words, the inequality \eqref{eq.asymptotic} holds. We thus finish the proof of Proposition \ref{th.decay.intro}.

\bigskip
\subsection*{Acknowledgements}
Q. Miao was partially supported by NNSF of China (Nos. 12001041, 11871088).
C. Tan was partially supported by NSF grants DMS-1853001 and DMS-2108264.
L. Xue was partially supported by National Key Research and Development Program of China (No. 2020YFA0712900) and NNSF of China (No. 11771043).
\bigskip


\bibliography{EAS-pressure}
\bibliographystyle{plain}

\end{document}